\documentclass[11pt,reqno]{amsart} 

\usepackage{amssymb,tensor,enumitem,verbatim,theoremref,mathtools,stmaryrd}
\usepackage[bookmarksdepth=3]{hyperref}
\usepackage[margin=1in]{geometry}
\usepackage[all,cmtip]{xy}
\usepackage[sort,noadjust]{cite}
\usepackage{nicematrix}

\newcommand{\sm}[1]{[ \begin{smallmatrix} #1 \end{smallmatrix} ]}

\newcommand{\subgrp}[1]{\langle #1 \rangle}

\newcommand{\Set}[1]{\left\{ #1 \right\}}
\newcommand{\set}[1]{\left\{ #1 \right\}}
\newcommand{\smallset}[1]{\{ #1 \}}

\newcommand{\abs}[1]{| #1 |}

\newcommand{\wt}[1]{\widetilde{ #1}}
\newcommand{\ol}[1]{\overline{#1}}

\newcommand{\dbl}[1]{\llbracket #1 \rrbracket}

\DeclareMathOperator{\ad}{ad}

\DeclareMathOperator{\Aut}{Aut}
\DeclareMathOperator{\cont}{res}

\DeclareMathOperator{\End}{End}
\DeclareMathOperator{\Grp}{Grp}
\DeclareMathOperator{\Hom}{Hom}
\DeclareMathOperator{\id}{id}
\DeclareMathOperator{\im}{im}

\DeclareMathOperator{\Inf}{Inf}
\DeclareMathOperator{\Irr}{Irr}
\DeclareMathOperator{\Lie}{Lie}
\DeclareMathOperator{\otr}{otr}
\DeclareMathOperator{\res}{res}

\DeclareMathOperator{\Res}{Res}
\DeclareMathOperator{\rk}{rk}

\DeclareMathOperator{\Span}{span}
\DeclareMathOperator{\tr}{tr}

\newcommand{\zero}{\ol{0}}
\newcommand{\one}{\ol{1}}

\newcommand{\ve}{\varepsilon}
\newcommand{\velammu}{\ve_{(\lambda,\mu)}}
\newcommand{\velamlam}{\ve_{(\lambda,\lambda)}}

\newcommand{\C}{\mathbb{C}}
\newcommand{\N}{\mathbb{N}}

\newcommand{\T}{\mathbb{T}}
\newcommand{\Z}{\mathbb{Z}}

\newcommand{\CA}{\C A}
\newcommand{\CB}{\C \calB}
\newcommand{\CD}{\C \calD}
\newcommand{\CG}{\C G}
\newcommand{\CH}{\C H}

\newcommand{\CS}{\C \fS}

\newcommand{\calB}{\mathcal{B}}
\newcommand{\calC}{\mathcal{C}}
\newcommand{\calD}{\mathcal{D}}

\newcommand{\calP}{\mathcal{P}}
\newcommand{\calT}{\mathcal{T}}

\newcommand{\calX}{\mathcal{X}}

\newcommand{\calBP}{\mathcal{BP}}

\newcommand{\BPn}{\calBP(n)}
\newcommand{\BPsetn}{\calBP\{n\}}
\newcommand{\BPboxn}{\calBP[n]}
\newcommand{\BPdbln}{\calBP\dbl{n}}

\newcommand{\Eboxn}{E[n]}
\newcommand{\Fboxn}{F[n]}
\newcommand{\Edbln}{E\dbl{n}}
\newcommand{\Fdbln}{F\dbl{n}}

\newcommand{\fA}{\mathfrak{A}}

\newcommand{\fb}{\mathfrak{b}}
\newcommand{\fD}{\mathfrak{D}}
\newcommand{\fd}{\mathfrak{d}}

\newcommand{\g}{\mathfrak{g}}
\newcommand{\gl}{\mathfrak{gl}}
\newcommand{\fh}{\mathfrak{h}}
\newcommand{\fs}{\mathfrak{s}}
\newcommand{\fS}{\mathfrak{S}}
\newcommand{\fsl}{\mathfrak{sl}}
\newcommand{\fso}{\mathfrak{so}}
\newcommand{\fsp}{\mathfrak{sp}}
\newcommand{\fq}{\mathfrak{q}}
\newcommand{\sq}{\mathfrak{sq}}

\newcommand{\glmn}{\gl(m|n)}

\newcommand{\Bzero}{\calB_{\zero}}
\newcommand{\Dzero}{\calD_{\zero}}
\newcommand{\gone}{\g_{\one}}
\newcommand{\gzero}{\g_{\zero}}
\newcommand{\Gone}{G_{\one}}
\newcommand{\Gzero}{G_{\zero}}
\newcommand{\Hzero}{H_{\zero}}
\newcommand{\hzero}{\fh_{\zero}}

\newcommand{\Vone}{V_{\one}}
\newcommand{\Vzero}{V_{\zero}}
\newcommand{\Wone}{W_{\one}}
\newcommand{\Wzero}{W_{\zero}}

\newcommand{\fbone}{\fb_{\one}}
\newcommand{\fbzero}{\fb_{\zero}}

\newcommand{\fdzero}{\fd_{\zero}}
\newcommand{\sone}{\fs_{\one}}
\newcommand{\szero}{\fs_{\zero}}

\newcommand{\CGzero}{\CG_{\zero}}
\newcommand{\CHzero}{\CH_{\zero}}
\newcommand{\CBzero}{\CB_{\zero}}
\newcommand{\CDzero}{\CD_{\zero}}

\numberwithin{equation}{subsection}
\counterwithin{table}{section}

\newtheorem{theorem}{Theorem}[subsection]
\newtheorem{proposition}[theorem]{Proposition}

\newtheorem{corollary}[theorem]{Corollary}
\newtheorem{lemma}[theorem]{Lemma}

\theoremstyle{definition}

\newtheorem{definition}[theorem]{Definition}

\newtheorem{example}[theorem]{Example}
\newtheorem{remark}[theorem]{Remark}

\newtheorem{convention-construction}[theorem]{Convention/Construction}
\newtheorem{remark-convention}[theorem]{Remark/Convention}

\title[Lie superalgebras generated by reflections in classical Weyl groups]{Lie superalgebras generated by reflections in Weyl groups of classical type}

\author{Christopher M.\ Drupieski}
\address{Department of Mathematical Sciences,
		DePaul University,
		Chicago, IL 60614, USA}
\email{c.drupieski@depaul.edu}

\thanks{CMD was supported by a Summer Research Grant from the DePaul University College of Science and Health.}

\author{Jonathan R.\  Kujawa}
\address{Department of Mathematics \\
		Oregon State University \\
		Corvallis, OR 97331, USA}
\email{kujawaj@oregonstate.edu}

\thanks{JRK was supported in part by Simons Collaboration Grant for Mathematicians No.\ 525043 and No.\ 963912.}

\date{\today}


\allowdisplaybreaks

\begin{document}

\begin{abstract}
We consider the finite Weyl groups of classical type --- $W(A_r)$ for $r \geq 1$, $W(B_r) = W(C_r)$ for $r \geq 2$, and $W(D_r)$ for $r \geq 4$ --- as supergroups in which the reflections are of odd superdegree. Viewing the corresponding complex group algebras as Lie super\-algebras via the graded commutator bracket, we determine the structure of the Lie sub-super\-algebras generated by the sets of reflections. In each case, this Lie superalgebra is equal to the full derived subalgebra of the group algebra plus the span of the class sums of the reflections.
\end{abstract}



\maketitle

\section{Introduction}

\subsection{Overview}
Let $G$ be a finite reflection group. By declaring the reflections in $G$ to be of odd super\-degree, one endows $G$ with the structure of a super\-group, i.e., a group equipped with a multiplicative grading by $\Z_{2} = \Z /2\Z$. In turn, the complex group algebra $\CG$ is endowed with the structure of an associative super\-algebra, and then the graded commutator bracket
	\[
	[x,y] = xy - (-1)^{\ol{x} \cdot \ol{y}}yx
	\]
defines a Lie super\-algebra structure on $\CG$, which we denote $\Lie(\CG)$. It is now natural to ask for an explicit description of the Lie superalgebra structure on $\Lie(\CG)$, and for the structure of the (in general, proper) Lie sub-superalgebra $\g$ generated by the set of reflections in $G$.

The ungraded (non-super) analogues of the preceding questions were considered by Marin \cite{Marin:2007} for the case $G = \fS_n$ of the symmetric group on $n$ letters. In that paper and its sequels, Marin and others used the structure of the Lie algebra generated by the set of transpositions to obtain new results on the representation theory of braid groups and Iwahori--Hecke algebras in type A. More recently, the authors of this paper determined the structure of the Lie sub-superalgebra of $\Lie(\CS_n)$ generated by the set of transpositions \cite{DK:2025}. When the authors have given talks on that work, the first question from the audience has invariably been, ``What about Weyl groups of other types?'' This paper and its companion \cite{DK:2025-preprint} answer that question for Weyl groups of types BC and D, in the super and non-super settings, respectively. In particular, the current paper provides a new, vastly streamlined, more conceptual approach to the Lie superalgebra calculations carried out in \cite{DK:2025}: while some modest case-by-case verifications are still required for each Lie type, the overall approach is uniform across types, and many tedious calculations are now avoided.

\subsection{Main results}

Let $G$ be a finite reflection group as above. The simple $\CG$-super\-modules exist in two flavors: a simple $\CG$-super\-module $W$ is of type M if $W$ is simple as an ungraded $\CG$-module, and is of type Q if $W$ is reducible as an ungraded $\CG$-module. Let $\Irr_M(G)$ and $\Irr_Q(G)$ be complete sets (up to homogeneous isomorphism) of the simple $\CG$-supermodules of types M and Q, respectively. Then the graded Artin--Wedderburn isomorphism (see Corollary \ref{cor:super-artin-wedderburn-CG}) restricts to a Lie superalgebra isomorphism
	\begin{equation} \label{E:IntroEquation}
	\Lie(\CG) \cong 
	\Big[ \bigoplus_{W \in \Irr_Q(G)} \fq(W) \Big]
	\oplus 
	\Big[ \bigoplus_{W \in \Irr_M(G)} \gl(W) \Big];
	\end{equation}
see Section \ref{subsec:lie-superalgebras-M-Q} for an explanation of the notation $\fq(W)$ and $\gl(W)$. The problem of describing the Lie superalgebra $\Lie(\CG)$ thus reduces to the problem of describing the simple $\CG$-supermodules of types M and Q. If $G$ is a Weyl group of classical type, the simple $\CG$-supermodules can all be described in terms of (bi-)tableaux combinatorics. This is summarized for type BC in Section \ref{subsec:simple-supermodules-B}, for type A in Section~\ref{sec:type-A}, and for type D in Section \ref{subsec:simple-supermodules-D}.

The problem of describing the Lie superalgebra $\g \subseteq \Lie(\CG)$ generated by the set of reflections in $G$ is rather more subtle, though for Weyl groups of classical type, the overall description is uniform. Let $\fD(\CG) = \fD(\Lie(\CG))$ denote the derived Lie sub-superalgebra of $\Lie(\CG)$. Say that a reflection $s$ in $G$ is \emph{long} (resp.\ \emph{short}) according to the length of the corresponding root (i.e., the root, unique up to the scalar $\pm 1$, determining the hyperplane through which $s$ is a reflection). If $G$ is a Weyl group of simply-laced type (ADE), then all reflections are long. Let $C_\ell$ and $C_s$ be the sets of all long and short reflections in $G$, respectively. Then $C_\ell$ is a conjugacy class in $G$, as is $C_s$ if $C_s$ is nonempty. Let $\calC_\ell$ and $\calC_s$ be the corresponding class sums in $\CG$, where by definition $\calC_s$ is taken to be zero if $C_s$ is empty. We prove in Theorems \ref{theorem:main-theorem-B}, \ref{theorem:main-theorem-Sn}, and \ref{theorem:main-theorem-D} that 
	\begin{equation} \label{eq:main-equality}
	\g = \fD(\CG) + \C \cdot \calC_\ell + \C \cdot \calC_s.
	\end{equation}
Coupled with \eqref{E:IntroEquation} and the description of derived subalgebras in Section~\ref{subsec:lie-superalgebras-M-Q}, this gives a complete description of $\g$. In particular, it shows that the Lie superalgebra $\g$ is equal to most of $\Lie(\CG)$. This is in contrast to the non-graded (ordinary Lie algebra) situation investigated in \cite{Marin:2007,DK:2025-preprint}, where the analogous Lie algebras are roughly half the dimension of $\Lie(\CG)$, illustrating the divergence between the graded and non-graded settings.

The overall strategy for proving \eqref{eq:main-equality} is by induction on the rank $n$ of the Weyl group $G = G_n$. An important role in the argument is played by the even subgroup $\Gzero \subseteq G$, defined to be the kernel of the linear character $\kappa : G \to \set{\pm 1}$ that evaluates to $-1$ on all reflections in $G$, and by the under\-lying even Lie algebra $\gzero = \g \cap \CGzero$. Under some mild general hypotheses, the Lie algebra $\gzero$ is seen to be reductive (Proposition \ref{prop:g0-reductive}), and hence its derived subalgebra $\fD(\gzero)$ is a semisimple complex Lie algebra. Further, if $\gzero$ contains an appropriate set of group generators for $\Gzero$, then $\fD(\gzero)$ is able to detect isomorphisms among (most) simple $\CGzero$-modules (Proposition \ref{prop:simples-equivalent-G0-g0}) and among (most) simple $\CG$-supermodules (Corollary \ref{cor:module-iso-equivalent-G-Gzero}). The key step in the induction argument is then carried out through an analysis of the branching rule for the restriction of simple modules from $(G_n)_{\zero}$ to $(G_{n-1})_{\zero}$. Using the dimensional rigidity of the simple complex Lie algebras, one shows for each simple $\CGzero$-module $V$ that the image of $\gzero$ in $\End(V)$ is all of $\fsl(V)$. From this one deduces for each simple $\CG$-supermodule $W$ that, modulo the operators $\calC_\ell$ and $\calC_s$, the image of $\g$ in $\End(W)$ is the full derived algebra of either $\fq(W)$ or $\gl(W)$ (Lemma \ref{lemma:W(gn)}). Finally, one applies the graded Artin--Wedderburn isomorphism and the semisimplicity of $\fD(\gzero)$ to deduce that $\g$ contains all of $\fD(\CG)$.

\subsection{Outline of the paper}

In Section~\ref{sec:preliminaries} we establish necessary results on the representation theory of finite supergroups and their associated Lie superalgebras. In particular, in Sections \ref{subsec:lie-superalgebras-odd-elements} and \ref{subsec:further-lie-superalgebra-structure} we identify an axiomatic framework under which the inductive strategy described above can be carried out. Section \ref{subsec:dihedral} provides a brief digression into the example of the dihedral groups, where the structure of the Lie superalgebra $\g$ can most easily be made explicit.

Sections \ref{sec:Type-B-classical}, \ref{sec:type-A}, and \ref{sec:Type-D-classical} are devoted to verifying the general axiomatic framework for the Weyl groups of classical types BC, A, and D, respectively. In those sections we record the classification of the simple (super)modules, the branching rules for the even subgroups, and other useful facts that may also be of independent interest for the study of the (graded) representation theory of the classical Weyl groups. Readers who are more familiar with the representation theory of the symmetric groups may find it enlightening to read Section \ref{sec:type-A} before delving into the more complicated bookkeeping required for types BC and D in Sections \ref{sec:Type-B-classical} and \ref{sec:Type-D-classical}. (Since type A is not the primary focus of this paper---as we already treated it by different methods in \cite{DK:2025}---we have opted to present it after the more-thoroughly explained type BC, with some proofs in type A referring back to arguments discussed at more length in the context of type BC.)

The last section of the paper, Section \ref{subsec:A5-for-D}, is devoted to verifying one piece of the axiomatic framework---axiom \ref{item:D(gzero)-to-sl(V)}---for the Weyl group $\calD_n$ of type D. This is the most difficult of the classical types to handle, on account of the fact that the branching rule from $(\calD_n)_{\zero}$ to $(\calD_{n-1})_{\zero}$ is not in general multiplicity free. Interestingly, the case in this paper that requires the most special treatment (handled in Section \ref{subsubsec:4b-D}) arises from the same simple module that required special treatment in our companion Lie algebra paper (see \cite[\S5.7]{DK:2025-preprint}), though the methods applied to resolve that case in each paper are quite different.

\subsection{Further directions}

For the sake of completeness, it would be interesting to determine if \eqref{eq:main-equality} can be extended to Weyl groups of all exceptional types. The groups $W(E_6)$, $W(E_7)$, and $W(E_8)$ are currently out of reach by direct calculation (at least to the authors) and some theory will be needed, but the conceptual approach developed in this paper should be helpful in resolving those cases. Likewise, the questions and some of the techniques considered here make sense for the affine Weyl groups. Going forward, it would be especially interesting to follow in the footsteps of Marin and apply the results of this paper and its prequels to the representation theory of the associated Iwahori--Hecke algebras and Artin braid groups. In particular, the graded (super) representation theory of these objects warrants investigation.

\subsection{Acknowledgements}

The first author thanks Javid Validashti for the suggestion to consider the finite subgroups of $PGL(2,\C)$.  The second author thanks Al Heine for early encouragement.

\section{Preliminaries}\label{sec:preliminaries}

We adopt the general notational and terminological conventions for vector super\-spaces and semisimple super\-algebras as laid out in \cite[\S2]{DK:2025}. If $A$ is a super\-algebra, then we may write $\abs{A}$ for the algebra obtained by forgetting its $\Z_2$-grading.

\subsection{Supergroups} \label{subsec:supergroups}

Recall that a finite group $G$ is a \emph{supergroup} if it admits a multiplicative grading $G = \Gzero \cup \Gone$ by $\Z_2$. This is equivalent to the existence of a group homomorphism $\kappa: G \to \set{\pm 1}$ such that $\Gzero = \ker(\kappa)$ and $\Gone = G \backslash \Gzero$. We will assume that $\Gzero \neq G$, and hence that $\Gzero$ is an index-$2$ subgroup of $G$. The $\Z_2$-grading on $G$ extends by linearity to a $\Z_2$-grading on the group algebra $\CG$, making $\CG$ into a superalgebra. Since $\CG$ is semisimple as an ordinary algebra, it is also semisimple as a superalgebra; see \cite[\S2]{Brundan:2002} or \cite[Lemma 2.2.11]{DK:2025}.


Let $\{ S^{\lambda}: \lambda \in \Lambda \}$ be a complete set of simple $\CG$-modules up to isomorphism. By abuse of notation, let $\kappa$ also denote the one-dimensional $\CG$-module afforded by $\kappa$. Then for each $\lambda \in \Lambda$, one has $S^{\lambda} \otimes \kappa \cong S^{\lambda'}$ for a unique $\lambda' \in \Lambda$, and the map $\lambda \mapsto \lambda'$ defines an involution on $\Lambda$. Schur's Lemma implies for each $\lambda \in \Lambda$ that, up to mutual rescaling, there exist unique linear maps $\phi_\kappa^\lambda: S^\lambda \to S^{\lambda'}$ and $\phi_\kappa^{\lambda'}: S^{\lambda'} \to S^\lambda$ (corresponding to $\CG$-module isomorphisms $S^\lambda \otimes \kappa \cong S^{\lambda'}$ and $S^{\lambda'} \otimes \kappa \cong S^{\lambda}$) such that for all $g \in G$, $v \in S^\lambda$, and $w \in S^{\lambda'}$,
	\begin{align}
	\phi_\kappa^\lambda(g \cdot v) &= \kappa(g) g \cdot \phi_\kappa^\lambda(v), &
	\phi_\kappa^{\lambda'}(g \cdot w) &= \kappa(g) g \cdot \phi_\kappa^{\lambda'}(w), \label{eq:associator} \\
	\phi_\kappa^{\lambda'} \circ \phi_\kappa^\lambda &= \id_{S^\lambda}, & 
	\phi_\kappa^{\lambda} \circ \phi_\kappa^{\lambda'} &= \id_{S^{\lambda'}}. \label{eq:inverse-associator}
	\end{align}
In the case $\lambda = \lambda'$, this specifies the map $\phi_\kappa^\lambda : S^\lambda \to S^\lambda$ up to scalar multiplication by $\pm 1$. In the literature, the map $\phi_\kappa^\lambda$ is called an \emph{associator} for $S^\lambda$; cf.\ \cite{Headley:1996,Geetha:2018}.

\begin{example}[The symmetric group] \label{example:symmetric-group}
Let $n \geq 2$. The symmetric group $\fS_n$ is a supergroup via the sign character $\ve'': \fS_n \to \set{\pm 1}$. The even subgroup $(\fS_n)_{\zero}$ of $\fS_n$ is then the alternating group $\fA_n$. The simple $\CS_n$-modules are labeled by integer partitions $\lambda \vdash n$. Given $\lambda \vdash n$, let $\lambda^* \vdash n$ be its conjugate (or transpose) partition, and let $S^\lambda$ be the simple Specht module for $\CS_n$ labeled by $\lambda$. Then the trivial $\CS_n$-module is $S^{[n]}$, and the sign module (i.e., the module afforded by $\ve''$) is $S^{[1^n]}$. By abuse of notation, we may also denote $S^{[1^n]}$ by $\ve''$. Then for each $\lambda \vdash n$, one has $S^\lambda \otimes \ve'' \cong S^{\lambda^*}$. An explicit associator $\phi^\lambda: S^\lambda \to S^\lambda$ for $\lambda = \lambda^*$ is given in \cite[Theorem 2.2]{Headley:1996}; see also \cite[Eqn 9]{Geetha:2018}.
\end{example}

Given a $\CGzero$-module $W$ and an element $t \in G_{\one}$, let ${}^t W = \smallset{ {}^t w : w \in W}$ be the \emph{conjugate} $\CGzero$-module, in which the action of an element $h \in \Gzero$ is defined by $h.{}^tw = {}^t [(tht^{-1}).w]$. Up to isomorphism, the conjugate module does not depend on the choice of element $t \in G_{\one}$. We say that two $\CGzero$-modules $W$ and $W'$ are conjugate if $W' \cong {}^t W$ for some $t \in G_{\one}$.

If $V$ is a $\CG$-module, then we write $\Res_{\Gzero}^G(V)$ for the $\CGzero$-module obtained by restriction. The simple $\CGzero$-modules are now obtained as follows:

\begin{lemma}[{cf.\ \cite[Proposition 5.1]{Fulton:1991}}] \label{lemma:classical-Clifford}
With assumptions and notation as above, each simple $\CGzero$-module arises uniquely (up to isomorphism) in one of the following two ways:
	\begin{enumerate}
	\item \label{item:type-E-simple} If $\lambda \neq \lambda'$, then $\Res_{\Gzero}^G(S^\lambda) \cong \Res_{\Gzero}^G(S^{\lambda'})$ is simple and isomorphic to its conjugate.
	\item If $\lambda = \lambda'$, then $\Res_{\Gzero}^G(S^{\lambda})$ is (canonically, by uniqueness of isotypical components) the direct sum of two simple, conjugate, non-isomorphic $\CGzero$-sub\-modules $S^{\lambda^+}$ and $S^{\lambda^-}$, equal to the $+1$ and $-1$ eigenspaces of (a fixed choice of) the map $\phi_\kappa^{\lambda}: S^\lambda \to S^\lambda$.
	\end{enumerate}
\end{lemma}

Then applying \cite[Lemma 2.3]{Brundan:2002} and \cite[Corollary 2.8]{Brundan:2002} as in \cite[\S3.1]{DK:2025}, one deduces:

\begin{proposition} \label{prop:simple-supermodules}
With assumptions and notation as above, each simple $\CG$-supermodule occurs uniquely, up to homogeneous isomorphism, in one of the following two ways:
	\begin{enumerate}
	\item \label{item:self-associate-G-supermodule} If $\lambda \neq \lambda'$, then there exists a Type Q simple $\CG$-supermodule $W^\lambda = W^{\lambda'}$ such that $W^\lambda = S^\lambda \oplus S^{\lambda'}$ as a $\abs{\CG}$-module, with
		\[
		W_{\zero}^\lambda = \{u + \phi_\kappa^\lambda(u) : u \in S^\lambda \} 
		\qquad \text{and} \qquad 
		W_{\one}^\lambda = \{ u - \phi_\kappa^\lambda(u): u \in S^\lambda \},
		\]
and with odd involution $J^\lambda: W^\lambda \to W^\lambda$ defined by $J^\lambda(u \pm \phi_\kappa^\lambda(u)) = u \mp \phi_\kappa^\lambda(u)$.

	\item If $\lambda = \lambda'$, then there exists a Type $M$ simple $\CG$-supermodule $W^\lambda$ such that $W^\lambda = S^\lambda$ as a $\abs{\CG}$-module, with $W_{\zero}^\lambda = S^{\lambda^+}$ and $W_{\one}^\lambda = S^{\lambda^-}$.
	\end{enumerate}
Each simple $\CG$-supermodule is uniquely determined up to an even isomorphism by its summands as a $\CGzero$-module and by the superdegrees in which those summands are concentrated.
\end{proposition}

\begin{remark} \label{remark:even-odd-isomorphic} \ 
	\begin{enumerate}
	\item Each simple $\CG$-supermodule is even-dimensional.
	

	\item For $\lambda \neq \lambda'$ and $c \neq 0$, the $\CG$-supermodule structure on $S^\lambda \oplus S^{\lambda'}$ defined by the pair $(\phi_\kappa^\lambda,\phi_\kappa^{\lambda'})$ is isomorphic to that defined by $(c \cdot \phi_\kappa^\lambda,\frac{1}{c} \cdot \phi_\kappa^{\lambda'})$ via the map $(u,v) \mapsto (u,c \cdot v)$.

	\item \label{item:W-lam-homogeneous-basis} For $\lambda \neq \lambda'$, if $u_1,\ldots,u_n$ is a basis for $S^\lambda$, then $\{u_i + \phi_\kappa^\lambda(u) : 1 \leq i \leq n \}$ and $\{u_i - \phi_\kappa^\lambda(u) : 1 \leq i \leq n \}$ are bases for $\Wzero^\lambda$ and $\Wone^\lambda$, respectively.

	\item For $\lambda \neq \lambda'$, 
	the map $J^{\lambda}$ can be interpreted as defining an even $\CG$-super\-module isomorphism from $W^{\lambda}$ to the parity shifted module $\Pi(W^\lambda)$.

	\item For $\lambda = \lambda'$, the $\CG$-supermodule $W^{\lambda}$ is odd-isomorphic to $\Pi(W^\lambda)$ but not even-isomorphic because the $\CGzero$-module summands of $S^\lambda$ are not isomorphic. Fixing $W^\lambda$ up to an even isomorphism depends in our notation on a fixed choice for the map $\phi_\kappa^\lambda$. If one replaces $\phi_\kappa^\lambda$ with $-\phi_\kappa^\lambda$, then $W^\lambda$ is replaced by $\Pi(W^\lambda)$.
	\end{enumerate}
\end{remark}

Given a finite supergroup $G$, let $\Irr(G)$ be the set of isomorphism classes of the simple $\CG$-modules, and let $\Irr_s(G)$ be the set of isomorphism classes (up to homogeneous iso\-mor\-phism) of the simple $\CG$-supermodules. Then $\Irr_s(G) = \Irr_M(G) \cup \Irr_Q(G)$, where $\Irr_M(G)$ (resp.\ $\Irr_Q(G)$) is the set of isomorphism classes of the Type M (resp.\ Type Q) simple $\CG$-supermodules. Given $W \in \Irr(G)$ (resp.\ $W \in \Irr_s(G)$), we may abuse notation and write $W: \CG \to \End(W)$ to denote the corresponding module structure map. Thus if $W \in \Irr(G)$ and $g \in G$, then $W(g) \in \End(W)$ denotes the map $u \mapsto g.u$.

Recall that if $V$ is a vector superspace equipped with an odd involution $J$, then
	\begin{equation} \label{eq:Q(V)}
	Q(V) = Q(V,J) = \set{ \theta \in \End(V) : J \circ \theta = \theta \circ J}
	\end{equation}
is a simple subsuperalgebra of $\End(V)$. Now \cite[Theorem 2.2.11]{DK:2025} gives:

\begin{corollary} \label{cor:super-artin-wedderburn-CG}
Let $G$ be a finite supergroup. Then the sum of the simple supermodule structure maps $W: \CG \to \End(W)$ induces a superalgebra isomorphism
	\begin{equation} \label{eq:super-artin-wedderburn-CG}
	\CG \cong 
	\Big[ \bigoplus_{W \in \Irr_Q(G)} Q(W) \Big]
	\oplus 
	\Big[ \bigoplus_{W \in \Irr_M(G)} \End(W) \Big] 
	\end{equation}
In particular, for $W \in \Irr_Q(G)$, the structure map $W: \CG \to \End(W)$ has image in $Q(W)$.
\end{corollary}

\begin{remark} \label{remark:artin-wedderburn-CGzero}
Applying \cite[Lemma 2.2.14]{DK:2025}, one sees that, on restriction to the even subspaces, \eqref{eq:super-artin-wedderburn-CG} restricts to the usual Artin--Wedderburn isomorphism for $\CGzero$:
	\[
	\CGzero \cong \bigoplus_{V \in \Irr(\Gzero)} \End(V).
	\]
More precisely, for $W \in \Irr_M(G)$, the supermodule structure map $W: \CG \to \End(W)$ restricts to the surjective $\CGzero$-module structure map $W: \CGzero \to \End(W)$, while for $W \in \Irr_Q(G)$ with $W = S^\lambda \oplus S^{\lambda'}$ as $\abs{\CG}$-modules and $S^\lambda \cong S^{\lambda'}$ as $\CGzero$-modules, the structure map $W: \CG \to \End(W)$ restricts to map $\CGzero \to \End(S^\lambda) \oplus \End(S^{\lambda'})$, $x \mapsto (S^\lambda(x),S^{\lambda'}(x))$, whose image is a diagonally embedded copy of $\End(S^\lambda) \cong \End(S^{\lambda'})$.
\end{remark}

\subsection{Lie superalgebras arising from simple superalgebras of types M and Q} \label{subsec:lie-superalgebras-M-Q}

Given a vector superspace $V$, write $\gl(V)$ for the set $\End(V)$ considered as a Lie super\-algebra via the super commutator $[x,y] = xy - (-1)^{\ol{x} \cdot \ol{y}}yx$. Fixing a homogeneous basis for $V$, and making the identification $V \cong \C^{m|n} \coloneq \C^m \oplus \Pi(\C^n)$ for some $m,n \in \N$ via this choice of basis, $\gl(V)$ identifies with the matrix Lie superalgebra
	\[
	\gl(m|n) := \Set{ \left[	\begin{array}{c|c}
								A & B \\
								\hline
								C & D
								\end{array} \right] : A \in M_m(\C), B \in M_{m \times n}(\C), C \in M_{n \times m}(\C), D \in M_n(\C)}.
	\]
If $V$ is a vector superspace equipped with an odd involution $J : V \to V$, write $\fq(V)$ for the set $Q(V)$ of \eqref{eq:Q(V)} considered as a Lie super\-algebra via the super commutator. Fixing a basis $\set{v_1,\ldots,v_n}$ for $\Vzero$, and setting $v_i' = J(v_i)$ for $1 \leq i \leq n$, one then has $V \cong \C^{n|n}$, and $\fq(V)$ identifies with
	\[
	\fq(n) := \Set{ \left[	\begin{array}{c|c}
							A & B \\
							\hline
							B & A
							\end{array} \right] : A \in M_n(\C), B \in M_n(\C) }.
	\]

For an arbitrary Lie super\-algebra $\g$, we denote its derived subalgebra $[\g,\g]$ by $\fD(\g)$. Then $\fD(\gl(V))$ is the special linear Lie superalgebra $\fsl(V)$, and
	\[
	\fD(\glmn) = \fsl(m|n) := \Set{ \left[	\begin{array}{c|c}
								A & B \\
								\hline
								C & D
								\end{array} \right] \in \glmn : \tr(A) - \tr(D) = 0}.
	\]
Now let $V$ be a vector superspace equipped with an odd involution $J : V \to V$, let $\theta \in \fq(V)$, and let $\theta = \theta_{\zero} + \theta_{\one}$ be the decomposition of $\theta$ into its homogeneous components. Then $J \circ \theta_{\one} = \theta_{\one} \circ J$ restricts to an even linear map $(J \circ \theta_{\one})|_{\Vzero} : \Vzero \to \Vzero$. Identifying $\Vzero$ and $\Vone$ via $J$, this is equal to the even linear map $(\theta_{\one} \circ J)|_{\Vone} : \Vone \to \Vone$. Now define the \emph{odd trace} of $\theta$, denoted $\otr(\theta)$, by
	\[
	\otr(\theta) = \tr\big( (J \circ \theta_{\one})|_{\Vzero} \big) = \tr\big( (\theta_{\one} \circ J)|_{\Vone} \big),
	\]
i.e., the trace of the linear maps $(J \circ \theta_{\one})|_{\Vzero}$ or $(\theta_{\one} \circ J)|_{\Vone}$, and define $\sq(V) \subseteq \fq(V)$ by
	\[
	\sq(V) = \Set{ \theta \in \fq(V) : \otr(\theta) = 0 }.
	\]
Then one can show that $\fD(\fq(V)) = \sq(V)$, and
	\[
	\fD(\fq(n)) = \sq(n) := \Set{ \left[	\begin{array}{c|c}
								A & B \\
								\hline
								B & A
								\end{array} \right] \in \fq(n) : \tr(B) = 0 }.
	\]
	
\begin{lemma} \label{lem:generated-by-odd} \ 
	\begin{enumerate}
	\item If $m \geq 2$, then $\fsl(m|m)$ is generated as a Lie superalgebra by $\fsl(m|m)_{\one}$.
	\item If $m \geq 3$, then $\sq(m)$ is generated as a Lie superalgebra by $\sq(m)_{\one}$.
	\end{enumerate}
\end{lemma}

\begin{proof}
It is an exercise to show that various matrix units (or sums of two matrix units) spanning the even parts of the Lie super\-algebras can be obtained as Lie brackets of odd elements.
\end{proof}

\subsection{Lie superalgebras generated by odd elements} \label{subsec:lie-superalgebras-odd-elements}

Let $G$ be a finite supergroup, and write $\Lie(\CG)$ for the superalgebra $\CG$ considered as a Lie superalgebra via the super commutator
	\[
	[x,y] = xy - (-1)^{\ol{x} \cdot \ol{y}}yx.
	\]
Then \eqref{eq:super-artin-wedderburn-CG} restricts to a Lie superalgebra isomorphism
	\begin{equation} \label{eq:LieCG}
	\Lie(\CG) \cong 
	\Big[ \bigoplus_{W \in \Irr_Q(G)} \fq(W) \Big]
	\oplus 
	\Big[ \bigoplus_{W \in \Irr_M(G)} \gl(W) \Big].
	\end{equation}
Set $\fD(\CG) = \fD(\Lie(\CG))$. Then making the identification \eqref{eq:LieCG}, one has
	\begin{equation} \label{eq:D(LieCG)}
	\fD(\CG) = 
	\Big[ \bigoplus_{W \in \Irr_Q(G)} \sq(W) \Big] 
	\oplus 
	\Big[ \bigoplus_{W \in \Irr_M(G)} \fsl(W) \Big]
	\end{equation}

Let $S \subseteq \Gone$ be a subset of the odd elements of $G$, and let $\g$ be the Lie subsuperalgebra of $\CG$ generated by the elements of the set $S$. In this section we will prove results under various combinations of the following assumptions:
	\begin{enumerate}[label={(A\arabic*)}]
	\item \label{item:S-union-conj-classes} The set $S$ is a union of conjugacy classes in $G$, say $S = C_1 \cup \cdots \cup C_t$.

	\item \label{item:gzero-generates-CGzero} There exists $S' \subseteq \gzero$ such that $\Gzero = \subgrp{S'}$. In particular, $S'$ generates $\CGzero$ as an algebra.

	\item \label{item:S'-one-class-gen-G0} The set $S'$ consists of a single conjugacy class of order-two elements in $\Gzero$.
	

	\end{enumerate}
	
\begin{lemma} \label{lemma:g-in-D(g)+span}
Suppose \ref{item:S-union-conj-classes} holds. For $1 \leq i \leq t$, let $T_i = \sum_{c \in C_i} c \in \g$. Then
	\begin{equation} \label{eq:g-in-D(g)+span}
	\g \subseteq \fD(\CG) + \Span\{T_1,\ldots,T_t\}.
	\end{equation}
\end{lemma}

\begin{proof}
Since any Lie bracket between elements of $\CG$ is by definition equal to an element of $\fD(\CG)$, it suffices to show that $S \subseteq \fD(\CG) + \Span\{T_1,\ldots,T_t\}$. Let $\tau \in S$ with say $\tau \in C_1$. Set $C = C_1$ and $T = T_1$. For each $V \in \Irr(G)$, the trace of the operator $V(c): V \to V$ is constant across elements $c \in C$, because all elements of $C$ are conjugate. Then $\tr(V(T)) = \abs{C} \cdot \tr(V(\tau))$, and hence $\wt{\tau} \coloneq \tau - \frac{1}{\abs{C}} \cdot T$ is a traceless operator on $V$.

Now let $W \in \Irr_s(G)$. Then $W(\wt{\tau}) \in \gl(W)_{\one} = \fD(\gl(W))_{\one} \subset \fD(\gl(W))$. If $W$ is of Type Q, with say $W = S^\lambda \oplus S^{\lambda'}$ as in Proposition \ref{prop:simple-supermodules}, and if one fixes a homogeneous basis for $W$ as in Remark \ref{remark:even-odd-isomorphic}\eqref{item:W-lam-homogeneous-basis} by fixing a basis for $S^\lambda$, then one sees that the matrix of $W(\wt{\tau})$ is
	\[
	W^\lambda(\wt{\tau}) = 
	\left[ 
	\begin{array}{c|c} 0 & S^\lambda(\wt{\tau}) \\ 
	\hline 
	S^\lambda(\wt{\tau}) & 0 \end{array} 
	\right],
	\]
in which the off-diagonal blocks are the matrix of $S^\lambda(\wt{\tau})$ with respect to the chosen basis for $S^\lambda$. In particular, $S^\lambda(\wt{\tau})$ is traceless by the previous paragraph, so $W^\lambda(\wt{\tau}) \in \sq(W) = \fD(\fq(W))$.

Finally, make the identifications \eqref{eq:LieCG} and \eqref{eq:D(LieCG)}. The observations of the previous paragraph show for each $W \in \Irr_s(G)$ that the projection $W(\wt{\tau})$ is an element of $\fD(\CG)$. Then $\wt{\tau} \in \fD(\CG)$, and hence $\tau = \wt{\tau} + \frac{1}{\abs{C}} \cdot T \in \fD(\CG) + \Span\{T_1,\ldots,T_t\}$.
\end{proof}

\begin{lemma} \label{lemma:lie-algebra-center}
Suppose \ref{item:S-union-conj-classes} and \ref{item:gzero-generates-CGzero} hold. Let $Z(\gzero)$ be the center of the Lie algebra $\gzero$. Then $Z(\gzero) = \gzero \cap Z(\CGzero)$, and the projection map $p: \CGzero \to Z(\CGzero)$, defined by
		\[
		p(z) = \frac{1}{\abs{\Gzero}}  \sum_{g \in \Gzero} g z g^{-1},
		\]
	restricts to a projection map $p: \gzero \to Z(\gzero)$. For this map, one has $p(\fD(\gzero)) \subseteq \fD(\gzero)$.
\end{lemma}

\begin{proof}
Let $z \in \gzero$. The adjoint map $\ad_z: x \mapsto [z,x]$ is a derivation on $\CGzero$. If $z \in Z(\gzero)$, then $\ad_z(s) =0$ for each $s \in S' \subseteq \gzero$. Since $\CGzero$ is generated by the set $S'$ by \ref{item:gzero-generates-CGzero}, this implies that $\ad_z(x) = 0$ for all $x \in \CGzero$, and hence $z \in Z(\CGzero)$. Then $z \in Z(\gzero)$ if and only if $z \in Z(\CGzero)$.

Next, since the set $S$ is a union of conjugacy classes by \ref{item:S-union-conj-classes}, it follows that $\g$ is closed under conjugation by $G$. Conjugation is an even linear map, so $\gzero$ is also closed under conjugation. Then $p$ sends elements of $\gzero$ to elements of $\gzero \cap Z(\CGzero) = Z(\gzero)$. Now since $\gzero$ is closed under conjugation, it follows that $\fD(\gzero) = [\gzero,\gzero]$ is also closed under conjugation, and hence $p(\fD(\gzero)) \subseteq \fD(\gzero)$.
\end{proof}

\begin{remark}
In general, $\fD(\gzero) \subset \fD(\g)_{\zero}$ but $\fD(\gzero) \neq \fD(\g)_{\zero}$.
\end{remark}

\begin{lemma} \label{lemma:G0-g0-submodule-equivalent}
Suppose \ref{item:gzero-generates-CGzero} holds. Let $W$ be a $\CGzero$-module. Then the following are equivalent for a subspace $V \subseteq W$:
	\begin{enumerate}
	\item $V$ is a $\CGzero$-submodule.
	\item $V$ is a submodule for the action of the Lie subalgebra $\gzero \subseteq \CGzero$.
	\end{enumerate}
In particular, if $W$ is a simple $\CGzero$-module, then $W$ is simple as a $\gzero$-module.
\end{lemma}

\begin{proof}
This is immediate from \ref{item:gzero-generates-CGzero}.
\end{proof}

\begin{proposition} \label{prop:g0-reductive}
Suppose \ref{item:S-union-conj-classes} and \ref{item:gzero-generates-CGzero} hold. Then:
	\begin{enumerate}
	\item Each simple $\CG$-supermodule is semisimple as a $\gzero$-module.
	\item The Lie algebra $\gzero$ is reductive. In particular, $\fD(\gzero)$ is semisimple, $\gzero = Z(\gzero) \oplus \fD(\gzero)$, and the projection map $p: \gzero \to Z(\gzero)$ of Lemma \ref{lemma:lie-algebra-center} has $\ker(p) = \fD(\gzero)$.
	\end{enumerate}
\end{proposition}

\begin{proof}
Let $W$ be a simple $\CG$-supermodule. Then by the classification of the simple $\CG$-super\-modules, the subspaces $\Wzero$ and $\Wone$ are simple $\CGzero$-modules, and hence are simple $\gzero$-modules by Lemma \ref{lemma:G0-g0-submodule-equivalent}. Now the sum $\bigoplus_{W \in \Irr_s(G)} W$ is a faithful (by Corollary \ref{cor:super-artin-wedderburn-CG}), finite-dimensional, completely reducible $\gzero$-module. This implies by Proposition~5 of \cite[Chapter 1, \S6, no.\ 4]{Bourbaki:1998} that $\gzero$ is reductive, $\fD(\gzero)$ is semisimple, and $\gzero = Z(\gzero) \oplus \fD(\gzero)$. Finally, the projection map $p$ sends $\fD(\gzero)$ into $\fD(\gzero) \cap Z(\gzero) = \set{0}$, implying that $\ker(p) = \fD(\gzero)$.
\end{proof}

\begin{lemma} \label{lemma:Gzero-simple-implies-D(g0)-simple}
Suppose \ref{item:S-union-conj-classes} and \ref{item:gzero-generates-CGzero} hold. Let $W$ be a simple $\CGzero$-module. Then $W$ is simple as a $\fD(\gzero)$-module.
\end{lemma}

\begin{proof}
Let $V \subseteq W$ be a nonzero $\fD(\gzero)$-submodule of $W$. Since $Z(\gzero) \subseteq Z(\CGzero)$ by Lemma \ref{lemma:lie-algebra-center}, each element $z \in Z(\gzero)$ acts on $W$ as a scalar multiple of the identity, by Schur's Lemma. Then $V$ is closed under the action of both $Z(\gzero)$ and $\fD(\gzero)$, and hence closed under the action of $\gzero$, by Proposition \ref{prop:g0-reductive}. This implies by Lemma \ref{lemma:G0-g0-submodule-equivalent} that $V$ is a nonzero $\CGzero$-submodule of $W$. Then $V = W$, by the simplicity of $W$ as a $\CGzero$-module, so $W$ is simple as a $\fD(\gzero)$-module.
\end{proof}

\begin{proposition} \label{prop:simples-equivalent-G0-g0}
Suppose \ref{item:S-union-conj-classes}, \ref{item:gzero-generates-CGzero}, and \ref{item:S'-one-class-gen-G0} hold. Let $V_1$ and $V_2$ be two simple $\CGzero$-modules of dimensions greater than $1$. Then the following are equivalent:
	\begin{enumerate}
	\item \label{item:Gzero-iso} $V_1$ and $V_2$ are isomorphic as $\CGzero$-modules.
	\item \label{item:gzero-iso} $V_1$ and $V_2$ are isomorphic as modules over the Lie subalgebra $\gzero \subseteq \CGzero$.
	\item \label{item:D(gzero)-iso} $V_1$ and $V_2$ are isomorphic as modules over the Lie subalgebra $\fD(\gzero) \subseteq \CGzero$.
	\end{enumerate}
\end{proposition}

\begin{proof}
Our argument is an adaptation of the proof of \cite[Proposition 2]{Marin:2007}. The fact that \eqref{item:Gzero-iso} implies \eqref{item:gzero-iso}, and that \eqref{item:gzero-iso} implies \eqref{item:D(gzero)-iso}, is evident. We will show that \eqref{item:D(gzero)-iso} implies \eqref{item:Gzero-iso}.

Suppose $\phi: V_1 \to V_2$ is an isomorphism of $\fD(\gzero)$-modules, and let $\rho_1: \CGzero \to \End(V_1)$ and $\rho_2: \CGzero \to \End(V_2)$ be the structure maps for $V_1$ and $V_2$, respectively. Then for all $x \in \fD(\gzero)$, one has $\phi \circ \rho_1(x) = \rho_2(x) \circ \phi$, or equivalently, $\rho_2(x) = \phi \circ \rho_1(x) \circ \phi^{-1}$. Let $s \in S'$, and let $T = p(s)$, where $p: \gzero \to Z(\gzero) \subseteq Z(\CGzero)$ is the projection map of Lemma \ref{lemma:lie-algebra-center}. The elements of the set $S'$ form a single conjugacy class in $\Gzero$ by \ref{item:S'-one-class-gen-G0}, so $T$ is independent of the choice of $s$. Since $T \in Z(\CGzero)$, Schur's Lemma implies that $\rho_1(T) = c_1 \cdot \id_{V_1}$ and $\rho_2(T) = c_2 \cdot \id_{V_2}$ for some scalars $c_1,c_2 \in \C$, which also do not depend on the choice of $s$. One has $p(s-T) = p(s)-p(T) = T - T = 0$, so $s-T \in \fD(\gzero)$ by Proposition \ref{prop:g0-reductive}. Then
	\begin{align*}
	\rho_2(s) - c_2 \cdot \id_{V_2} = \rho_2(s-T) &= \phi \circ \rho_1(s-T) \circ \phi^{-1} \\
	&= \phi \circ \rho_1(s) \circ \phi^{-1} - \phi \circ (c_1 \cdot \id_{V_1}) \circ \phi^{-1} \\
	&= \phi \circ \rho_1(s) \circ \phi^{-1} - c_1 \cdot \id_{V_2},
	\end{align*}
or equivalently,
	\begin{equation} \label{eq:rho2=rho1+omega}
	\rho_2(s) = \phi \circ \rho_1(s) \circ \phi^{-1} + \omega \cdot \id_{V_2},
	\end{equation}
where $\omega = c_2 - c_1$. Squaring both sides of \eqref{eq:rho2=rho1+omega}, and applying the assumption \ref{item:S'-one-class-gen-G0} that $s^2 = 1$ for each $s \in S'$, it follows that 
$-2 \omega \cdot \phi \circ \rho_1(s) \circ \phi^{-1} = \omega^2 \cdot \id_{V_2}$. The scalar $\omega$ does not depend on the choice of $s$, so if $\omega \neq 0$, we would deduce first for all $s \in S'$, and then for all $s \in \Gzero$ by multiplicativity---using \ref{item:gzero-generates-CGzero}, that $\rho_1(s)$ is a nonzero scalar multiple of $\id_{V_1}$. Since $\dim(V_1) \geq 2$ by assumption, this would contradict the simplicity of $V_1$ as a $\CGzero$-module. Then $\omega = 0$, and \eqref{eq:rho2=rho1+omega} implies first for all $s \in S'$, and then for all $s \in \Gzero$ by multiplicativity, that $\phi \circ \rho_1(s) = \rho_2(s) \circ \phi$, and hence $\phi$ is a $\CGzero$-module isomorphism.
\end{proof}

\begin{corollary} \label{cor:module-iso-equivalent-G-Gzero}
Suppose \ref{item:S-union-conj-classes}, \ref{item:gzero-generates-CGzero}, and \ref{item:S'-one-class-gen-G0} hold. Let $V_1$ and $V_2$ be two simple $\CG$-super\-modules of dimensions greater than $2$. Then the following statements (in which `isomorphic' is taken to mean `isomorphic via a homogeneous isomorphism') are equivalent:
	\begin{enumerate}
	\item \label{item:Sn-iso-general} $V_1$ and $V_2$ are isomorphic as $\CG$-supermodules.
	\item \label{item:An-iso-general} $V_1$ and $V_2$ are isomorphic as $\CGzero$-supermodules.
	\item \label{item:gzero-iso-general} $V_1$ and $V_2$ are isomorphic as supermodules over the Lie subalgebra $\gzero \subseteq \CGzero$.
	\item \label{item:D(gzero)-iso-general} $V_1$ and $V_2$ are isomorphic as supermodules over the Lie subalgebra $\fD(\gzero) \subseteq \CGzero$.
	\end{enumerate}
\end{corollary}

\begin{proof}
By Proposition \ref{prop:simple-supermodules}, each simple $\CG$-supermodule is uniquely determined by its summands as a $\CGzero$-module, and by the homogeneous degrees in which those summands are concentrated. Then passing to the homogeneous subspaces of $V_1$ and $V_2$, the result follows by Proposition \ref{prop:simples-equivalent-G0-g0}.
\end{proof}

\subsection{Example: Dihedral groups} \label{subsec:dihedral}

In this section let $G = D_m$ be the dihedral group of order $2m$:
	\begin{align*}
	D_m = \subgrp{r,s : r^m = s^2 = (sr)^2 = 1} = \smallset{ 1, r, r^2, \ldots, r^{m-1}, s, sr, sr^2, \ldots, sr^{m-1}}.
	\end{align*}
The dihedral group is a supergroup in which $\Gzero = \subgrp{r}$ is the cyclic subgroup generated by $r$, and $\Gone = S = \smallset{s,sr,\ldots,sr^{m-1}}$ is the set of reflections in $D_m$. If $m$ is odd then $S$ consists of a single conjugacy class, which we also denote $C_1$, while if $m$ is even then $S$ consists of the two conjugacy classes $C_1 = \smallset{s,sr^2,\ldots,sr^{m-2}}$ and $C_2 = \smallset{sr,sr^3,\ldots,sr^{m-1}}$.

\begin{proposition} \label{prop:D(s(Dm)-LSA}
Let $G$ be the dihedral group $D_m$ of order $2m$ ($m \geq 3$), and let $\g$ be the Lie sub-superalgebra of $\Lie(\CG)$ generated by the set $S$ of reflections in $G$. Then the inclusion \eqref{eq:g-in-D(g)+span} is an equality, i.e., $\g = \fD(\CD_m) + \Span\smallset{T_1,T_2}$, where $T_2$ is taken to be zero if $m$ is odd.
\end{proposition}

\begin{proof}
The argument is by dimension comparison. First suppose $m = 2k$ is even. Then by \cite[\S2.4]{DK:2025}, there are $k-1$ simple $\CG$-supermodules of type M and two simple $\CG$-supermodules of type Q. This implies by \eqref{eq:g-in-D(g)+span} and \eqref{eq:D(LieCG)} that $\dim(\g) \leq [2m-(k-1)-2] + 2 = 3k+1$. On the other hand, $\g$ contains the $m$ odd reflections of the set $S$, as well as the even elements $[s,sr^i] = r^i + r^{m-i}$ for $0 \leq i \leq k$. Then $\dim(\g) \geq m + k+1 = 3k+1$, so the inclusion \eqref{eq:g-in-D(g)+span} must be an equality.

Now suppose $m = 2k+1$ is odd. Then by \cite[\S2.4]{DK:2025}, there are $k$ simple $\CG$-supermodules of type M, and one simple $\CG$-supermodule of type Q. Then $\dim(\g) \leq [(2m-k-1] + 1 = 3k+2$. On the other hand, $\g$ contains the $m$ odd reflections of the set $S$, as well as the even elements $[s,sr^i] = r^i + r^{m-i}$ for $0 \leq i \leq k$. Then $\dim(\g) \geq m + (k+1) = 3k+2$, so again by dimension comparison, the inclusion \eqref{eq:g-in-D(g)+span} must be an equality.
\end{proof}

The proof shows that for $G = D_m$ ($m \geq 3$), one has $\gone = \Span\smallset{s,sr,\ldots,sr^{m-1}}$, and
	\[
	\gzero = \begin{cases}
	\Span\smallset{e,r+r^{m-1},r^2+r^{m-2},\ldots,r^{k-1}+r^{k+1},r^k} & \text{if $m = 2k$ is even,} \\
	\Span\smallset{e,r+r^{m-1},r^2+r^{m-2},\ldots,r^k+r^{k+1}} & \text{if $m = 2k+1$ is odd.} 
	\end{cases}
	\]
In particular, $\gzero = Z(\gzero)$ and $\fD(\gzero) = 0$. Assumptions \ref{item:gzero-generates-CGzero} and \ref{item:S'-one-class-gen-G0} are not satisfied, since for $m \geq 3$ the Lie algebra $\gzero$ does not contain any generators for $\Gzero$. One can check that $\fD(\CG)$ is spanned by $\gzero$ together with all odd elements of the form $[sr^j,r^i] = sr^{j+i} - sr^{j-i}$. (Note that $j+i$ and $j-i$ are of the same parity.) If $m = 2k+1$ is odd, then also $s - sr = [sr^{k+1},r^k] \in \fD(\CG)$. It follows that
	\[
	\fD(\CG)_{\one} = \begin{cases}
	\Span\{ sr^{2i}-s, sr^{2i+1}-sr : 1 \leq i \leq k-1 \} & \text{if $m = 2k$ is even,} \\
	\Span\{ sr-s, sr^2-s, sr^3-s,\ldots,sr^{m-1}-s \} & \text{if $m = 2k+1$ is odd.}
	\end{cases}
	\]





\subsection{Further Lie superalgebra structure} \label{subsec:further-lie-superalgebra-structure}

As in Section \ref{subsec:lie-superalgebras-odd-elements}, let $S \subseteq \Gone$ be a subset of the odd elements of $G$, and let $\g$ be the Lie sub\-super\-algebra of $\Lie(\CG)$ generated by $S$. In addition to assumptions \ref{item:S-union-conj-classes}, \ref{item:gzero-generates-CGzero}, and \ref{item:S'-one-class-gen-G0} from Section \ref{subsec:lie-superalgebras-odd-elements}, we will impose combinations of the assumptions \ref{item:inductive-subgroup}, \ref{item:D(gzero)-to-sl(V)}, and \ref{item:no-G-simple-dim-2} stated below. In practice, if $G$ is a Weyl group of rank $n$, then the subgroup $H$ in assumption \ref{item:inductive-subgroup} will be a Weyl group of the same Coxeter type of rank $n-1$ and for which, by induction, we will already know that \eqref{eq:g-in-D(g)+span} is an equality. Assumption \ref{item:no-G-simple-dim-2} helps to simplify some arguments, and will be satisfied as long as the rank of $G$ is not too small. The one assumption that requires some true case-by-case work to verify is \ref{item:D(gzero)-to-sl(V)}.
	\begin{enumerate}[label={(A\arabic*)}]
	\setcounter{enumi}{3}
	\item \label{item:inductive-subgroup} There exists a subsupergroup $H \subseteq G$ and a subset $S_H \subseteq S$ such that:
		\begin{itemize}
		\item \ref{item:S-union-conj-classes} and \ref{item:gzero-generates-CGzero} hold for $H$ and the Lie subsuperalgebra $\fh \subseteq \Lie(\CH)$ generated by $S_H$;

		\item $\fh_{\zero} = \fD(\CH)_{\zero}$, and hence (as a consequence of Remark \ref{remark:artin-wedderburn-CGzero}), the analogue of \eqref{eq:D(LieCG)} for the group $H$ restricts to an isomorphism
			\begin{equation} \label{eq:D(hzero)}
			\fD(\hzero) \cong \bigoplus_{V \in \Irr(\Hzero)} \fsl(V),
			\end{equation}
induced by the sum of the simple $\CHzero$-module structure maps $V: \CHzero \to \End(V)$;

		\item there are no simple $\CHzero$-modules of dimension $2$; and

		\item if $V$ is a simple $\CGzero$-module with $\dim(V) > 1$, then $\Res^{\Gzero}_{\Hzero}(V)$ contains some simple $\CHzero$-module $W$ with $\dim(W) > 1$.
		\end{itemize}

	\item \label{item:D(gzero)-to-sl(V)} For each $V \in \Irr(\Gzero)$, the module structure map $V: \CGzero \to \End(V)$ restricts to a surjective Lie algebra map $\rho_V: \fD(\gzero) \twoheadrightarrow \fsl(V)$.
	
	\item \label{item:no-G-simple-dim-2} There are no simple $\abs{\CG}$-modules of dimension $2$.
	\end{enumerate}

Let $V$ be a $\CGzero$-module with structure map $\rho: \CGzero \to \End(V)$. There are two evident ways to define an action of the Lie algebra $\fD(\gzero) \subseteq \CGzero$ on the dual space $V^* = \Hom(V,\C)$. The first is the restriction from $\CGzero$ to $\fD(\gzero)$ of the usual $\CGzero$-module structure, defined for $g \in \Gzero$, $\phi \in V^*$, and $v \in V$ by $(g.\phi)(v) = \phi(g^{-1}.v)$. We denote this $\fD(\gzero)$-module structure on $V^*$ by $V^{*,\Grp}$. The second structure, which we denote $V^{*,\Lie}$, is via the contragredient action of a Lie algebra, defined for $x \in \fD(\gzero)$, $\phi \in V^*$, and $v \in V$ by $(x.\phi)(v) = -\phi(x.v)$. Fixing a basis for $V$ and the corresponding dual basis for $V^*$, the structure maps $\rho_{\Grp}$ and $\rho_{\Lie}$ for $V^{*,\Grp}$ and $V^{*,\Lie}$ are related to $\rho$ by $\rho_{\Grp}(x) = \rho(\imath(x))^T$ and $\rho_{\Lie}(x) = - \rho(x)^T$, where $\imath: \CGzero \to \CGzero$ is linear extension of the group inversion map $\sigma \mapsto \sigma^{-1}$, and $u^T$ is the transpose of a matrix $u$.

\begin{lemma} \label{lemma:Gzero-modules-not-dual}
Suppose \ref{item:S-union-conj-classes}--\ref{item:inductive-subgroup} hold. Let $V$ and $W$ be simple $\CGzero$-modules of dimensions greater than $1$. Then $W \not\cong V^{*,\Lie}$.
\end{lemma}

\begin{proof}
First we will show that if $V$ and $W$ are simple $\CHzero$-modules of dimension greater than $1$, then $W \not\cong V^{*,\Lie}$ as $\fD(\hzero)$-modules. There are two cases: $V \cong W$ or $V \not\cong W$. In the first case we may assume that $V = W$. The structure maps $\rho: \fD(\hzero) \to \End(V)$ and $\rho_{\Lie}: \fD(\hzero) \to \End(V^{*,\Lie})$ have the same kernel, and hence both factor, via the identification \eqref{eq:D(hzero)}, through the projection $\fD(\hzero) \twoheadrightarrow \fsl(V)$. Then $V \not\cong V^{*,\Lie}$, because $\dim(V) \geq 3$ by \ref{item:inductive-subgroup}, and the natural module for $\fsl(V)$ is self-dual only if $\dim(V) \leq 2$. In the case $V \not\cong W$, we see from \eqref{eq:D(hzero)} that the image of the structure map $\fD(\hzero) \to \End(V) \oplus \End(W)$ is of dimension $\dim(\fsl(V)) + \dim(\fsl(W)) > \dim(\fsl(V))$. On the other hand, if $W \cong V^{*,\Lie}$, then up to a change of basis for $W$, the structure map is of the form $x \mapsto (\rho(x),-\rho(x)^T)$, and hence its image is of dimension $\dim(\fsl(V))$. Thus $W \not\cong V^{*,\Lie}$.

Now let $V$ and $W$ be simple $\CGzero$-modules of dimension greater than $1$, and suppose $W \cong V^{*,\Lie}$ as $\fD(\gzero)$-modules. Let $V = \bigoplus_{i=1}^r V_i$ and $W = \bigoplus_{j=1}^s W_j$ be decompositions of $V$ and $W$ into simple $\CHzero$-modules. By \ref{item:inductive-subgroup} and Lemma \ref{lemma:Gzero-simple-implies-D(g0)-simple}, these are also decompositions of $V$ and $W$ into simple $\fD(\hzero)$-modules. Then $V^{*,\Lie} = \bigoplus_{i=1}^r (V_i)^{*,\Lie}$ is a decomposition of $V^{*,\Lie}$ into simple $\fD(\hzero)$-modules. Now since $W \cong V^{*,\Lie}$ as $\fD(\gzero)$-modules, it follows for each $j$ that $W_j \cong (V_i)^{*,\Lie}$ as $\fD(\hzero)$-modules for some $i$ (depending on $j$). But by \ref{item:inductive-subgroup}, at least one of the $W_j$ is of dimension greater than $1$, so this isomorphism contradicts the previous paragraph. Thus, $W \not\cong V^{*,\Lie}$.
\end{proof}

\begin{lemma} \label{lemma:D(gzero)}
Suppose \ref{item:S-union-conj-classes}--\ref{item:D(gzero)-to-sl(V)} hold. Then
	\[
	\fD(\gzero) = \fD(\fD(\CG)_{\zero}).
	\]
In particular, the sum of the structure maps of the simple $\CGzero$-modules induces an isomorphism
	\begin{equation} \label{eq:D(D(CG)0)-AW-iso}
	\fD(\gzero) \cong \bigoplus_{V \in \Irr(\Gzero)} \fsl(V).
	\end{equation}
\end{lemma}

\begin{proof}
First, $\gzero \subseteq \fD(\CG)_{\zero}$ by Lemma \ref{lemma:g-in-D(g)+span}, and hence $\fD(\gzero) \subseteq \fD(\fD(\CG)_{\zero})$. Next, it follows from Remark \ref{remark:artin-wedderburn-CGzero} that \eqref{eq:D(LieCG)} restricts to an isomorphism between $\fD(\fD(\CG)_{\zero})$ and the right-hand side of \eqref{eq:D(D(CG)0)-AW-iso}. In particular,
	\[
	\dim(\fD(\fD(\CG)_{\zero})) = \sum_{V \in \Irr(\Gzero)} \dim(\fsl(V)).
	\]
We will show by dimension comparison that $\fD(\gzero) = \fD(\fD(\CG)_{\zero})$.

Recall from Proposition \ref{prop:g0-reductive} that $\fD(\gzero)$ is a semisimple Lie algebra. Given $V \in \Irr(\Gzero)$, let $\rho_V : \fD(\gzero) \to \End(V)$ be the module structure map for the action of $\fD(\gzero)$ on $V$. By \ref{item:D(gzero)-to-sl(V)}, the image of $\rho_V$ is $\fsl(V)$. Let $\g_V = \ker(\rho_V)$, and let $\g^V$ be the orthogonal complement of $\g_V$ with respect to the Killing form on $\fD(\gzero)$. Then $\g_V$ and $\g^V$ are ideals in $\fD(\gzero)$, $\fD(\gzero) = \g_V \oplus \g^V$, and $\rho_V$ induces a Lie algebra isomorphism $\g^V \cong \im(\rho_V) = \fsl(V)$. In particular, if $\dim(V) = 1$, then $\g^V = 0$. We want to show that the (nonzero) simple ideals $\g^V$ for $V \in \Irr(\Gzero)$ are distinct. Equivalently, we want to show that if $\g^V = \g^W \neq 0$, then $V = W$. This will imply that the sum of the $\g^V$ is a direct sum, and will then imply by dimension comparison that $\fD(\fD(\CG)_{\zero}) = \fD(\gzero) = \bigoplus_{V \in \Irr(\Gzero)} \g^V$.

Let $V,W \in \Irr(\Gzero)$ such that $\g^V = \g^W \neq 0$. Since $\fsl(V) \cong \g^V = \g^W \cong \fsl(W)$, then $\dim(V) = \dim(W) > 1$. Fixing bases, we may write $V = \C^m = W$ for some $m$. Then the composite map 
	\[
	\fsl(\C^m) \xrightarrow{\rho_V^{-1}} \g^V = \g^W \xrightarrow{\rho_W} \fsl(C^m)
	\]
is a Lie algebra automorphism. Lie algebra automorphisms of $\fsl(\C^m)$ come in two forms:
	\begin{enumerate}
	\item $X \mapsto gXg^{-1}$ for some $g \in GL(\C^m)$, or
	\item $X \mapsto -(gX^t g^{-1})$ for some $g \in GL(\C^m)$, where $X^t$ denotes the transpose of $X$;
	\end{enumerate}
see \cite[IX.5]{Jacobson:1979}. If $\rho_W \circ \rho_V^{-1}$ is of the second form, then the $\fD(\gzero)$-module structure on $\C^m$ afforded by $\rho_W$ is isomorphic to the dual of the $\fD(\gzero)$-module structure afforded by $\rho_V$, i.e., $W \cong V^{*,\Lie}$ as $\fD(\gzero)$-modules, contradicting Lemma \ref{lemma:Gzero-modules-not-dual}. Then $\rho_W \circ \rho_V^{-1}$ must be of the first form, meaning the $\fD(\gzero)$-module structures on $\C^m$ afforded by $\rho_V$ and $\rho_W$ are isomorphic, and hence $V = W$ by Proposition \ref{prop:simples-equivalent-G0-g0}.
\end{proof}

\begin{lemma} \label{lemma:W(gn)}
Suppose \ref{item:S-union-conj-classes}--\ref{item:no-G-simple-dim-2} hold. Then for each $W \in \Irr_s(G)$, the image of $\g$ under the supermodule structure map $W: \CG \to \End(W)$ is
	\begin{equation} \label{eq:Wlambda(gn)}
	W(\g) = \begin{cases} 
	\sq(W) + \Span\{W(T_1),\ldots,W(T_t)\} & \text{if $W$ is of Type Q,} \\
	\fsl(W) & \text{if $W$ is of Type M,}
	\end{cases}
	\end{equation}
where $T_1,\ldots,T_t \in \gone$ are the class sums defined in Lemma \ref{lemma:g-in-D(g)+span}.
\end{lemma}

\begin{proof}
The ``$\subseteq$'' direction of \eqref{eq:Wlambda(gn)} follows from Lemma \ref{lemma:g-in-D(g)+span}. For the reverse inclusion, first recall that each $W \in \Irr_s(G)$ is even-dimensional. Condition \ref{item:no-G-simple-dim-2} implies that if $W \in \Irr_M(G)$, then $\dim(W) \geq 4$, and implies that if $W \in \Irr_Q(G)$, then either $\dim(W) = 2$ or $\dim(W) \geq 6$.

Suppose for the moment that $W \in \Irr_Q(G)$ and $\dim(W) = 2$. Then $\sq(W) = \C \cdot \id_W$. The homogeneous subspaces of $W$ are isomorphic one-dimensional $\CGzero$-modules, on which $\Gzero$ must act by a common linear character, say $\kappa$. Let $\tau \in S$. Then $\tau^2 = \frac{1}{2} \cdot [\tau,\tau] \in \g$, and $W(\tau^2) = \kappa(\tau^2) \cdot \id_W$ because $\tau^2 \in \Gzero$. Since $\kappa(\tau^2) \neq 0$, this implies that \eqref{eq:Wlambda(gn)} is true if $W \in \Irr_Q(G)$ and $\dim(W) = 2$.

Now suppose that $W \in \Irr_s(G)$ and $\dim(W) > 2$. Then either $W \in \Irr_M(G)$ with $\dim(W) \geq 4$, or $W \in \Irr_Q(G)$ with $\dim(W) \geq 6$. Since $W(\g)$ is a Lie subsuperalgebra of $\End(W)$, to finish the proof it then suffices by Lemma \ref{lem:generated-by-odd} to show that
	\begin{equation} \label{eq:sufficient-inclusion}
	W(\g) \supseteq \begin{cases}
	\sq(W)_{\one} & \text{if $W \in \Irr_Q(G)$,} \\
	\fsl(W)_{\one} & \text{if $W \in \Irr_M(G)$.}
	\end{cases}
	\end{equation}

First suppose $W \in \Irr_Q(G)$, with $W = S^\lambda \oplus S^{\lambda'}$ as a $\abs{\CG}$-module, $S^\lambda \cong S^{\lambda'}$ as $\CGzero$-modules, and $\dim(S^\lambda) \geq 3$. Lemma \ref{lemma:D(gzero)} implies that $W(\g)$ contains the Lie subalgebra $\fD(\sq(W)_{\zero}) \cong \fsl(S^\lambda)$, over which $\sq(W)_{\one}$ identifies with the adjoint representation of $\fsl(S^\lambda)$. In particular, $\fD(\sq(W)_{\zero})$ acts irreducibly on $\sq(W)_{\one}$. Now choose some $\tau \in S$ such that $\tau$ does not act as a scalar multiple of the identity on $S^\lambda$; this is possible because otherwise it would follow that $\gzero$ acts trivially on $S^\lambda$, contradicting Lemma \ref{lemma:G0-g0-submodule-equivalent} and the assumption $\dim(S^\lambda) \geq 3$. Then with notation and reasoning as in the proof of Lemma \ref{lemma:g-in-D(g)+span}, the operator $W(\tau - \frac{1}{\abs{C}} \cdot T) \in W(\g)$ is a nonzero element of $\sq(W)_{\one}$. Then by the simplicity of $\sq(W)_{\one}$ as a $\fD(\sq(W)_{\one})$-module, we must have $\sq(W)_{\one} \subseteq W(\g)$, and hence $\sq(W) \subseteq W(\g)$.

Now suppose $W \in \Irr_M(G)$, with $W = S^{\lambda^+} \oplus S^{\lambda^-}$ as a $\CGzero$-module. In this case, Lemma \ref{lemma:D(gzero)} implies that $W(\g)$ contains the Lie subalgebra $\fD(\fsl(W)_{\zero}) \cong \fsl(S^{\lambda^+}) \oplus \fsl(S^{\lambda^-})$, over which
	\begin{equation} \label{eq:sl(W-lam)-odd-decomp-Fn}
	\fsl(W)_{\one} \cong \Hom(S^{\lambda^+},S^{\lambda^-}) \oplus \Hom(S^{\lambda^-},S^{\lambda^+})
	\end{equation}
is the sum of two simple, non-isomorphic modules. If $\tau \in S$, then $W(\tau) \in W(\g)$ is an element of $\fsl(W)_{\one}$ having nonzero components in both summands of \eqref{eq:sl(W-lam)-odd-decomp-Fn}. Then it follows from the uniqueness of isotypical components that the $\fD(\fsl(W)_{\zero})$-submodule of $W(\g)$ generated by $W(\tau)$ must contain all of $\fsl(W)_{\one}$, and hence $\fsl(W)_{\one} \subseteq W(\g)$.
\end{proof}

\section{Type BC representation theory} \label{sec:Type-B-classical}

In Sections \ref{sec:Type-B-classical} and \ref{sec:Type-D-classical} we recall some classical results for Weyl groups of types BC and D, and then interpret those results in the context of supergroups. A somewhat more detailed (but still brief) recollection of the classical results can also be found in \cite[\S\S2--3]{DK:2025-preprint}; see also \cite{Geck:2000}. To avoid trivialities, assume throughout this section that $n \geq 2$.

\subsection{Weyl groups of type B}

Let $\calB_n = W(B_n)$ be the Weyl group of type $B_n$. It is the wreath product $\Z_2 \wr \fS_n$, i.e., a semidirect product $\Z_2^n \rtimes \fS_n$ in which $\fS_n$ acts on $\Z_2^n$ by place permutation. For $1 \leq i \leq n$ let $t_i$ be a multiplicative generator for the $i$-th factor in $\Z_2^n$. Each element of $\calB_n$ can then be uniquely written in the form $t_1^{i_1} \cdots t_n^{i_n} \sigma$ with $\sigma \in \fS_n$ and $i_j \in \set{0,1}$ for each $j$.

For $1 \leq i \leq n-1$, let $s_i = (i,i+1) \in \fS_n$. Then for $n \geq 2$, the pair $(\calB_n,\set{s_1,\ldots,s_{n-1},t_n})$ is a Coxeter system of type $B_n$, and the set of reflections in $\calB_n$ is
	\begin{equation} \label{eq:B-reflections}
	\set{ (i,j), \, t_i t_j (i,j) : 1 \leq i < j \leq n} \cup \set{ t_i : 1 \leq i \leq n}.
	\end{equation}

Let $\ve$, $\ve'$, and $\ve''$ be the group homomorphisms $\calB_n \to \set{\pm 1}$ that are defined on generators by
	\begin{equation} \label{eq:Bn-linear-characters}
	\begin{aligned}
	\ve(t_i) &= -1, \qquad & \ve'(t_i) &= -1, \qquad & \ve''(t_i) &= +1, \\
	\ve(s_j) &= -1, \qquad & \ve'(s_j) &= +1, \qquad & \ve''(s_j) &= -1.
	\end{aligned}
	\end{equation}
The map $\ve$ is the sign character of $\calB_n$.

\subsection{Simple modules for Weyl groups of type B} \label{subsec:-type-B-simples}

The simple $\CB_n$-modules are labeled by bi\-par\-ti\-tions of $n$, i.e., by ordered pairs of integer partitions $(\lambda,\mu)$ such that $\abs{\lambda}+\abs{\mu} = n$. Let $\BPn$ be the set of all bipartitions of $n$, and write $S^{(\lambda,\mu)}$ for the simple $\CB_n$-module labeled by $(\lambda,\mu)$. The trivial character and the characters  $\ve$, $\ve'$, and $\ve''$ afford the one-dimensional $\CB_n$-modules labeled by $([n],\emptyset)$, $(\emptyset,[1^n])$, $(\emptyset,[n])$, and $([1^n],\emptyset)$, respectively. More generally, if $\lambda \vdash n$, then
	\[
	S^{(\lambda,\emptyset)} = \Inf_{\fS_n}^{\calB_n}(S^\lambda) \qquad \text{and} \qquad S^{(\emptyset,\lambda)} = \Inf_{\fS_n}^{\calB_n}(S^\lambda) \otimes \ve',
	\]
where $\Inf_{\fS_n}^{\calB_n}(V)$ denotes the inflation of a $\CS_n$-module to $\CB_n$ along the canonical quotient map $\calB_n \twoheadrightarrow \fS_n$. By \cite[Theorem 5.5.6(c)]{Geck:2000}, one has
	\begin{equation} \label{eq:B-tensor-with-character}
	\begin{aligned}
	S^{(\lambda,\mu)} \otimes \ve &\cong S^{(\mu^*,\lambda^*)}, \qquad &
	S^{(\lambda,\mu)} \otimes \ve' &\cong S^{(\mu,\lambda)}, \qquad &
	S^{(\lambda,\mu)} \otimes \ve'' &\cong S^{(\lambda^*,\mu^*)}.
	\end{aligned}
	\end{equation}

\begin{lemma}[{cf.\ \cite[Lemma 2.2.1]{DK:2025-preprint}}] \label{lemma:Bn-subminimal-dimensions}
Let $(\lambda,\mu) \in \BPn$.
\begin{enumerate}
\item \label{item:dim-S(lam,mu)} $\dim(S^{(\lambda,\mu)}) = \binom{n}{\abs{\lambda}} \cdot \dim(S^\lambda) \cdot \dim(S^\mu)$, where $S^\lambda$ is the simple $\CS_n$-module labeled by $\lambda$.

\item The only one-dimensional $\CB_n$-modules are those labeled by $([n],\emptyset)$, $(\emptyset,[1^n])$, $(\emptyset,[n])$, and $([1^n],\emptyset)$, i.e., the trivial module and the modules afforded by $\ve$, $\ve'$, and $\ve''$, respectively.

\item For $n \geq 5$, if $\dim(S^{(\lambda,\mu)}) \neq 1$, then $\dim(S^{(\lambda,\mu)}) \geq n-1$.

\item If $n$ is even and $\lambda \vdash n/2$, then $\dim(S^{(\lambda,\lambda^*)}) = \dim(S^{(\lambda,\lambda)}) = \binom{n}{n/2} \cdot \dim(S^\lambda)^2$.
\end{enumerate}
\end{lemma}

For $(\lambda,\mu) \in \BPn$, let $\T(\lambda,\mu)$ be the set of all standard bitableaux of shape $(\lambda,\mu)$, i.e., the set of all ordered pairs $T = (T_{+1},T_{-1})$ such that $T_{+1}$ and $T_{-1}$ are standard tableaux of shapes $\lambda$ and $\mu$, respectively, whose boxes have collectively been filled by the integers $1,2,\ldots,n$. Given $T \in \T(\lambda,\mu)$, let $\rho_T(i) \in \set{\pm 1}$ be the subscript of the tableau ($T_{+1}$ or $T_{-1}$) in which the integer $i$ is located, and let $\res_T(i)$ be the residue (or \emph{content}, in the terminology of \cite{Mishra:2016}) of the box (in either $T_{+1}$ or $T_{-1}$) in which the integer $i$ is located. Given $T \in \T(\lambda,\mu)$ and $\sigma \in \fS_n$, the (not necessarily standard) bitableau $\sigma \cdot T$ of shape $(\lambda,\mu)$ is defined by applying $\sigma$ to the entries of $T$. The next theorem is a consequence of \cite[Theorem 6.12]{Mishra:2016}.

\begin{theorem} \label{thm:type-B-normal-form-action}
Let $(\lambda,\mu) \in \BPn$. Then there exists an orthonormal basis
	\begin{equation} \label{eq:B(lam,mu)}
	B^{(\lambda,\mu)} = \set{ c_T : T \in \T(\lambda,\mu)}
	\end{equation}
for $S^{(\lambda,\mu)}$ such that the action of the generators  $s_i,t_j \in \calB_n$ for $1 \leq i < n$ and $1 \leq j \leq n$ is as follows: Given $T \in \T(\lambda,\mu)$, set $r_i = \res_T(i+1) - \res_T(i)$, and let $S_T^{(\lambda,\mu)}$ be the span in $S^{(\lambda,\mu)}$ of $c_T$.
	\begin{enumerate}
	\item $t_j \cdot c_T = \rho_T(j) \cdot c_T$.
	
	\item Suppose $i$ and $i+1$ are not in the same tableau in $T$. Then $S \coloneq s_i \cdot T$ is standard, $s_i$ leaves $S_T^{(\lambda,\mu)} \oplus S_S^{(\lambda,\mu)}$ invariant, and the matrix of $s_i$ with respect to the basis $\{c_T,c_S \}$ is $\sm{0 & 1 \\ 1 & 0}$.
	
	\item If $i$ and $i+1$ are the same tableau in $T$, in the same row ($r_i = +1$) or in the same column $(r_i = -1$), then $s_i \cdot c_T = r_i \cdot c_T$.
	
	\item Suppose $i$ and $i+1$ are in the same tableau in $T$, but not in the same row or the same column. Then $\abs{r_i} \geq 2$, the bitableau $S \coloneq s_i \cdot T$ is standard, $s_i$ leaves $S_T^{(\lambda,\mu)} \oplus S_S^{(\lambda,\mu)}$ invariant, and the matrix of $s_i$ with respect to the basis $\set{c_T,c_S}$ is
		\[
		\begin{bmatrix}
		r_i^{-1} & \sqrt{1-r_i^{-2}} \\
		\sqrt{1-r_i^{-2}} & -r_i^{-1}
		\end{bmatrix}.
		\]
	\end{enumerate}
\end{theorem}

\begin{remark} \label{remark:YJM-eigenvectors}
Let $(\lambda,\mu) \in \BPn$ and $T \in \T(\lambda,\mu)$. The vector $c_T$ is a simultaneous eigenvector for the action of the generalized Young--Jucys--Murphy (YJM) elements $X_1,\ldots,X_n \in \CB_n$, which are defined by $X_i = \sum_{k=1}^{i-1} [(k,i) + t_k t_i(k,i)]$. The element $X_i$ acts on $c_T$ with eigenvalue $2 \cdot \res_T(i)$; see \cite[Theorem 6.5]{Mishra:2016}. Then $\calX_n \coloneq \sum_{i=1}^n X_i$ acts on $S^{(\lambda,\mu)}$ as scalar multiplication by $2 \cdot \res(\lambda,\mu)$, where $\res(\lambda,\mu) \coloneq \res(\lambda)+\res(\mu)$ is the sum of the residues of the boxes in the Young diagrams of $\lambda$ and $\mu$.
\end{remark}

Given $\lambda \vdash n$ and $\nu \vdash (n-1)$, write $\nu \prec \lambda$ if the Young diagram of $\nu$ can be obtained by removing a box from the Young diagram of $\lambda$. In this case, let $\res(\lambda/\nu)$ denote the residue of the box that is removed from $\lambda$ to obtain $\nu$. For $(\lambda,\mu) \in \BPn$ and $(\nu,\tau) \in \calBP(n-1)$, write $(\nu,\tau) \prec (\lambda,\mu)$ if either $\nu \prec \lambda$ and $\tau = \mu$, or $\nu = \lambda$ and $\tau \prec \mu$. For $(\lambda,\mu) \in \BPn$ and $T \in \T(\lambda,\mu)$, let $b_T(n)$ denote the box in $T$ in which the integer $n$ is located, and let $T/b_T(n)$, or simply $T/n$, denote the standard bitableau obtained from $T$ by removing the box $b_T(n)$. Identify $\calB_{n-1}$ with the subgroup of $\calB_n$ generated by the set $\set{s_1,\ldots,s_{n-2},t_{n-1}}$. Then for $(\lambda,\mu) \in \BPn$, one has
	\begin{equation} \label{eq:restriction-B-simple}
	\Res^{\calB_n}_{\calB_{n-1}}(S^{(\lambda,\mu)}) = \bigoplus_{(\nu,\tau) \prec (\lambda,\mu)} \Big[ \bigoplus_{\substack{T \in \T(\lambda,\mu) \\ T/n \in \T(\nu,\tau)}} S_T^{(\lambda,\mu)} \Big] \cong \bigoplus_{(\nu,\tau) \prec (\lambda,\mu)} S^{(\nu,\tau)}.
	\end{equation}
The fact that the summand indexed by $(\nu,\tau)$ is isomorphic as a $\CB_{n-1}$-module to $S^{(\nu,\tau)}$ can be seen from Theorem \ref{thm:type-B-normal-form-action}.
	
Given $(\lambda,\mu) \in \BPn$, let $R(\lambda,\mu) = (R_{+1},R_{-1})$ be the row major bitableau of shape $(\lambda,\mu)$, i.e., the standard bitableau of shape $(\lambda,\mu)$ in which $R_{+1}$ (resp.\ $R_{-1}$) is filled with the integers $1,\ldots,\abs{\lambda}$ (resp.\ $\abs{\lambda}+1,\ldots,n$) in row major order. For $T \in \T(\lambda,\mu)$, let $\sigma_T \in \fS_n$ be the permutation that maps $R(\lambda,\mu)$ to $T$. The length of $T$ is defined by $\ell(T) = \ell(\sigma_T)$, the length of $\sigma_T$ as an element of the Coxeter group $\fS_n$. If $\sigma \in \fS_n$ and $\sigma \cdot T \in \T(\lambda,\mu)$, then
	\[
	(-1)^{\ell(\sigma \cdot T)} = (-1)^{\ell(\sigma)} \cdot (-1)^{\ell(T)} = \ve''(\sigma) \cdot \ve''(\sigma_T).
	\]
Given a standard $\lambda$-tableau $T$, let $T^*$ be its transpose, which is then a standard $\lambda^*$-tableau. For $T = (T_{+1},T_{-1}) \in \T(\lambda,\mu)$, set $T^* = (T_{+1}^*,T_{-1}^*) \in \T(\lambda^*,\mu^*)$ and $T^\natural = (T_{-1},T_{+1}) \in \T(\mu,\lambda)$.

\begin{lemma} \label{lemma:sigma-R-relations}
Let $(\lambda,\mu) \in \BPn$, and let $\pi \in \fS_n$ be the $n$-cycle $(1,2,\ldots,n)$. Then
	\begin{gather}
	\sigma_{R(\lambda,\mu)^{\natural}} = \pi^{\abs{\lambda}}, \label{eq:sigma-R(lam,mu)nat} \\
	\sigma_{R(\mu^*,\lambda^*)^\natural} = [\sigma_{R(\lambda,\mu)^\natural}]^{-1} = \sigma_{R(\mu,\lambda)^\natural}, \label{eq:sigma-R(mu*,lam*)nat} \\
	\sigma_{R(\lambda^*,\mu^*)^*} = [\sigma_{R(\lambda,\mu)^*}]^{-1} = \pi^{\abs{\lambda}} \circ \sigma_{R(\mu^*,\lambda^*)^*} \circ \pi^{-\abs{\lambda}}. \label{eq:sigma-R(lam,mu)*}
	\end{gather}
\end{lemma}

\begin{proof}
The equality \eqref{eq:sigma-R(lam,mu)nat} is straightforward to check. Then $\sigma_{R(\mu^*,\lambda^*)^\natural} = \pi^{\abs{\mu^*}} = \pi^{\abs{\mu}} = \sigma_{R(\mu,\lambda)^\natural}$, giving \eqref{eq:sigma-R(mu*,lam*)nat}. Next, if $\sigma = \sigma_{R(\lambda,\mu)^*}$ maps $R(\lambda^*,\mu^*)$ to $R(\lambda,\mu)^*$, then $\sigma$ also maps $R(\lambda^*,\mu^*)^*$ to $[R(\lambda,\mu)^*]^* = R(\lambda,\mu)$, and hence $\sigma^{-1}$ maps $R(\lambda,\mu)$ to $R(\lambda^*,\mu^*)^*$, giving the first equality in \eqref{eq:sigma-R(lam,mu)*}. Now let $\sigma = \sigma_{R(\mu^*,\lambda^*)^*}$ be the permutation that maps $R(\mu,\lambda)$ to $R(\mu^*,\lambda^*)^*$. Then $\sigma \circ \pi^{-\abs{\lambda}}$ maps $R(\lambda,\mu)^{\natural}$ to $R(\mu^*,\lambda^*)^*$, and so maps $R(\lambda,\mu)^* = [R(\lambda,\mu)^\natural]^{\natural*}$ to $[R(\mu^*,\lambda^*)^*]^{\natural*} = R(\mu^*,\lambda^*)^\natural$. Then $\pi^{\abs{\lambda}} \circ \sigma \circ \pi^{-\abs{\lambda}}$ maps $R(\lambda,\mu)^*$ to $R(\lambda^*,\mu^*)$, implying the second equality in \eqref{eq:sigma-R(lam,mu)*}.
\end{proof}

Let $\approx$ be the equivalence relation on $\BPn$ generated by $(\lambda,\mu) \approx (\mu,\lambda)$ and $(\lambda,\mu) \approx (\lambda^*,\mu^*)$. Lemma \ref{lemma:sigma-R-relations} implies that the cycle types of $\sigma_{R(\lambda,\mu)^\natural}$ and $\sigma_{R(\lambda,\mu)^*}$ are constant across equivalence classes. Say $\sigma_{R(\lambda,\mu)^\natural}$ is of cycle type $(a_1, \ldots, a_r)$ and set $\velammu' = \prod_{j=1}^r i^{a_j-1} \in \set{\pm 1, \pm i}$. Similarly, define $\velammu''$ in terms of the cycle type of $\sigma_{R(\lambda,\mu)^*}$, and set $\velammu = \velammu' \cdot \velammu''$. Then
	\begin{gather*}
	[\velammu']^2 = (-1)^{\ell(R(\lambda,\mu)^\natural)}, \qquad [\velammu'']^2 = (-1)^{\ell(R(\lambda,\mu)^*)}, \\
	\text{and} \quad [\velammu]^2 = (-1)^{\ell(R(\lambda,\mu)^\natural) + \ell(R(\lambda,\mu)^*)}.
	\end{gather*}

\begin{remark} \label{remark:ve''-(lam,mu)-(lam*,mu*)}
Let $(\lambda,\mu) \in \BPn$ such that $(\lambda,\mu) = (\lambda^*,\mu^*)$, and let $d(\lambda)$ and $d(\mu)$ be the lengths of the main diagonals of the Young diagrams of $\lambda$ and $\mu$, respectively.
	\begin{enumerate}
	\item \label{item:sigma-R(lam,mu)*-sign} One sees that $\sigma_{R(\lambda,\mu)^*} = \sigma_1 \cdot \sigma_2$, where $\sigma_1$ is a product of $(\abs{\lambda}-d(\lambda))/2$ disjoint transpositions of the set $\set{1,\ldots,\abs{\lambda}}$, and $\sigma_2$ is a product of $(\abs{\mu}-d(\mu))/2$ disjoint transpositions of the set $\set{\abs{\lambda}+1,\ldots,n}$. Then $\velammu'' = i^{(n-d(\lambda)-d(\mu))/2}$.
	
	\item If $\nu$ is a partition such that $\nu = \nu^*$, then it is possible to add or remove (but not both) a box of residue zero from the Young diagram of $\nu$ to obtain the diagram of a partition $\tau$ such that $\tau = \tau^*$. It follows that $(\lambda,\mu)$ is uniquely an element of a family of bipartitions $(\lambda_0,\mu_0)$, $(\lambda_1,\mu_0)$, $(\lambda_0,\mu_1)$, and $(\lambda_1,\mu_1)$ such that $\lambda_1 \prec \lambda_0$, $\mu_1 \prec \mu_0$, and $\lambda_i = \lambda_i^*$ and $\mu_i = \mu_i^*$ for $i \in \set{0,1}$. Applying \eqref{item:sigma-R(lam,mu)*-sign}, one sees that the value of $\velammu''$ is constant across this family.
	\end{enumerate}
\end{remark}

\begin{remark} \label{remark:(lam,lam)-ve-ve''-factors}
Suppose $n$ is even and $\lambda \vdash n/2$.
	\begin{enumerate}
	\item One sees that $\sigma_{R(\lambda,\lambda)^*} = \sigma_1 \cdot \sigma_2$, where $\sigma_1$ is a permutation of the set $\set{1,\ldots,n/2}$, $\sigma_2$ is a permutation of $\set{n/2+1,\ldots,n}$, and $\sigma_2 = \pi^{n/2} \circ \sigma_1 \circ \pi^{-n/2}$. Similarly, one sees that $\sigma_{R(\lambda,\lambda^*)^*} = \sigma_3 \cdot \sigma_4$, where $\sigma_4$ is conjugate to $\sigma_3^{-1}$. Then $[\velamlam'']^2 = 1 = [\ve_{(\lambda,\lambda^*)}'']^2$.
	
	\item One sees that $\sigma_{R(\lambda,\lambda)^{\natural}}$ and $\sigma_{R(\lambda,\lambda^*)^{\natural}}$ are each the product of $n/2$ disjoint transpositions. Then $\velamlam' = i^{n/2} = \ve_{(\lambda,\lambda^*)}'$.
	\end{enumerate}
\end{remark}

\begin{lemma}
The following families of linear maps, for $(\lambda,\mu) \in \BPn$, satisfy the conditions \eqref{eq:associator} and \eqref{eq:inverse-associator} for the group $G = \calB_n$ and the characters $\ve$, $\ve'$, and $\ve''$, respectively:
	\begin{align}
	\phi_{\ve'}^{(\lambda,\mu)} &: S^{(\lambda,\mu)} \to S^{(\mu,\lambda)} & \text{defined by} \quad c_T &\mapsto c_{T^\natural}, \label{eq:phi-ve'} \\
	\phi_{\ve''}^{(\lambda,\mu)} &: S^{(\lambda,\mu)} \to S^{(\lambda^*,\mu^*)} & \text{defined by} \quad c_T &\mapsto \ve_{(\lambda,\mu)}'' \cdot (-1)^{\ell(T)} \cdot c_{T^*}, \label{eq:phi-ve''} \\
\phi_{\ve}^{(\lambda,\mu)} &: S^{(\lambda,\mu)} \to S^{(\mu^*,\lambda^*)} & \text{defined by} \quad c_T &\mapsto \velammu \cdot (-1)^{\ell(T^\natural)} \cdot c_{T^{\natural*}}. \label{eq:phi-ve}
	\end{align}
\end{lemma}

\begin{proof}
Up to scalar multiples, the maps defined here are equal to the maps of the same names defined in \cite{DK:2025-preprint}. This implies by Lemma 2.4.3, Lemma 2.4.4, and Remark 2.4.5 of \cite{DK:2025-preprint} that each family satisfies \eqref{eq:associator}, and \eqref{eq:phi-ve'} evidently satisfies \eqref{eq:inverse-associator}. To see that the family \eqref{eq:phi-ve''} satisfies \eqref{eq:inverse-associator}, observe by \eqref{eq:associator} that the composite $\phi_{\ve''}^{(\lambda^*,\mu^*)} \circ \phi_{\ve''}^{(\lambda,\mu)}$ is a $\CB_n$-module homomorphism, and hence by Schur's lemma is a scalar multiple of the identity. Letting $R = R(\lambda,\mu)$, one sees that
	\begin{align*}
	(\phi_{\ve''}^{(\lambda^*,\mu^*)} \circ \phi_{\ve''}^{(\lambda,\mu)})(c_R) &= \velammu'' \cdot (-1)^{\ell(R^*)} \cdot \ve_{(\lambda^*,\mu^*)}'' \cdot (-1)^{\ell(R)} \cdot c_R \\
	&= [\velammu'']^2 \cdot (-1)^{\ell(R^*)} \cdot c_R = c_R,
	\end{align*}
and hence $\phi_{\ve''}^{(\lambda^*,\mu^*)} \circ \phi_{\ve''}^{(\lambda,\mu)} = \id_{S^{(\lambda,\mu)}}$. Similarly one checks that $\phi_{\ve}^{(\mu^*,\lambda^*)} \circ \phi_{\ve}^{(\lambda,\mu)} = \id_{S^{(\lambda,\mu)}}$.
\end{proof}

\begin{remark} \label{remark:phi-ve-composition}
For $(\lambda,\mu) \in \BPn$, one has $\phi_{\ve}^{(\lambda,\mu)} = \velammu' \cdot \phi_{\ve''}^{(\mu,\lambda)} \circ \phi_{\ve'}^{(\lambda,\mu)}$.
\end{remark}

\begin{remark} \label{remark:phi-ve''-ve'-commutation}
Using \eqref{eq:sigma-R(lam,mu)nat}, one sees that $\phi_{\ve''}^{(\mu,\lambda)} \circ \phi_{\ve'}^{(\lambda,\mu)} = (-1)^{\abs{\lambda} \cdot (n-1)} \cdot \phi_{\ve'}^{(\lambda^*,\mu^*)} \circ \phi_{\ve''}^{(\lambda,\mu)}$.
\end{remark}

\begin{remark}
For $\lambda \vdash n$ with $\lambda = \lambda^*$, the map $\phi_{\ve''}^{(\lambda,\emptyset)}: S^{(\lambda,\emptyset)} \to S^{(\lambda,\emptyset)}$ can be identified with the associator $\phi_\lambda$ defined by Geetha and Prasad \cite[Eqn 9]{Geetha:2018}.
\end{remark}

\subsection{The even subgroup of the type B Weyl group} \label{subsec:even-subgroup-type-B}

We now consider $\calB_n$ as a super\-group via its sign character $\ve: \calB_n \to \set{\pm 1}$. Making the identification $\calB_n = \Z_2^n \rtimes \fS_n$, one has
	\[
	\calB_{\zero} = (\calB_n)_{\zero} = \ker(\ve) = \big[(\Z_2^n)_{\zero} \rtimes \fA_n \big] \cup \big[(\Z_2^n)_{\one} \rtimes \fS_{\one} \big].
	\]
Here $(\Z_2^n)_{\zero}$ (resp.\ $(\Z_2^n)_{\one}$) is the subset of elements that are a product of an even (resp.\ odd) number of the generators $t_1,\ldots,t_n$, and $\fS_{\one} = (\fS_n)_{\one}$ is the set of odd permutations in $\fS_n$.

Recall the $\CB_n$-module isomorphism $S^{(\lambda,\mu)} \otimes \ve \cong S^{(\mu^*,\lambda^*)}$ from \eqref{eq:B-tensor-with-character}. Let $\sim$ be the equivalence relation on $\BPn$ generated by $(\lambda,\mu) \sim (\mu^*,\lambda^*)$. Write $[\lambda,\mu]$ for the equivalence class of $(\lambda,\mu)$, and let $\BPboxn$ be the set of all equivalence classes under $\sim$. Then $\BPboxn = \Eboxn \cup \Fboxn$, where
	\begin{equation} \label{eq:BP[n]-subsets}
	\begin{split}
	\Eboxn &= \set{ [\lambda,\mu] : (\lambda,\mu) \neq (\mu^*,\lambda^*)} = \set{ [\lambda,\mu] : \mu \neq \lambda^*}, \\ 
	\Fboxn &= \set{ [\lambda,\mu] : (\lambda,\mu) = (\mu^*,\lambda^*)} = \set{ [\lambda,\lambda^*] : \lambda \vdash n/2}.
	\end{split}
	\end{equation}
One has $\Fboxn \neq \emptyset$ only if $n$ is even. Now applying Lemma \ref{lemma:classical-Clifford} for $(G,\kappa) = (\calB_n,\ve)$, one gets:

\begin{lemma} \label{lemma:B0-simples}
Up to isomorphism, each simple $\CBzero$-module arises uniquely via either:
	\begin{enumerate}
	\item If $[\lambda,\mu] \in \Eboxn$, then $S^{[\lambda,\mu]} \coloneq \Res^{\calB_n}_{\Bzero}( S^{(\lambda,\mu)}) \cong \Res^{\calB_n}_{\Bzero}( S^{(\mu^*,\lambda^*)})$ is simple and self-conjugate.
	
	\item If $[\lambda,\lambda^*] \in \Fboxn$, then $S^{[\lambda,\lambda^*]} \coloneq \Res^{\calB_n}_{\Bzero}(S^{(\lambda,\lambda^*)})$ is the direct sum of two simple, conjugate, non-iso\-mor\-phic $\CBzero$-modules $S^{[\lambda,\lambda^*,+]}$ and $S^{[\lambda,\lambda^*,-]}$, equal to the $\pm 1$ eigen\-spaces of $\phi_{\ve}^{(\lambda,\lambda^*)}$.
	\end{enumerate}
\end{lemma}

\begin{lemma} \label{lemma:Bzero-subminimal-dimensions}
Let $n \geq 5$.
\begin{enumerate}
\item The only one-dimensional $\CBzero$-modules are those labeled by $[ [n], \emptyset] $ and $[[1^n],\emptyset]$, i.e., the restrictions of the trivial module and the module afforded by $\ve''$.

\item If $[\lambda,\mu] \in \BPboxn$ and $\dim(S^{[\lambda,\mu]}) \neq 1$, then $\dim(S^{[\lambda,\mu]}) \geq n-1$.

\item If $[\lambda,\lambda^*] \in \BPboxn$, then $\dim(S^{[\lambda,\lambda^*,\pm]}) = \frac{1}{2} \cdot \binom{n}{n/2} \cdot \dim(S^\lambda)^2 \geq \frac{1}{2} \cdot \binom{n}{n/2}$.
\end{enumerate}
\end{lemma}

\begin{proof}
Apply Lemma \ref{lemma:Bn-subminimal-dimensions}.
\end{proof}

\begin{lemma} \label{lemma:Bzero-restriction}
Restriction from $(\calB_n)_{\zero}$ to $(\calB_{n-1})_{\zero}$ is multiplicity free. Specifically:
	\begin{enumerate}
	\item \label{item:B0-restriction-Eboxn} If $[\lambda,\mu] \in \Eboxn$, then
		\[
		\Res^{(\calB_n)_{\zero}}_{(\calB_{n-1})_{\zero}}(S^{[\lambda,\mu]}) 
		\cong 
		\Big[ \bigoplus_{\substack{(\nu,\tau) \prec (\lambda,\mu) \\ (\nu,\tau) \neq (\tau^*,\nu^*)}} S^{[\nu,\tau]} \Big] 
		\oplus 
		\Big[ \bigoplus_{\substack{(\nu,\tau) \prec (\lambda,\mu) \\ (\nu,\tau) = (\tau^*,\nu^*)}} S^{[\nu,\tau,+]} \oplus S^{[\nu,\tau,-]} \Big].
		\]

	\item \label{item:B0-restriction-Fboxn} If $[\lambda,\lambda^*] \in \Fboxn$, then
		\[
		\Res^{(\calB_n)_{\zero}}_{(\calB_{n-1})_{\zero}}(S^{[\lambda,\lambda^*,\pm]}) 
		\cong 
		\bigoplus_{\nu \prec \lambda} S^{[\nu,\lambda^*]}.
		\]
	\end{enumerate}
\end{lemma}

\begin{proof}
The formula in \eqref{item:B0-restriction-Eboxn} follows immediately from \eqref{eq:restriction-B-simple} and Lemma \ref{lemma:B0-simples}. For \eqref{item:B0-restriction-Fboxn}, note that $\nu \prec \lambda$ if and only if $\nu^* \prec \lambda^*$, and observe that
	\[
	\Res^{\calB_n}_{\calB_{n-1}}(S^{(\lambda,\lambda^*)}) \cong \bigoplus_{\nu \prec \lambda} [S^{(\nu,\lambda^*)} \oplus S^{(\lambda,\nu^*)}].
	\]
This restriction is multiplicity free, so the $\CB_{n-1}$-module decomposition of $S^{(\lambda,\lambda^*)}$ is canonical. For $\nu \prec \lambda$, one sees that $\phi_\ve^{(\lambda,\lambda^*)}(S^{(\nu,\lambda^*)})$ is a $\CB_{n-1}$-submodule of $S^{(\lambda,\lambda^*)}$ isomorphic to $S^{(\nu,\lambda^*)} \otimes \ve$, and hence $\phi_\ve^{(\lambda,\lambda^*)}(S^{(\nu,\lambda^*)}) = S^{(\lambda,\nu^*)}$. Then $\phi_\ve^{(\lambda,\lambda^*)}$ leaves the subspace $S^{(\nu,\lambda^*)} \oplus S^{(\lambda,\nu^*)}$ invariant. This subspace decomposes nontrivially into $+1$ and $-1$ eigenspaces for $\phi_\ve^{(\lambda,\lambda^*)}$, and the eigenspaces are then nonzero $\C(\calB_{n-1})_{\zero}$-modules. Since $S^{(\nu,\lambda^*)} \cong S^{[\nu,\lambda^*]} \cong S^{(\lambda,\nu^*)}$ as $\C(\calB_{n-1})_{\zero}$-modules, it follows that $S^{[\lambda,\lambda^*,+]}$ and $S^{[\lambda,\lambda^*,-]}$ each contain one factor of $S^{[\nu,\lambda^*]}$.
\end{proof}

\subsection{Simple supermodules in Type B} \label{subsec:simple-supermodules-B}

Applying Proposition \ref{prop:simple-supermodules} and \eqref{eq:B-tensor-with-character}, one now gets:

\begin{proposition} \label{prop:B-simple-supermodules}
Up to homogeneous isomorphism, each simple $\CB_n$-supermodule occurs in exactly one of the following ways:
	\begin{enumerate}
	\item \label{item:B-simple-super-self-associate} For each $[\lambda,\mu] \in \Eboxn$, there exists a Type Q simple $\CB_n$-supermodule $W^{[\lambda,\mu]}$ such that $W^{[\lambda,\mu]} = S^{(\lambda,\mu)} \oplus S^{(\mu^*,\lambda^*)}$ as a $\abs{\CB_n}$-module, with
		\[
		\Wzero^{[\lambda,\mu]} = \{u + \phi_{\ve}^{(\lambda,\mu)}(u) : u \in S^{(\lambda,\mu)} \}, \quad
		\Wone^{[\lambda,\mu]} = \{ u - \phi_{\ve}^{(\lambda,\mu)}(u): u \in S^{(\lambda,\mu)} \},
		\]
	and odd involution $J^{(\lambda,\mu)}$ defined for $u \in S^{(\lambda,\mu)}$ by $J^{(\lambda,\mu)}(u \pm \phi_{\ve}^{(\lambda,\mu)}(u)) = u \mp \phi_{\ve}^{(\lambda,\mu)}(u)$.

	\item For each $[\lambda,\lambda^*] \in \Fboxn$, there exists a Type M simple $\CB_n$-supermodule $W^{[\lambda,\lambda^*]}$ such that $W^{[\lambda,\lambda^*]} = S^{(\lambda,\lambda^*)}$ as a $\abs{\CB_n}$-module, with $\Wzero^{[\lambda,\lambda^*]} = S^{[\lambda,\lambda^*,+]}$ and $\Wone^{[\lambda,\lambda^*]} = S^{[\lambda,\lambda^*,-]}$.
	\end{enumerate}
Each simple $\CB_n$-supermodule is uniquely determined, up to an even isomorphism, by its summands as a $\CBzero$-module, and by the superdegrees in which those summands are concentrated.
\end{proposition}

Then applying Corollary \ref{cor:super-artin-wedderburn-CG}, one gets:

\begin{corollary} \label{cor:CBn-superalgebra-structure}
The sum of the simple $\CB_n$-supermodule structure maps induces a superalgebra isomorphism
	\begin{align*}
	\CB_n &\cong \Big[ \bigoplus_{[\lambda,\mu] \in \Eboxn} Q( W^{[\lambda,\mu]} ) \Big] \oplus \Big[ \bigoplus_{[\lambda,\lambda^*] \in \Fboxn} \End( W^{[\lambda,\lambda^*]} ) \Big].
	\end{align*}
In particular, for $[\lambda,\mu] \in \Eboxn$, the map $W^{[\lambda,\mu]}: \CB_n \to \End(W^{[\lambda,\mu]})$ has image in $Q(W^{[\lambda,\mu]})$.
\end{corollary}

\begin{lemma} \label{lemma:Bn-super-subminimal-dimensions}
Suppose $n \geq 5$.
	\begin{enumerate}
	\item \label{item:Bn-super-subminimal-En} Let $[\lambda,\mu] \in \Eboxn$. If $[\lambda,\mu] = [[n],\emptyset]$ or $[\lambda,\mu] = [[1^n],\emptyset]$, then $\dim(W^{[\lambda,\mu]}) = 2$. Otherwise,
		\[
		\dim(W^{[\lambda,\mu]}) \geq 2n-2.
		\]
	
	\item \label{item:Bn-super-subminimal-Fn} Let $[\lambda,\lambda^*] \in \Fboxn$. Then $\dim(W^{[\lambda,\lambda^*]}) = \binom{n}{n/2} \cdot \dim(S^\lambda)^2 \geq \binom{n}{n/2}$.
	\end{enumerate}
\end{lemma}

\begin{proof}
Apply Lemma \ref{lemma:Bn-subminimal-dimensions}.
\end{proof}

\begin{proposition} \label{prop:B-to-Bn-1-supermodule-restriction}
Let $n \geq 3$.
	\begin{enumerate}
	\item \label{item:En-restriction-to-Bn-1} If $[\lambda,\mu] \in \Eboxn$, then
		\[
		\Res^{\calB_n}_{\calB_{n-1}}(W^{[\lambda,\mu]}) \cong
		\Big[ \bigoplus_{\substack{(\nu,\tau) \prec (\lambda,\mu) \\ (\nu,\tau) \neq (\tau^*,\nu^*)}} W^{[\nu,\tau]} \Big] \oplus
		\Big[ \bigoplus_{\substack{(\nu,\tau) \prec (\lambda,\mu) \\ (\nu,\tau) = (\tau^*,\nu^*)}} W^{[\nu,\tau]} \oplus \Pi(W^{[\nu,\tau]}) \Big].
		\]
	
	\item \label{item:Fn-restriction-to-Bn-1} If $[\lambda,\lambda^*] \in \Fboxn$, then
		\[
		\Res^{\calB_n}_{\calB_{n-1}}(W^{[\lambda,\lambda^*]}) \cong \bigoplus_{\nu \prec \lambda} W^{[\nu,\lambda^*]}.
		\]
	\end{enumerate}
\end{proposition}

\begin{proof}
The formulas are obtained by considering $W^{[\lambda,\mu]}$ as a $\abs{\CB_{n-1}}$-module using \eqref{eq:restriction-B-simple}, noting that $(\nu,\tau) \prec (\lambda,\mu)$ if and only if $(\tau^*,\nu^*) \prec (\mu^*,\lambda^*)$, and from the fact that the simple $\CB_{n-1}$-super\-modules are determined (up to even isomorphism) by their restrictions to $\C(\calB_{n-1})_{\zero}$ and by the super\-degrees in which the simple $\C(\calB_{n-1})_{\zero}$-module summands are concentrated. First one observes that the formulas in \eqref{item:En-restriction-to-Bn-1} and \eqref{item:Fn-restriction-to-Bn-1} are true at the level of $\abs{\CB_{n-1}}$-modules. This implies by the multiplicity of the composition factors that the restriction formulas in \eqref{item:En-restriction-to-Bn-1} and \eqref{item:Fn-restriction-to-Bn-1} are also true at the level of $\CB_{n-1}$-super\-modules, except perhaps in case \eqref{item:En-restriction-to-Bn-1} if there exists some $(\nu,\tau) \prec (\lambda,\mu)$ such that $(\nu,\tau) = (\tau^*,\nu^*)$.

There exists at most one $(\nu,\tau) \prec (\lambda,\mu)$ such that $(\nu,\tau) = (\tau^*,\nu^*)$. If such a $(\nu,\tau)$ exists, then the multiplicity-two $\abs{\CB_{n-1}}$-composition factor of $S^{(\nu,\tau)}$ in $S^{(\lambda,\mu)} \oplus S^{(\mu^*,\lambda^*)}$ must correspond to supermodule summands $W^{[\nu,\tau]}$ and $\Pi(W^{[\nu,\tau]})$, because the Type Q supermodule $W^{[\lambda,\mu]}$ is even-isomorphic to $\Pi(W^{[\lambda,\mu]})$, but $W^{[\nu,\tau]}$ is not even-isomorphic to itself for $[\nu,\tau] \in \Fboxn$.
\end{proof}

\subsection{The Lie superalgebra of reflections in type B} \label{subsec:Lie-superalgebra-B}

For $n \geq 2$, let $\fb_n \subseteq \Lie(\CB_n)$ be the Lie superalgebra generated by the set \eqref{eq:B-reflections} of all reflections in $\calB_n$. When the value of $n$ is clear from the context, we may write $\fbzero = (\fb_n)_{\zero}$. In this context, \eqref{eq:LieCG} and \eqref{eq:D(LieCG)} take the forms
	\begin{gather}
	\Lie (\CB_n) \cong 
	\Big[ \bigoplus_{[\lambda,\mu] \in \Eboxn} \fq( W^{[\lambda,\mu]} ) \Big] 
	\oplus 
	\Big[ \bigoplus_{[\lambda,\lambda^*] \in \Fboxn} \gl( W^{[\lambda,\lambda^*]} ) \Big], \label{eq:CBn-Artin-Wedderburn} \\
	\fD(\CB_n) \cong 
	\Big[ \bigoplus_{[\lambda,\mu] \in \Eboxn} \sq( W^{[\lambda,\mu]} ) \Big] 
	\oplus 
	\Big[ \bigoplus_{[\lambda,\lambda^*] \in \Fboxn} \fsl( W^{[\lambda,\lambda^*]} ) \Big]. \label{eq:D(CBn)-Artin-Wedderburn}
	\end{gather}

Let $\calX_n = \sum_{j=1}^n X_j = \sum_{j=1}^n \sum_{i=1}^{j-1} [(i,j) + t_it_j(i,j)]$ be the sum in $\CB_n$ of the YJM elements, and let $\calT_n = \sum_{j=1}^n t_j$. The set \eqref{eq:B-reflections} is a union of two conjugacy classes in $\calB_n$. Then
	\begin{equation} \label{eq:bn-easy-inclusion}
	\fb_n \subseteq \fD(\CB_n) + \C \cdot \calX_n + \C \cdot \calT_n,
	\end{equation}
by Lemma \ref{lemma:g-in-D(g)+span}. Our goal in this section is to prove the following theorem:

\begin{theorem} \label{theorem:main-theorem-B}
Let $n \geq 2$. Then $\fb_n = \fD(\CB_n) + \C \cdot \calX_n + \C \cdot \calT_n$.
\end{theorem}

\begin{lemma} \label{lemma:base-cases-B}
If $n \in \set{2,3,4,5}$, then $\fb_n = \fD(\CB_n) + \C \cdot \calX_n + \C \cdot \calT_n$.
\end{lemma}

\begin{proof}
By dimension comparison in GAP \cite{GAP4}; see Appendix \ref{app:dim-small-ranks} for more details.
\end{proof}

\begin{lemma} \label{lemma:b0-generates-CB0}
For $n \geq 5$, the set
	\begin{equation} \label{eq:ti(j,k)-CC}
	C_{(21^{n-3},1)} \coloneq \set{ t_i(j,k), \, t_it_jt_k (j,k) : \text{$i,j,k$ distinct}}
	\end{equation}
is a conjugacy class $S'$ in $\Bzero$ such that $S' \subseteq \fbzero$ and $\Bzero = \subgrp{S'}$.
\end{lemma}

\begin{proof}
The set $S' = C_{(21^{n-3},1)}$ is a conjugacy class in $\calB_n$; the labeling corresponds to that in \cite[\S2]{Okada:1990} or \cite[\S2]{Mishra:2016}. If $i,j,k,\ell$ are distinct, then the element $t_i(j,k) \in S'$ is centralized by the element $t_\ell$, which is in $\calB_n$ but not $\Bzero$. Then by the standard criterion, $S'$ remains a conjugacy class in $\Bzero$ (rather than splitting into a union of two conjugacy classes in $\Bzero$).

If $i,j,k,\ell$ are distinct, then
	\begin{align*}
	t_i (j,k) &= \tfrac{1}{2} \cdot [t_i, (j,k)] \in \fbzero, \\
	t_i t_j t_k (j,k) &= \tfrac{1}{2} \cdot [t_i, t_jt_k (j,k)] \in \fbzero,
	\end{align*}
so $S' \subseteq \fbzero$. The group $\Dzero$ is an index-$2$ subgroup of $\Bzero$, and hence $\Bzero$ is generated by $\Dzero$ together with any single element of the form $t_i(j,k)$, i.e., by any element of $\Bzero$ not in $\Dzero$. The group $\Dzero$ is generated by all elements of the form
	\[
	[t_i(j,k)] \cdot [t_it_jt_k(j,k)] = t_j t_k
	\]
for $i,j,k$ distinct, and by the alternating group $\fA_n$, which for $n \geq 5$ is generated by all elements
	\[
	[t_i (j,k)] \cdot [t_i (\ell,m)] = (j,k)(\ell,m).
	\]
for $i,j,k,\ell,m$ distinct. Then for $n \geq 5$, the group $\Bzero$ is generated by the set $S'$.
\end{proof}

\begin{proposition} \label{prop:n-geq-5-Bn-good}
Let $n \geq 5$. Then \ref{item:S-union-conj-classes}--\ref{item:S'-one-class-gen-G0} and \ref{item:no-G-simple-dim-2} are satisfied for the supergroup $G = \calB_n$, with $S$ the set \eqref{eq:B-reflections} of all reflections in $\calB_n$, and $S'$ the set \eqref{eq:ti(j,k)-CC}. Consequently, the results of Section \ref{subsec:lie-superalgebras-odd-elements} all hold in this context.
\end{proposition}

\begin{proof}
Apply Lemmas \ref{lemma:b0-generates-CB0} and \ref{lemma:Bn-subminimal-dimensions}.
\end{proof}

\begin{lemma} \label{lemma:B-inductive-supergroup-ngeq6}
Let $n \geq 6$, let $G = \calB_n$, and let $S$ be the set \eqref{eq:B-reflections} of all reflections in $\calB_n$. Suppose Theorem \ref{theorem:main-theorem-B} is true for the value $n-1$. Then \ref{item:inductive-subgroup} is satisfied by the subsupergroup $H = \calB_{n-1}$, with $S_H$ equal to the set of all reflections in $\calB_{n-1}$.
\end{lemma}

\begin{proof}
Condition \ref{item:S-union-conj-classes} evidently holds for $H$, and \ref{item:gzero-generates-CGzero} holds by Lemma \ref{lemma:b0-generates-CB0}. If Theorem \ref{theorem:main-theorem-B} is true for the value $n-1$, then $\fh = \fD(\CH) + \C \cdot \calX_{n-1} + \C \cdot \calT_{n-1}$, and hence $\hzero = \fD(\CH)_{\zero}$. Finally, the last two bullet points in \ref{item:inductive-subgroup} are satisfied by Lemmas \ref{lemma:Bzero-subminimal-dimensions} and \ref{lemma:Bzero-restriction}.
\end{proof}

\begin{proposition} \label{prop:A5-for-B}
Let $n \geq 6$, and suppose Theorem \ref{theorem:main-theorem-B} is true for the value $n-1$. Then \ref{item:D(gzero)-to-sl(V)} holds for the supergroup $G = \calB_n$.
\end{proposition}

\begin{proof}
Set $G = \calB_n$ and $\g = \fb_n$. Let $V \in \Irr(\Gzero)$, and let $\rho_V: \fD(\gzero) \to \End(V)$ be the restriction to $\fD(\gzero)$ of the module structure map $V: \CGzero \to \End(V)$. Then $\im(\rho_V) \subseteq \fsl(V)$. The map $\rho_V$ is evidently a surjection of $\dim(V) = 1$, since then $\fsl(V) = 0$, so assume that $\dim(V) > 1$.

Let $H = \calB_{n-1}$, and let $\fh \subseteq \Lie(\CH)$ be the Lie subsuperalgebra generated by the set $S_H$ of reflections in $H$. By Lemma \ref{lemma:B-inductive-supergroup-ngeq6}, $H$ and $S_H$ satisfy \ref{item:inductive-subgroup}. Then $\fD(\hzero)$ is semisimple by \eqref{eq:D(hzero)}, and $\fD(\gzero)$ is semisimple by Proposition \ref{prop:n-geq-5-Bn-good}. Let $\g^V = \rho_V(\fD(\gzero))$ and let $\fh^V = \rho_V(\fD(\hzero))$. Then $\fh^V \subseteq \g^V$ are semisimple subalgebras of $\fsl(V)$, and $\g^V$ acts irreducibly on $V$ (again by Proposition \ref{prop:n-geq-5-Bn-good}). The restriction of $V$ to $\CHzero$ is multiplicity free by Lemma \ref{lemma:Bzero-restriction}; say $V = \bigoplus_{i=1}^t V_i$, where the $V_i$ are pairwise non-isomorphic simple $\CHzero$-modules. Then \eqref{eq:D(hzero)} implies that $\fh^V = \bigoplus_{i=1}^t \fsl(V_i)$, and hence each simple ideal $\fsl(V_i)$ of $\fh^V$ admits a composition factor $V_i$ in $V$ of multiplicity one.\footnote{If $\dim(V_i) = 1$, then $\fsl(V_i) = 0$, and we ignore that term for the purpose of this sentence.} Write $\rk(L)$ for the rank of a semisimple Lie algebra $L$. Then $\rk(\fh^V) \leq \rk(\g^V)$. We want to show that $\dim(V) < 2 \cdot \rk(\fh^V)$; this will imply first by \cite[Lemme 15]{Marin:2007} that $\g^V$ is a simple Lie algebra, and then by \cite[Lemme 13]{Marin:2007} that $\g^V \cong \fsl(V)$, finishing the proof.

Since $\rk(\fh^V) = \sum_{i=1}^t \rk(\fsl(V_i))$ and since $\dim(V) = \sum_{i=1}^t \dim(V_i)$, then
	\[
	2 \cdot \rk(\fh^V) - \dim(V) = 2 \cdot \Big( \sum_{i=1}^t \dim(V_i)-1 \Big) - \dim(V) = \sum_{i=1}^t \left[ \dim(V_i)-2 \right].
	\]
Since $n \geq 6$, and hence $n-1 \geq 5$, we get by Lemma \ref{lemma:Bzero-subminimal-dimensions} that the only one-dimensional $\CHzero$-modules are those labeled by $[[n-1],\emptyset]$ and $[[1^{n-1}],\emptyset]$, and all other simple $\CHzero$ are of dimension at least $4$. Thus, if $\dim(V_i) > 1$ for all $V_i$, then $2 \cdot \rk(\fh^V) - \dim(V) > 0$.

If $S^{[[n-1],\emptyset]}$ occurs in $V$, then we see from Lemma \ref{lemma:Bzero-restriction} that either $V \cong S^{[[n-1,1],\emptyset]}$ or $V \cong S^{[[n-1],[1]]}$; the possibility $V \cong S^{[[n],\emptyset]}$ is precluded by the assumption that $\dim(V) > 1$. Similarly, if $S^{[[1^{n-1}],\emptyset]}$ occurs in $V$, then either $V \cong S^{[[2,1^{n-2}],\emptyset]}$ or $V \cong S^{[[1^{n-1}],[1]]}$. For $V \cong S^{[[n-1,1],\emptyset]}$, one has $\Res^{\Gzero}_{\Hzero}(V) \cong S^{[[n-2,1],\emptyset]} \oplus S^{[[n-1],\emptyset]}$ and $\dim(V) = n-1$. Then
	\[
	2 \cdot \rk(\fh_V) - \dim(V) = 2 \cdot \big( [(n-2)-1] + [1-1] \big) - (n-1) = n - 5 > 0. 
	\]
For $V \cong S^{[[n-1],[1]]}$, one has $\Res^{\Gzero}_{\Hzero}(V) \cong S^{[[n-2],[1]]} \oplus S^{[[n-1],\emptyset]}$ and $\dim(V) = n$. Then
	\[
	2 \cdot \rk(\fh_V) - \dim(V) = 2 \cdot \big( [n-1] + [1-1] \big) - n = n - 2 > 0. 
	\]
Similarly, one checks for $V \cong S^{[[2,1^{n-2}],\emptyset]}$ or $V \cong S^{[[1^{n-1}],[1]]}$ that $2 \cdot \rk(\fh^V) - \dim(V) > 0$.
\end{proof}

We can now prove Theorem \ref{theorem:main-theorem-B}.

\begin{proof}[Proof of Theorem \ref{theorem:main-theorem-B}]
The proof is by induction on $n$. The base cases $n \in \set{2,3,4,5}$ are handled by Lemma \ref{lemma:base-cases-B}, so let $n \geq 6$, and suppose the theorem is true for the value $n-1$. By \eqref{eq:bn-easy-inclusion} one has $\fb_n \subseteq \fD(\CB_n) + \C \cdot \calX_n + \C \cdot \calT_n$, and $\calX_n,\calT_n \in \fb_n$, so it remains to show that $\fD(\CB_n) \subseteq \fb_n$. In the rest of the proof we identify $\CB_n$ with its image under the isomorphism \eqref{eq:CBn-Artin-Wedderburn}.

First we will show that $\fD(\CB_n)_{\one} \subseteq \fb_n$. By \eqref{eq:D(CBn)-Artin-Wedderburn}, one has
	\begin{multline*}
	\fD(\CB_n)_{\one} = 
	\Big[ \bigoplus_{[\lambda,\mu] \in \Eboxn} \sq(W^{[\lambda,\mu]})_{\one} \Big] 
	\oplus 
	\Big[ \bigoplus_{[\lambda,\lambda^*] \in \Fboxn} \fsl(W^{[\lambda,\lambda^*]})_{\one} \Big] \\
	\cong 
	\Big[ \bigoplus_{[\lambda,\mu] \in \Eboxn} \fsl(S^{[\lambda,\mu]}) \Big] 
	\oplus 
	\Big[ \bigoplus_{[\lambda,\lambda^*] \in \Fboxn} \Hom(S^{[\lambda,\lambda^*,+]},S^{[\lambda,\lambda^*,-]}) \oplus \Hom(S^{[\lambda,\lambda^*,-]},S^{[\lambda,\lambda^*,+]}) \Big],
	\end{multline*}
and by Lemma \ref{lemma:D(gzero)}, one has
	\begin{align*}
	\fD(\fbzero) &\cong 
	\Big[ \bigoplus_{[\lambda,\mu] \in \Eboxn} \fsl(S^{[\lambda,\mu]}) \Big] 
	\oplus 
	\Big[ \bigoplus_{[\lambda,\lambda^*] \in \Fboxn} \fsl(S^{[\lambda,\lambda^*,+]}) \oplus \fsl(S^{[\lambda,\lambda^*,-]}) \Big].
	\end{align*}
Then $\fD(\CB_n)_{\one}$ is a direct sum of pairwise non-isomorphic simple modules for the semisimple Lie algebra $\fD(\fbzero)$. Since $\fbone$ is a $\fD(\fbzero)$-submodule of $\fD(\CB_n)_{\one} + \C \cdot \calX_n + \C \cdot \calT_n$, it must contain some subset of the summands in $\fD(\CB_n)_{\one}$. Using Lemma \ref{lemma:W(gn)}, we see that each of these summands is contained in the image of the corresponding projection map $W^{[\lambda,\mu]}: \fb_n \to \End(W^{[\lambda,\mu]})$, and hence must have been contained in $\fbone$. Thus $\fD(\CB_n)_{\one} \subseteq \fb_n$.

Now applying Lemmas \ref{lemma:Bn-super-subminimal-dimensions} and \ref{lem:generated-by-odd}, we deduce that
	\begin{enumerate}
	\item \label{item:sl-in-bn} $\fsl(W^{[\lambda,\lambda^*]}) \subseteq \fb_n$ for each $[\lambda,\lambda^*] \in \Fboxn$, and
	\item \label{item:sq-in-bn} $\sq(W^{[\lambda,\mu]}) \subseteq \fb_n$ for all $[\lambda,\mu] \in \Eboxn$, except perhaps for $[\lambda,\mu] = [[n],\emptyset]$ or $[\lambda,\mu] = [[1^n],\emptyset]$, i.e., $\sq(W) \subseteq \fb_n$ for all $W \in \Irr_s(\calB_n)$ such that $\dim(W) > 2$.
	\end{enumerate}
It remains to show that $\sq(W^{[\lambda,\mu]}) \subseteq \fb_n$ for $[\lambda,\mu] = [[n],\emptyset]$ and $[\lambda,\mu] = [[1^n],\emptyset]$. In these cases one has $\sq(W^{[\lambda,\mu]}) = \C \cdot \id_{W^{[\lambda,\mu]}}$. First, if $\tau \in \calB_n$ is any reflection, then $1_{\CB_n} = \tau^2 = \frac{1}{2} \cdot [\tau,\tau] \in \fb$. Under the isomorphism of Corollary \ref{cor:CBn-superalgebra-structure}, one has $1_{\CB_n} = \sum_{[\lambda,\mu] \in \BPboxn} \id_{W^{[\lambda,\mu]}}$. Let
	\begin{equation} \label{eq:Psi-B}
	\Psi = 1_{\CB_n} 
	- \Big( \sum_{\substack{[\lambda,\mu] \in \BPboxn \\ [\lambda,\mu] \neq [[n],\emptyset] \\ [\lambda,\mu] \neq [[1^n],\emptyset]}} \id_{W^{[\lambda,\mu]}} \Big) 
	= 
	1_{\CB_n} 
	- \Big( \sum_{\substack{W \in \Irr_s(\calB_n) \\ \dim(W) > 2}} \id_{W} \Big)
	\end{equation}
Then $\Psi \in \fb$, because it is a linear combination of elements known to be in $\fb$, and $\Psi$ acts as zero on all simple $\CB_n$-supermodules except for $W^{[[n],\emptyset]}$ and $W^{[[1^n],\emptyset]}$, on which it acts as the identity. Next consider the element $t_1 s_2 = \frac{1}{2} \cdot [t_1,s_2] \in \fbzero$. Items \eqref{item:sl-in-bn} and \eqref{item:sq-in-bn} above imply that $W^{[\lambda,\mu]}(t_1s_2) \in \fbzero$ for all $[\lambda,\mu] \notin \set{[[n],\emptyset], [[1^n],\emptyset]}$. Set
	\begin{equation} \label{eq:Phi-B}
	\Phi = t_1s_2 
	- \Big( \sum_{\substack{[\lambda,\mu] \in \BPboxn \\ [\lambda,\mu] \neq [[n],\emptyset] \\ [\lambda,\mu] \neq [[1^n],\emptyset]}} W^{[\lambda,\mu]}(t_1s_2) \Big).
	\end{equation}
Then $\Phi \in \fb$, and $\Phi$ acts as zero on all simple $\CB_n$-supermodules except for $W^{[[n],\emptyset]}$ and $W^{[[1^n],\emptyset]}$. One checks that $t_1s_2$ (hence also $\Phi$) acts on $W^{[[n],\emptyset]}$ as the identity, and acts on $W^{[[1^n],\emptyset]}$ as multiplication by $-1$. Then $\id_{W^{[[n],\emptyset]}} = \frac{1}{2}(\Psi + \Phi) \in \fb_n$, and $\id_{W^{[[1^n],\emptyset]}} = \frac{1}{2}(\Psi - \Phi) \in \fb_n$. Thus, $\sq(W^{[\lambda,\mu]}) \subseteq \fb_n$ for all $[\lambda,\mu] \in \Eboxn$, and we conclude that $\fD(\CB_n) \subseteq \fb_n$.
\end{proof}

\section{The Lie superalgebra of transpositions}\label{sec:type-A}

\subsection{The symmetric group} Recall from Example \ref{example:symmetric-group} that the symmetric group $\fS_n = W(A_{n-1})$ is a supergroup via the sign character $\ve'': \fS_n \to \set{\pm1}$. Its even subgroup is the alternating group $\fA_n$. The simple $\CS_n$-modules are the Specht modules $S^\lambda$ for $\lambda \vdash n$. And for each $\lambda \vdash n$ one has $S^\lambda \otimes \ve'' \cong S^{\lambda^*}$. Recalling that the $\CB_n$-module $S^{(\lambda,\emptyset)}$ is the inflation to $\calB_n$ of $S^\lambda$, and hence $S^{(\lambda,\emptyset)} = S^\lambda$ as a $\CS_n$-module, an explicit choice of associator $\phi_{\ve''}^\lambda : S^\lambda \to S^{\lambda^*}$ satisfying \eqref{eq:associator} and \eqref{eq:inverse-associator} for each $\lambda \vdash n$ can be made by taking $\phi_{\ve''}^\lambda = \phi_{\ve''}^{(\lambda,\emptyset)}$. Let $\calP(n) = \set{ \lambda : \lambda \vdash n}$ and let
	\[
	E_n = \set{ \lambda \in \calP(n): \lambda \neq \lambda^*}, \qquad
	F_n = \set{ \lambda \in \calP(n): \lambda = \lambda^*}.
	\]
Let $\sim$ be the equivalence relation on $\calP(n)$ generated by $\lambda \sim \lambda^*$. Then the simple $\CS_n$-supermodules are labeled by elements of $\calP(n)/\!\!\sim$. Now \eqref{eq:LieCG} and \eqref{eq:D(LieCG)} take the forms
	\begin{gather}
	\Lie (\CS_n) \cong 
	\Big[ \bigoplus_{\lambda \in E_n/\sim} \fq( W^\lambda ) \Big] 
	\oplus 
	\Big[ \bigoplus_{\lambda \in F_n} \gl( W^\lambda ) \Big], \label{eq:CSn-Artin-Wedderburn} \\
	\fD(\CS_n) \cong 
	\Big[ \bigoplus_{\lambda \in E_n/\sim} \sq( W^\lambda ) \Big] 
	\oplus 
	\Big[ \bigoplus_{\lambda \in F_n} \fsl( W^\lambda ) \Big], \label{eq:D(CSn)-Artin-Wedderburn}
	\end{gather}
where the sum over elements of $E_n/\!\!\sim$ means that we take the sum over any complete set of representatives in $E_n$ under the relation $\sim$.

For $n \geq 2$, the set of reflections in $\fS_n$ is
	\[
	S = \set{ (i,j) : 1 \leq i < j \leq n},
	\]
i.e., the set of transpositions in $\fS_n$. Let $\fs = \fs_n$ be the Lie subsuperalgebra of $\Lie(\CS_n)$ generated by the elements of the set $S$. The set $S$ is a single conjugacy class in $\fS_n$. Let
	\[ \textstyle
	T_n = \sum_{1 \leq i < j \leq n} (i,j)
	\]
be the sum in $\CS_n$ of the transpositions. Then
	\begin{equation} \label{eq:sn-easy-inclusion}
	\fs_n \subseteq \fD(\CS_n) + \C \cdot T_n,
	\end{equation}
by Lemma \ref{lemma:g-in-D(g)+span}. Our goal in this section is to establish the following theorem:

\begin{theorem} \label{theorem:main-theorem-Sn}
Let $n \geq 2$. Then $\fs_n = \fD(\CS_n) + \C \cdot T_n$.
\end{theorem}

\begin{lemma} \label{lemma:base-cases-Sn}
If $n \in \set{2,3,4,5}$, then $\fs_n = \fD(\CS_n) + \C \cdot T_n$.
\end{lemma}

\begin{proof}
By dimension comparison in GAP \cite{GAP4}; see Appendix \ref{app:dim-small-ranks} for more details.
\end{proof}

\begin{example}[$n=2$]
The set of reflections in $\fS_2$ is $S = \set{ (1,2)}$. One has
	\begin{gather*}
	[(1,2),(1,2)] = (1,2)(1,2) + (1,2)(1,2) = 2 \cdot e, \quad \text{and} \\
	[e,(1,2)] = e \cdot (1,2) - (1,2) \cdot e = (1,2) - (1,2) = 0.
	\end{gather*}
Then $\fs_2 = \Span\{e,(1,2) \}$, with $\szero = \Span\{e \} = \fD(\CS_2)$ and $\sone = \Span\{(1,2) \} = \C \cdot T_2$.
\end{example}

\begin{example}[$n=3$]
The set of reflections in $\fS_3$ is $S = \set{(1,2), (2,3), (1,3)}$. One finds that
	\begin{gather*}
	\szero = \Span\{e,(1,2,3)+(1,3,2)\}, \\
	\sone = \Span\{(1,2),(2,3),(1,3)\}, \\
	\fD(\CS_3) = \Span\{e,(1,2,3)+(1,3,2), (1,2)-(1,3), (1,3)-(2,3) \}, \\
	T_3 = (1,2) + (2,3) + (1,3).
	\end{gather*}
\end{example}


\begin{lemma} \label{lemma:s0-generates-An0}
For $n \geq 5$, the set
	\begin{equation} \label{eq:(i,j)(k,ell)-CC}
	S' \coloneq \set{ (i,j)(k,\ell) : \text{$i,j,k,\ell$ distinct}}
	\end{equation}
is a conjugacy class in $\fA_n$ such that $S' \subseteq \szero$ and $\fA_n = \subgrp{S'}$.
\end{lemma}

\begin{proof}
The set $S'$ is a conjugacy class in $\fS_n$. If $i,j,k,\ell$ are distinct, then $(i,j)(k,\ell)$ is centralized by $(i,j)$, which is in $\fS_n$ but not $\fA_n$. Then by the standard criterion, $S'$ remains a conjugacy class in $\fA_n$. If $i,j,k,\ell$ are distinct, then
	\[
	(i,j)(k,\ell) = \tfrac{1}{2} \cdot [(i,j),(k,\ell)] \in \szero,
	\]
so $S' \subseteq \szero$. Finally, it is well-known for $n \geq 5$ that $\fA_n$ is generated by the set $S'$.
\end{proof}

\begin{proposition} \label{prop:n-geq-5-Sn-good}
Let $n \geq 5$. Then \ref{item:S-union-conj-classes}--\ref{item:S'-one-class-gen-G0} and \ref{item:no-G-simple-dim-2} are satisfied for the supergroup $G = \fS_n$, with $S$ the set of all transpositions in $\fS_n$, and $S'$ the set \eqref{eq:(i,j)(k,ell)-CC}. Consequently, the results of Section \ref{subsec:lie-superalgebras-odd-elements} all hold in this context.
\end{proposition}

\begin{proof}
Apply Lemma \ref{lemma:s0-generates-An0} and \cite[Theorem 2.4.10]{James:1981}.
\end{proof}

\begin{lemma} \label{lemma:Sn-inductive-supergroup-ngeq6}
Let $n \geq 6$, let $G = \fS_n$, and let $S$ be the set of all transpositions in $\fS_n$. Suppose Theorem \ref{theorem:main-theorem-Sn} is true for the value $n-1$. Then \ref{item:inductive-subgroup} is satisfied by the subsupergroup $H = \fS_{n-1}$, with $S_H$ equal to the set of all transpositions in $\fS_{n-1}$.
\end{lemma}

\begin{proof}
Condition \ref{item:S-union-conj-classes} evidently holds for $H$, and \ref{item:gzero-generates-CGzero} holds by Lemma \ref{lemma:s0-generates-An0}. If Theorem \ref{theorem:main-theorem-Sn} is true for the value $n-1$, then $\fh = \fD(\CH) + \C \cdot T_{n-1}$, and hence $\hzero = \fD(\CH)_{\zero}$. Finally, the last two bullet points in \ref{item:inductive-subgroup} are satisfied by \cite[Theorem 2.5.15]{James:1981} and as a consequence of the branching rule for restriction from $\fA_n$ to $\fA_{n-1}$; see \cite[Theorem 2]{Geetha:2018} or \cite[Remark 3.4.4]{DK:2025}.
\end{proof}

\begin{proposition} \label{prop:A5-for-Sn}
Let $n \geq 6$, and suppose Theorem \ref{theorem:main-theorem-Sn} is true for the value $n-1$. Then \ref{item:D(gzero)-to-sl(V)} holds for the supergroup $G = \fS_n$.
\end{proposition}

\begin{proof}
The argument is entirely parallel to that given for the proof of Proposition \ref{prop:A5-for-B}, using the multiplicity-free branching rule from $\fA_n$ to $\fA_{n-1}$; see \cite[Theorem 2]{Geetha:2018} or \cite[Remark 3.4.4]{DK:2025}.
\end{proof}

We can now prove Theorem \ref{theorem:main-theorem-Sn}.

\begin{proof}[Proof of Theorem \ref{theorem:main-theorem-Sn}]
The proof is entirely parallel to the proof of Theorem \ref{theorem:main-theorem-B} given at the end of Section \ref{subsec:Lie-superalgebra-B}, up through the definition of the operator $\Psi$ in \eqref{eq:Psi-B}. In this case one has
	\begin{equation} \label{eq:Psi-Sn}
	\Psi = 1_{\CS_n} 
	- \Big( \sum_{\substack{W \in \Irr_s(\fS_n) \\ \dim(W) > 2}} \id_{W} \Big).
	\end{equation}
For $n \geq 6$ there is only one simple $\CS_n$-supermodule $W$ with $\dim(W) \leq 2$, namely $W = W^{[n]}$, and one sees that the operator $\Psi$ defined by \eqref{eq:Psi-Sn} is equal to $\id_{W^{[n]}}$, which spans $\sq(W^{[n]})$.
\end{proof}

\section{Type D representation theory} \label{sec:Type-D-classical}

In this section let $n \geq 4$.

\subsection{Weyl groups of type D}

The group $\calD_n = W(D_n)$ is the kernel of the group homomorphism $\ve'$ defined in \eqref{eq:Bn-linear-characters}; in the terminology of Section \ref{subsec:supergroups}, $\calD_n$ is the even subgroup of $\calB_n$ with respect to the homomorphism $\ve': \calB_n \to \set{\pm 1}$. Each element of $\calD_n$ is uniquely of the form $t_1^{i_1} \cdots t_n^{i_n} \sigma$ with $\sigma \in \fS_n$, $i_j \in \set{0,1}$ for each $j$, and $i_1 + \cdots + i_n$ even. For $1 \leq i \leq n-1$, set $u_i = t_1 t_{i+1}$. Then each element of $\calD_n$ is also uniquely of the form $u_1^{i_1} \cdots u_{n-1}^{i_{n-1}} \sigma$ with $\sigma \in \fS_n$ and $i_j \in \set{0,1}$ for each $j$.

Set $\wt{s}_n = t_n s_{n-1} t_n$. Then for $n \geq 4$, the pair $(\calD_n,\set{s_1,\ldots,s_{n-1},\wt{s}_n})$ is a Coxeter system of type $D_n$. The set of reflections in $\calD_n$ is
	\begin{equation} \label{eq:D-reflections}
	\set{ (i,j), \, t_i t_j (i,j) : 1 \leq i < j \leq n}.
	\end{equation}
The character $\ve'': \calB_n \to \set{\pm 1}$ restricts to the sign character of $\calD_n$.

\subsection{Simple modules for Weyl groups of type D} \label{subsec:-type-D-simples}

Recall from \eqref{eq:B-tensor-with-character} that for $(\lambda,\mu) \in \BPn$, one has $S^{(\lambda,\mu)} \otimes \ve' \cong S^{(\mu,\lambda)}$ as $\CB_n$-modules. Let $\BPsetn$ be the set of all unordered pairs $\set{\lambda,\mu}$ of partitions such that $\abs{\lambda} + \abs{\mu} = n$. Now applying Lemma \ref{lemma:classical-Clifford} for $(G,\kappa) = (\calB_n,\ve')$, one gets:

\begin{lemma} \label{lemma:Dn-simples}
Up to isomorphism, each simple $\CD_n$-module arises uniquely via either:
	\begin{enumerate}
	\item If $\set{\lambda,\mu} \in \BPsetn$ and $\lambda \neq \mu$, then
		\[
		S^{\set{\lambda,\mu}} \coloneq \Res^{\calB_n}_{\calD_n}( S^{(\lambda,\mu)}) \cong \Res^{\calB_n}_{\calD_n}( S^{(\mu,\lambda)})
		\]
	is simple and self-conjugate.
	
	\item If $\set{\lambda,\lambda} \in \BPsetn$, then $S^{\set{\lambda,\lambda}} \coloneq \Res^{\calB_n}_{\calD_n}(S^{(\lambda,\lambda)})$ is the direct sum of two simple, conjugate, non-isomorphic $\CD_n$-modules $S^{\set{\lambda,+}}$ and $S^{\set{\lambda,-}}$, equal to the $\pm 1$ eigen\-spaces of $\phi_{\ve'}^{(\lambda,\lambda)}$.
	\end{enumerate}
\end{lemma}

\begin{remark} \label{remark:S(lam,pm)-basis}
Suppose $\set{\lambda,\lambda} \in \BPsetn$, and set $\T(\lambda,\lambda)_+ = \{ T \in \T(\lambda,\lambda) : \rho_T(1) = +1 \}$. Then
	\[
	B^{\set{\lambda,\pm}} = \set{ c_T \pm c_{T^\natural}: T \in \T(\lambda,\lambda)_+}
	\]
is a basis for $S^{\set{\lambda,\pm}}$; cf.\ \cite[Remark 3.2.2]{DK:2025-preprint}. For $T \in \T(\lambda,\lambda)$, set $S_T^{\set{\lambda,\pm}} = \Span\{c_T \pm c_{T^\natural}\}$. Then
	\[
	S^{\set{\lambda,\pm}} = \bigoplus_{T \in \T(\lambda,\lambda)_+} S_T^{\set{\lambda,\pm}}.
	\]
\end{remark}

\begin{lemma}[{\cite[Lemma 3.2.4]{DK:2025-preprint}}] \label{lemma:Dn-subminimal-dimensions}
Let $n \geq 5$.
\begin{enumerate}
\item The only one-dimensional $\CD_n$-modules are those labeled by $\set{[n],\emptyset}$ and $\set{[1^n],\emptyset}$, i.e., the trivial $\CD_n$-module and the one-dimensional module afforded by the sign character $\ve''$.

\item If $\set{\lambda,\mu} \in \BPsetn$ and $\dim(S^{\set{\lambda,\mu}}) \neq 1$, then $\dim(S^{\set{\lambda,\mu}}) \geq n-1$.

\item If $\set{\lambda,\lambda} \in \BPsetn$, then $\dim(S^{\set{\lambda,\pm}})  = \frac{1}{2} \cdot \binom{n}{n/2} \cdot \dim(S^\lambda)^2 \geq \frac{1}{2} \cdot \binom{n}{n/2}$.
\end{enumerate}
\end{lemma}

For $\set{\lambda,\mu} \in \calBP\{n\}$, write $\set{\nu,\tau} \prec \set{\lambda,\mu}$ if either $(\nu,\tau) \prec (\lambda,\mu)$ or $(\nu,\tau) \prec (\mu,\lambda)$. Identify $\calD_{n-1}$ with the sub(super)group of $\calD_n$ generated by $\set{s_1,\ldots,s_{n-2},\wt{s}_{n-1}}$, where $\wt{s}_{n-1} = t_{n-1}s_{n-2}t_{n-1}$.

\begin{lemma}[{\cite[Lemma 3.5.2]{DK:2025-preprint}}] \label{lemma:restriction-D-to-Dn-1}
Restriction from $\calD_n$ to $\calD_{n-1}$ is multiplicity free. Specifically:
	\begin{enumerate}
	\item If $\set{\lambda,\mu} \in \BPsetn$ and $\lambda \neq \mu$, then
		\[
		\Res^{\calD_n}_{\calD_{n-1}}(S^{\set{\lambda,\mu}}) 
		\cong 
		\Big[ \bigoplus_{\substack{\set{\nu,\tau} \prec \set{\lambda,\mu} \\ \nu \neq \tau}} S^{\set{\nu,\tau}} \Big] 
		\oplus
		\Big[ \bigoplus_{\set{\nu,\nu} \prec \set{\lambda,\mu}} S^{\set{\nu,+}} \oplus S^{\set{\nu,-}} \Big].
		\]

	\item If $n$ is even and $\lambda \vdash n/2$, then
	\[
	\Res^{\calD_n}_{\calD_{n-1}}(S^{\set{\lambda,\pm}}) \cong \bigoplus_{\nu \prec \lambda} S^{\set{\nu,\lambda}}.
	\]
	\end{enumerate}
\end{lemma}

\begin{proof}
This can be deduced in precisely the same manner as Lemma \ref{lemma:Bzero-restriction}.
\end{proof}

\begin{remark} \label{remark:D-restriction-via-T-spaces}
The explicit $\CB_{n-1}$-module decomposition of $S^{(\lambda,\mu)}$ given in \eqref{eq:restriction-B-simple} restricts to a $\CD_{n-1}$-module decomposition of  $S^{\set{\lambda,\mu}}$. For $\lambda \vdash n/2$, an analogous $\CD_{n-1}$-module decomposition of $S^{\set{\lambda,\pm}}$ is given as follows; cf.\ \cite[Remark 3.5.4]{DK:2025-preprint}. Set $\T\{\nu,\tau \} = \T(\nu,\tau) \cup \T(\tau,\nu)$. Then
	\[
	S^{\set{\lambda,\pm}} = \bigoplus_{\nu \prec \lambda} \Big[ \bigoplus_{\substack{T \in \T(\lambda,\lambda)_+ \\ T/n \in \T\{\nu,\lambda\} }} S_T^{\set{\lambda,\pm}} \Big] \cong \bigoplus_{\nu \prec \lambda} S^{\set{\nu,\lambda}}.
	\]
\end{remark}

By \cite[Remark 3.2.3]{DK:2025-preprint}, one has $\CD_n$-module isomorphisms
	\begin{equation} \label{eq:S(lam,mu)-ve''-D}
	S^{\set{\lambda,\mu}} \otimes \ve'' \cong S^{\set{\lambda^*,\mu^*}} \qquad \text{and} \qquad 
	S^{\set{\lambda,\pm}} \otimes \ve'' \cong \begin{cases}
	S^{\set{\lambda^*,\pm}} & \text{if $n/2$ is even,} \\
	S^{\set{\lambda^*,\mp}} & \text{if $n/2$ is odd.}
	\end{cases}
	\end{equation}

\begin{definition} \label{def:D-associators}
Define linear maps satisfying \eqref{eq:associator} and \eqref{eq:inverse-associator} for $(G,\kappa) = (\calD_n,\ve'')$ as follows:
	\begin{itemize}
	\item If $\set{\lambda,\lambda^*} \in \BPsetn$ and $\lambda \neq \lambda^*$, let $\phi_{\ve''}^{\set{\lambda,\lambda^*}}$ be the restriction of $\phi_{\ve}^{(\lambda,\lambda^*)}: S^{(\lambda,\lambda^*)} \to S^{(\lambda,\lambda^*)}$.
	
	\item For all other $\set{\lambda,\mu} \in \BPsetn$, let $\phi_{\ve''}^{\set{\lambda,\mu}}$ be the restriction of $\phi_{\ve''}^{(\lambda,\mu)}: S^{(\lambda,\mu)} \to S^{(\lambda^*,\mu^*)}$.
	
	\item If $n$ is even and $\lambda \vdash n/2$, let
		\[
		\phi_{\ve''}^{\set{\lambda,\pm}} : \begin{cases}
		S^{\set{\lambda,\pm}} \to S^{\set{\lambda^*,\pm}} & \text{if $n/2$ is even,} \\
		S^{\set{\lambda,\pm}} \to S^{\set{\lambda^*,\mp}} & \text{if $n/2$ is odd.}
		\end{cases}
		\]
	be the restriction of $\phi_{\ve''}^{(\lambda,\lambda)}$ to the summand $S^{\set{\lambda,\pm}}$ of $S^{(\lambda,\lambda)}$.
	\end{itemize}
\end{definition}

\begin{remark} \label{remark:type-D-associators-compatible-restriction} \ 
First observe that if $\set{\lambda,\lambda^*} \in \BPsetn$ and $\lambda \neq \lambda^*$, then $\set{\nu,\tau} \neq \set{\nu^*,\tau^*}$ for all $\set{\nu,\tau} \prec \set{\lambda,\lambda^*}$. Then making the identifications of Remark \ref{remark:D-restriction-via-T-spaces}, and applying Remark \ref{remark:ve''-(lam,mu)-(lam*,mu*)}, one can see the following:
	\begin{enumerate}
	\item If $\set{\nu,\tau} \prec \set{\lambda,\mu} \in \BPsetn$ with $\lambda \neq \mu$, $\set{\lambda,\mu} = \set{\lambda^*,\mu^*}$, and $\set{\nu,\tau} = \set{\nu^*,\tau^*}$, then the map $\phi_{\ve''}^{\set{\lambda,\mu}}$ restricts via Lemma \ref{lemma:restriction-D-to-Dn-1} to $\phi_{\ve''}^{\set{\nu,\tau}}$.
	
	\item Supppose $\set{\lambda,\lambda} \in \BPsetn$ and $\nu \prec \lambda$ with $\lambda = \lambda^*$ and $\nu = \nu^*$.
		\begin{enumerate}
		\item If $n/2$ is even, then the map $\phi_{\ve''}^{\set{\lambda,\pm}}$ restricts via Lemma \ref{lemma:restriction-D-to-Dn-1} to $\phi_{\ve''}^{\set{\nu,\lambda}}$.
		
		\item If $n/2$ is odd, then $\phi_{\ve''}^{\set{\lambda,\pm}}$ restricts to a linear isomorphism from the summand $S^{\set{\nu,\lambda}}$ of $S^{\set{\lambda,\pm}}$ to the summand $S^{\set{\nu,\lambda}}$ of $S^{\set{\lambda,\mp}}$.
		\end{enumerate}
	\end{enumerate}
\end{remark}

\subsection{The even subgroup of the type D Weyl group}

We now consider $\calD_n$ as a supergroup via its sign character $\ve'': \calD_n \to \set{\pm 1}$. Making the identification $\calD_n = (\Z_2^n)_{\zero} \rtimes \fS_n$, one has
	\[
	\calD_{\zero} = (\calD_n)_{\zero} = \ker(\ve'') = (\Z_2^n)_{\zero} \rtimes \fA_n.
	\]
Let $\approx$ be the equivalence relation on $\BPsetn$ generated by $\set{\lambda,\mu} \approx \set{\lambda^*,\mu^*}$. Write $\dbl{\lambda,\mu}$ for the equivalence class of $\set{\lambda,\mu}$, and let $\BPdbln$ be the set of all equivalence classes under $\approx$. Then $\BPdbln = \Edbln \cup \Fdbln$, where
	\begin{equation} \label{eq:BP[[n]]-subsets}
	\begin{split}
	\Edbln &= \set{ \dbl{\lambda,\mu} : \set{\lambda,\mu} \neq \set{\lambda^*,\mu^*}}, \\
	\Fdbln &= \set{ \dbl{\lambda,\mu} : \set{\lambda,\mu} = \set{\lambda^*,\mu^*}}.
	\end{split}
	\end{equation}
Now applying Lemma \ref{lemma:classical-Clifford} for the pair $(G,\kappa) = (\calD_n,\ve'')$, one gets:

\begin{lemma} \label{lemma:D0-simples}
Up to isomorphism, each simple $\CDzero$-module occurs uniquely via either:
	\begin{enumerate}
	\item \label{eq:(lam,mu)-non-dual} If $\dbl{\lambda,\mu} \in \Edbln$ and $\lambda \neq \mu$, then
		\[
		S^{\dbl{\lambda,\mu}} \coloneq \Res^{\calD_n}_{\Dzero}( S^{\set{\lambda,\mu}}) \cong \Res^{\calD_n}_{\Dzero}( S^{\set{\lambda^*,\mu^*}})
		\]
	is simple and self-conjugate.

	\item \label{item:lam-non-dual} Suppose $n$ is even and $\dbl{\lambda,\lambda} \in \Edbln$.
		\begin{enumerate}
		\item \label{item:lam-non-dual-n/2-even} If $n/2$ is even, then the two $\CDzero$-modules
			\begin{align*}
			S^{\dbl{\lambda,+}} &\coloneq \Res^{\calD_n}_{\Dzero}( S^{\set{\lambda,+}}) \cong \Res^{\calD_n}_{\Dzero}( S^{\set{\lambda^*,+}}), \quad \text{and} \\
			S^{\dbl{\lambda,-}} &\coloneq \Res^{\calD_n}_{\Dzero}( S^{\set{\lambda,-}}) \cong \Res^{\calD_n}_{\Dzero}( S^{\set{\lambda^*,-}})
			\end{align*}
			are simple, self-conjugate, and non-isomorphic.

		\item \label{item:lam-non-dual-n/2-odd} If $n/2$ is odd, then the two $\CDzero$-modules
			\begin{align*}
			S^{\dbl{\lambda,+}} &\coloneq \Res^{\calD_n}_{\Dzero}( S^{\set{\lambda,+}}) \cong \Res^{\calD_n}_{\Dzero}( S^{\set{\lambda^*,-}}), \quad \text{and} \\
			S^{\dbl{\lambda,-}} &\coloneq \Res^{\calD_n}_{\Dzero}( S^{\set{\lambda,-}}) \cong \Res^{\calD_n}_{\Dzero}( S^{\set{\lambda^*,+}})
			\end{align*}
			are simple, self-conjugate, and non-isomorphic
		\end{enumerate}

	\item \label{eq:(lam,mu)-dual} If $\dbl{\lambda,\mu} \in \Fdbln$ and $\lambda \neq \mu$, then $S^{\dbl{\lambda,\mu}} \coloneq \Res^{\calD_n}_{\Dzero}(S^{\set{\lambda,\mu}})$ is the direct sum of two simple, conjugate, non-isomorphic $\CDzero$-modules $S^{\dbl{\lambda,\mu,+}}$ and $S^{\dbl{\lambda,\mu,-}}$, equal to the $+1$ and $-1$ eigen\-spaces of the map $\phi_{\ve''}^{\set{\lambda,\mu}}$ specified in Definition \ref{def:D-associators}.
	
	\item \label{eq:lam-dual} Suppose $n$ is even and $\dbl{\lambda,\lambda} \in \Fdbln$.
		\begin{enumerate}
		\item If $n/2$ is even, then $S^{\dbl{\lambda,+}} \coloneq \Res^{\calD_n}_{\Dzero}(S^{\set{\lambda,+}})$ and $S^{\dbl{\lambda,-}} \coloneq \Res^{\calD_n}_{\Dzero}(S^{\set{\lambda,-}})$ are each the direct sum of two simple, conjugate, non-isomorphic $\CDzero$-modules,
		\[
		S^{\dbl{\lambda,\pm}} = S^{\dbl{\lambda,\pm,+}} \oplus S^{\dbl{\lambda,\pm,-}},
		\]
	equal to the $+1$ and $-1$ eigen\-spaces of the maps $\phi_{\ve''}^{\set{\lambda,\pm}}$ specified in Definition \ref{def:D-associators}.

		\item If $n/2$ is odd, then
			\[
			S^{\dbl{\lambda,+}} \coloneq \Res^{\calD_n}_{\Dzero}( S^{\set{\lambda,+}}) \cong \Res^{\calD_n}_{\Dzero}( S^{\set{\lambda,-}}) \eqcolon S^{\dbl{\lambda,-}}
			\]
		is simple and self-conjugate.
		\end{enumerate}
	\end{enumerate}
\end{lemma}

\begin{remark} \label{remark:n/2-odd-abuse}
For $\dbl{\lambda,\lambda} \in \Edbln$ and $n/2$ odd, the isomorphism classes of $S^{\dbl{\lambda,+}}$ and $S^{\dbl{\lambda,-}}$ depend individually on the choice of representative $\set{\lambda,\lambda}$ for the equivalence class $\dbl{\lambda,\lambda}$, but as a \emph{pair} they are independent of the choice.
\end{remark}

\begin{remark} \label{remark:S-dbl(lam,lam)}
Suppose $n$ even and $\lambda \vdash n/2$. Set $S^{\dbl{\lambda,\lambda}} = \Res^{\calD_n}_{\Dzero}(S^{\set{\lambda,\lambda}}) = S^{\dbl{\lambda,+}} \oplus S^{\dbl{\lambda,-}}$.
	\begin{itemize}
	\item If $\lambda \neq \lambda^*$, then $S^{\dbl{\lambda,\lambda}}$ is the sum of two non-isomorphic simple $\CDzero$-modules.

	\item If $\lambda = \lambda^*$ and $n/2$ is even, then $S^{\dbl{\lambda,\lambda}} = S^{\dbl{\lambda,+,+}} \oplus S^{\dbl{\lambda,+,-}} \oplus S^{\dbl{\lambda,-,+}} \oplus S^{\dbl{\lambda,-,-}}$ is the sum of four non-isomorphic simple $\CDzero$-modules.

	\item If $\lambda = \lambda^*$ and $n/2$ is odd, then $S^{\dbl{\lambda,\lambda}}$ is the sum of two isomorphic simple $\CDzero$-modules.
	\end{itemize}
Let $\phi_{\ve'} = \phi_{\ve'}^{(\lambda,\lambda)}$ and $\phi_{\ve''} = \phi_{\ve''}^{(\lambda,\lambda)}$. Then $\phi_{\ve''} \circ \phi_{\ve'} = (-1)^{n/2} \cdot \phi_{\ve'} \circ \phi_{\ve''}$ by Remark \ref{remark:phi-ve''-ve'-commutation}. Let $S^{\dbl{\lambda,\lambda,\pm}}$ be the $\pm1$ eigenspace for $\phi_{\ve''}$. Then it is straightforward to see, as subspaces of $S^{(\lambda,\lambda)}$, that
	\[
	S^{\dbl{\lambda,\lambda,\pm}} = S^{\dbl{\lambda,+,\pm}} \oplus S^{\dbl{\lambda,-,\pm}} \quad \text{if $n/2$ is even.}
	\]
On the other hand, the decompositions $S^{\dbl{\lambda,\lambda,+}} \oplus S^{\dbl{\lambda,\lambda,-}}$ and $S^{\dbl{\lambda,+}} \oplus S^{\dbl{\lambda,-}}$ of $S^{\dbl{\lambda,\lambda}}$ imply that
	\[
	S^{\dbl{\lambda,\lambda,+}} \cong S^{\dbl{\lambda,+}} \cong S^{\dbl{\lambda,-}} \cong S^{\dbl{\lambda,\lambda,-}} \quad \text{as $\CDzero$-modules if $n/2$ is odd.}
	\]
\end{remark}

\begin{lemma} \label{lemma:Dzero-subminimal-dimensions}
Let $n \geq 5$.
\begin{enumerate}
\item The only one-dimensional $\CDzero$-module is the trivial module $S^{\dbl{[n],\emptyset}}$.

\item If $\dbl{\lambda,\mu} \in \BPdbln$ and $\dim(S^{\dbl{\lambda,\mu}}) \neq 1$, then $\dim(S^{\dbl{\lambda,\mu}}) \geq n-1$.

\item If $\dbl{\lambda,\mu} \in \Fdbln$ and $\lambda \neq \mu$, then $\dim(S^{\dbl{\lambda,\mu,\pm}}) = \frac{1}{2} \cdot \dim(S^{\dbl{\lambda,\mu}}) \geq \frac{1}{2} \cdot (n-1)$.

\item If $\dbl{\lambda,\lambda} \in \BPdbln$, then $\dim(S^{\dbl{\lambda,\pm}}) = \frac{1}{2} \cdot \binom{n}{n/2} \cdot \dim(S^\lambda)^2 \geq \frac{1}{2} \cdot \binom{n}{n/2}$.

\item If $\dbl{\lambda,\lambda} \in \Fdbln$ and $n/2$ is even, then
	\[ \textstyle
	\dim(S^{\dbl{\lambda,\pm,+}}) = \dim(S^{\dbl{\lambda,\pm,-}}) = \frac{1}{4} \cdot \binom{n}{n/2} \cdot \dim(S^\lambda)^2 \geq \frac{1}{4} \cdot \binom{n}{n/2}.
	\]
\end{enumerate}
\end{lemma}


\begin{remark} \label{remark:Dzero-subminimal-dimensions-n=5}
There are no elements in $\calBP\dbl{5}$ of the form $\dbl{\lambda,\lambda}$, and the only elements of the set $F\dbl{5}$ are $\dbl{[3,1,1],\emptyset}$ and $\dbl{[2,2],[1]}$. Applying Lemma \ref{lemma:Bn-subminimal-dimensions}\eqref{item:dim-S(lam,mu)}, one checks that
	\begin{gather*} \textstyle
	\dim(S^{\dbl{[3,1,1],\emptyset}})  = \binom{5}{5} \cdot \dim(S^{[3,1,1]}) \cdot \dim(S^{\emptyset}) = 1 \cdot \frac{5!}{1 \cdot 2 \cdot 5 \cdot 2 \cdot 1} \cdot 1 = 6, \\
	\textstyle \dim(S^{\dbl{[2,2],[1]}}) = \binom{5}{4} \cdot \dim(S^{[2,2]}) \cdot \dim(S^{[1]}) = 5 \cdot \frac{4!}{1 \cdot 2 \cdot 3 \cdot 2} \cdot 1 = 10.
	\end{gather*}
It follows for $n = 5$ that if $V$ is a simple $\CDzero$-module and $\dim(V) \neq 1$, then $\dim(V) \geq 3$.
\end{remark}

\begin{remark} \label{rem:commutator-subgroups}
The commutator subgroup $[G,G]$ of a finite group $G$ is equal to the intersection of the kernels of the degree-$1$ complex linear characters of $G$. Then Lemmas \ref{lemma:Bn-subminimal-dimensions}, \ref{lemma:Dn-subminimal-dimensions}, \ref{lemma:Bzero-subminimal-dimensions}, and \ref{lemma:Dzero-subminimal-dimensions} imply for $n \geq 5$ that $[\calB_n,\calB_n] = [\calD_n,\calD_n] = [\Bzero,\Bzero] = [\Dzero,\Dzero] = \Dzero$.
\end{remark}

\begin{lemma} \label{lemma:Dzero-restriction} \ 
	\begin{enumerate}
	\item \label{item:(lam,mu)-Eboxn} If $\dbl{\lambda,\mu} \in \Edbln$ and $\lambda \neq \mu$, then
		\[
		\Res^{(\calD_n)_{\zero}}_{(\calD_{n-1})_{\zero}}(S^{\dbl{\lambda,\mu}}) 
		\cong 
		\bigoplus_{\set{\nu,\tau} \prec \set{\lambda,\mu} } S^{\dbl{\nu,\tau}},
		\]
	with the understanding that $S^{\dbl{\nu,\tau}}$ decomposes into a sum of two or more simple modules if either $\set{\nu,\tau} = \set{\nu^*,\tau^*}$ or $\nu = \tau$; cf.\ Lemma \ref{lemma:D0-simples}\eqref{eq:(lam,mu)-dual} and Remark \ref{remark:S-dbl(lam,lam)}.
	
	\item If $n$ is even and $\dbl{\lambda,\lambda} \in \Edbln$, then
		\[
		\Res^{(\calD_n)_{\zero}}_{(\calD_{n-1})_{\zero}}(S^{\dbl{\lambda,\pm}}) 
		\cong 
		\bigoplus_{ \nu \prec \lambda } S^{\dbl{\nu,\lambda}}.
		\]
	
	\item Suppose $\dbl{\lambda,\mu} \in \Fdbln$ and $\lambda \neq \mu$.
		\begin{enumerate}
		\item If $\mu = \lambda^*$, so $\dbl{\lambda,\mu} = \dbl{\lambda,\lambda^*}$ and $\lambda \neq \lambda^*$, then
			\[
			\Res^{(\calD_n)_{\zero}}_{(\calD_{n-1})_{\zero}}(S^{\dbl{\lambda,\lambda^*,\pm}}) 
			\cong 
			\bigoplus_{ \nu \prec \lambda } S^{\dbl{\nu,\lambda^*}}.
			\]
		
		\item If $\mu \neq \lambda^*$, so $\lambda = \lambda^*$ and $\mu = \mu^*$, then
			\begin{align*}
			\Res^{(\calD_n)_{\zero}}_{(\calD_{n-1})_{\zero}}(S^{\dbl{\lambda,\mu,\pm}}) 
			&\cong 
			\Big[ \bigoplus_{\substack{\nu \prec \lambda \\ \res(\lambda/\nu)> 0}} S^{\dbl{\nu,\mu}} \Big] 
			\oplus
			\Big[ \bigoplus_{\substack{\nu \prec \lambda \\ \res(\lambda/\nu) = 0}} S^{\dbl{\nu,\mu,\pm}} \Big] \\
			&\phantom{\cong} \oplus \Big[ \bigoplus_{\substack{\tau \prec \mu \\ \res(\mu/\tau)> 0}} S^{\dbl{\lambda,\tau}} \Big] 
			\oplus
			\Big[ \bigoplus_{\substack{\tau \prec \mu \\ \res(\mu/\tau) = 0}} S^{\dbl{\lambda,\tau,\pm}} \Big],
			\end{align*}
		with the understanding from Remark \ref{remark:S-dbl(lam,lam)} that if $\mu \prec \lambda$, then $S^{\dbl{\mu,\mu,\pm}} = S^{\dbl{\mu,+,\pm}} \oplus S^{\dbl{\mu,-,\pm}}$ if $\abs{\mu}$ is even, and $S^{\dbl{\mu,\mu,\pm}} \cong S^{\dbl{\mu,\pm}}$ if $\abs{\mu}$ is odd (and similarly if $\lambda \prec \mu$).
		\end{enumerate}
	
	\item Suppose $n$ is even and $\dbl{\lambda,\lambda} \in \Fdbln$.
		\begin{enumerate}
		\item If $n/2$ is even, then for $\heartsuit \in \set{+,-}$,
			\[
			\Res^{(\calD_n)_{\zero}}_{(\calD_{n-1})_{\zero}}(S^{\dbl{\lambda,\pm,\heartsuit}}) 
			\cong 
			\Big[ \bigoplus_{\substack{\nu \prec \lambda \\ \res(\lambda/\nu) > 0}} S^{\dbl{\nu,\lambda}} \Big] 
			\oplus 
			\Big[ \bigoplus_{\substack{\nu \prec \lambda \\ \res(\lambda/\nu) = 0}} S^{\dbl{\nu,\lambda,+}} \oplus S^{\dbl{\nu,\lambda,-}} \Big].
			\]

		\item \label{item:(lam,lam)-Fdbln-n/2-odd} If $n/2$ is odd, then
			\[
			\Res^{(\calD_n)_{\zero}}_{(\calD_{n-1})_{\zero}}(S^{\dbl{\lambda,\pm}}) 
			\cong 
			\bigoplus_{\nu \prec \lambda} S^{\dbl{\nu,\lambda}},
			\]
		with the understanding from Lemma \ref{lemma:D0-simples}\eqref{eq:(lam,mu)-dual} that if $\nu = \nu^*$, then $S^{\dbl{\nu,\lambda}}$ is the direct sum of the two simple, non-isomorphic modules $S^{\dbl{\nu,\lambda,+}}$ and $S^{\dbl{\nu,\lambda,-}}$.
		
		\end{enumerate}
	
	\end{enumerate}
\end{lemma}

\begin{proof}
Apply definitions, Lemma \ref{lemma:restriction-D-to-Dn-1}, and Remark \ref{remark:type-D-associators-compatible-restriction}.
\end{proof}

\begin{remark} \label{remark:D-repeated-factors}
In general, restriction from $(\calD_n)_{\zero}$ to $(\calD_{n-1})_{\zero}$ is not multiplicity free. Repeated composition factors occur in the following situations:
	\begin{enumerate}
	\item \label{item:(lam-mu)-repeated} In case \eqref{item:(lam,mu)-Eboxn} of the lemma, if $\mu \prec \lambda$ with $\abs{\mu}$ odd and $\mu = \mu^*$ (and hence $\lambda \neq \lambda^*$), then the summand $S^{\dbl{\mu,\mu}}$ of $S^{\dbl{\lambda,\mu}}$ is the sum of the isomorphic modules $S^{\dbl{\mu,+}}$ and $S^{\dbl{\mu,-}}$. Similarly, if $\lambda \prec \mu$ with $\abs{\lambda}$ odd and $\lambda = \lambda^*$, then the summand $S^{\dbl{\lambda,\lambda}}$ is the sum of the isomorphic modules $S^{\dbl{\lambda,+}}$ and $S^{\dbl{\lambda,-}}$.
	
	\item In case \eqref{item:(lam,lam)-Fdbln-n/2-odd} of the lemma, if $\nu \prec \lambda$ with $\res(\lambda/\nu)> 0$, then also $\nu^* \prec \lambda$, and the summands $S^{\dbl{\nu,\lambda}}$ and $S^{\dbl{\nu^*,\lambda}}$ of $S^{\dbl{\lambda,\pm}}$ are isomorphic.
	\end{enumerate}
\end{remark}

\subsection{Simple supermodules in Type D} \label{subsec:simple-supermodules-D}

Applying Lemma \ref{prop:simple-supermodules} for the pair $(G,\kappa) = (\calD_n,\ve'')$, and fixing the maps $\phi_{\ve''}^{\set{\lambda,\mu}}$ and $\phi_{\ve''}^{\set{\lambda,\pm}}$ as in Definition \ref{def:D-associators}, one now gets:

\begin{proposition} \label{prop:D-simple-supermodules}
Up to homogeneous isomorphism, each simple $\CD_n$-super\-module occurs in exactly one of the following ways:
	\begin{enumerate}
	\item For each $\dbl{\lambda,\mu} \in \Edbln$ with $\lambda \neq \mu$, there exists a Type Q simple $\CD_n$-supermodule
		\[
		U^{\dbl{\lambda,\mu}} = S^{\set{\lambda,\mu}} \oplus S^{\set{\lambda^*,\mu^*}},
		\]
	with homogeneous subspaces and odd involution $J^{\set{\lambda,\mu}}$ defined as in Lemma \ref{prop:simple-supermodules}(\ref{item:self-associate-G-supermodule}).
	
	\item Suppose $n$ is even and $\dbl{\lambda,\lambda} \in \Edbln$.
		\begin{enumerate}
		\item \label{item:lam-non-dual-n/2-even-super} If $n/2$ is even, then there exist Type Q simple $\CD_n$-supermodules
			\[
			U^{\dbl{\lambda,+}} = S^{\set{\lambda,+}} \oplus S^{\set{\lambda^*,+}} 
			\quad \text{and} \quad 
			U^{\dbl{\lambda,-}} = S^{\set{\lambda,-}} \oplus S^{\set{\lambda^*,-}},
			\]
		with homogeneous subspaces and odd involutions $J^{\set{\lambda,\pm}}$ defined as in Lemma \ref{prop:simple-supermodules}(\ref{item:self-associate-G-supermodule}).
		
		\item \label{item:lam-non-dual-n/2-odd-super} If $n/2$ is odd, then there exist Type Q simple $\CD_n$-supermodules
			\[
			U^{\dbl{\lambda,+}} = S^{\set{\lambda,+}} \oplus S^{\set{\lambda^*,-}} 
			\quad \text{and} \quad 
			U^{\dbl{\lambda,-}} = S^{\set{\lambda,-}} \oplus S^{\set{\lambda^*,+}},
			\]
		with homogeneous subspaces and odd involutions $J^{\set{\lambda,\pm}}$ defined as in Lemma \ref{prop:simple-supermodules}(\ref{item:self-associate-G-supermodule}).
		
		\end{enumerate}

	\item For each $\dbl{\lambda,\mu} \in \Fdbln$ with $\lambda \neq \mu$, there exists a Type M simple $\CD_n$-supermodule
		\[
		U^{\dbl{\lambda,\mu}} = S^{\set{\lambda,\mu}}
		\]
		with $U^{\dbl{\lambda,\mu}}_{\zero} = S^{\dbl{\lambda,\mu,+}}$ and $U^{\dbl{\lambda,\mu}}_{\one} = S^{\dbl{\lambda,\mu,-}}$.
	
	\item Suppose $n$ is even and $\dbl{\lambda,\lambda} \in \Fdbln$.
		\begin{enumerate}
		\item If $n/2$ is even, then there exist Type M simple $\CD_n$-supermodules
			\[
			U^{\dbl{\lambda,+}} = S^{\set{\lambda,+}} 
			\quad \text{and} \quad 
			U^{\dbl{\lambda,-}} = S^{\set{\lambda,-}},
			\]
		with $U^{\dbl{\lambda,\pm}}_{\zero} = S^{\dbl{\lambda,\pm,+}}$ and $U^{\dbl{\lambda,\pm}}_{\one} = S^{\dbl{\lambda,\pm,-}}$.
		
		\item If $n/2$ is odd, then there exists a Type Q simple $\CD_n$-supermodule
			\[
			U^{\dbl{\lambda,\lambda}} = S^{\set{\lambda,+}} \oplus S^{\set{\lambda,-}}
			\]
		with homogeneous subspaces and odd involution $J^{\set{\lambda,+}}$ defined as in Lemma \ref{prop:simple-supermodules}(\ref{item:self-associate-G-supermodule}).
		
		\end{enumerate}
	\end{enumerate}
Each simple $\CD_n$-super\-module is uniquely determined, up to an even isomorphism, by its summands as a $\CDzero$-module, and by the superdegrees in which those summands are concentrated.
\end{proposition}

\begin{remark}
We are abusing notation in case \eqref{item:lam-non-dual-n/2-odd-super} in the same manner discussed in Remark \ref{remark:n/2-odd-abuse}. If $\dbl{\lambda,\lambda} \in \Edbln$, or if $\dbl{\lambda,\lambda} \in \Fdbln$ with $n/2$ even, we may write
	\[
	U^{\dbl{\lambda,\lambda}} = U^{\dbl{\lambda,+}} \oplus U^{\dbl{\lambda,-}}.
	\]
Then $U^{\dbl{\lambda,\lambda}}$ depends only on the equivalence class $\dbl{\lambda,\lambda}$.
\end{remark}

\begin{lemma} \label{Dn-super-subminimal-dimensions}
Let $n \geq 5$.
	\begin{enumerate}
	\item Let $\dbl{\lambda,\mu} \in \Edbln$ with $\lambda \neq \mu$. If $\dbl{\lambda,\mu} = \dbl{[n],\emptyset}$, then $\dim(U^{\dbl{\lambda,\mu}}) = 2$. Otherwise,
		\[
		\dim(U^{\dbl{\lambda,\mu}}) \geq 2n-2.
		\]

	\item If $\dbl{\lambda,\lambda} \in \Edbln$, then $\dim(U^{\dbl{\lambda,\pm}}) = \binom{n}{n/2} \cdot \dim(S^\lambda)^2 \geq \binom{n}{n/2}$.
	
	\item If $\dbl{\lambda,\mu} \in \Fdbln$ and $\lambda \neq \mu$, then $\dim(U^{\dbl{\lambda,\mu}}) \geq n-1$.
	
	\item If $\dbl{\lambda,\lambda} \in \Fdbln$ and $n/2$ is even, then $\dim(U^{\dbl{\lambda,\pm}}) = \frac{1}{2} \cdot \binom{n}{n/2} \cdot \dim(S^\lambda)^2 \geq \frac{1}{2} \cdot \binom{n}{n/2}$.

	\item If $\dbl{\lambda,\lambda} \in \Fdbln$ and $n/2$ is odd, then $\dim(U^{\dbl{\lambda,\lambda}}) = \binom{n}{n/2} \cdot \dim(S^\lambda)^2 \geq \binom{n}{n/2}$.

	\end{enumerate}
\end{lemma}

\begin{proposition} \label{prop:Dn-super-restriction}
Let $n \geq 5$.
	\begin{enumerate}
	\item \label{item:(lam,mu)-Edbln-super} Suppose $\dbl{\lambda,\mu} \in \Edbln$ with $\lambda \neq \mu$. Then
		\[
		\Res^{\calD_n}_{\calD_{n-1}}(U^{\dbl{\lambda,\mu}}) \cong 
		\Big[ \bigoplus_{\substack{\set{\nu,\tau} \prec \set{\lambda,\mu} \\ \set{\nu,\tau} \neq \set{\nu^*,\tau^*}}} U^{\dbl{\nu,\tau}} \Big] \oplus \Big[ \bigoplus_{\substack{\set{\nu,\tau} \prec \set{\lambda,\mu} \\ \set{\nu,\tau} = \set{\nu^*,\tau^*}}} U^{\dbl{\nu,\tau}} \oplus \Pi(U^{\dbl{\nu,\tau}}) \Big].
		\]
	
	\item Suppose $\dbl{\lambda,\lambda} \in \Edbln$. Then $\dbl{\nu,\lambda} \in E\dbl{n-1}$ for all $\nu \prec \lambda$, and
		\begin{align*}
		\Res^{\calD_n}_{\calD_{n-1}}(U^{\dbl{\lambda,\pm}}) &\cong \bigoplus_{\nu \prec \lambda} U^{\dbl{\nu,\lambda}}.
		\end{align*}

	\item \label{item:Dn-super-restriction-Fn'} Suppose $\dbl{\lambda,\mu} \in \Fdbln$ with $\lambda \neq \mu$.
		\begin{enumerate}
		\item Suppose $\dbl{\lambda,\mu} = \dbl{\lambda,\lambda^*}$ and $\lambda \neq \lambda^*$. Then $\dbl{\nu,\lambda^*} \in E\dbl{n-1}$ for all $\nu \prec \lambda$, and
			\begin{align*}
			\Res^{\calD_n}_{\calD_{n-1}}(U^{\dbl{\lambda,\lambda^*}}) \cong 
			\bigoplus_{\nu \prec \lambda} U^{\dbl{\nu,\lambda^*}}.
			\end{align*}

		\item \label{item:Dn-super-restriction-Fn'-case-a} Suppose $\lambda = \lambda^*$ and $\mu = \mu^*$. Then
			\[
			\Res^{\calD_n}_{\calD_{n-1}}(U^{\dbl{\lambda,\mu}}) 
			\cong 
			\Big[ \bigoplus_{\substack{\nu \prec \lambda \\ \res(\lambda/\nu) \geq 0}} U^{\dbl{\nu,\mu}} \Big] 
			\oplus 
			\Big[ \bigoplus_{\substack{\tau \prec \mu \\ \res(\mu/\tau) \geq 0 }} U^{\dbl{\lambda,\tau}} \Big].
			\]
		
		\end{enumerate}

	\item \label{item:Dn-super-restriction-Fn''} Let $\dbl{\lambda,\lambda} \in \Fdbln$.
		\begin{enumerate}
		\item Suppose $n/2$ is even. Then
			\[
			\Res^{\calD_n}_{\calD_{n-1}}(U^{\dbl{\lambda,\pm}}) \cong 
			\bigoplus_{\substack{\nu \prec \lambda \\ \cont(\lambda/\nu) \geq 0 }} U^{\dbl{\nu,\lambda}}.
			\]
		
		\item Suppose $n/2$ is odd. Then
			\begin{align*}
			\Res^{\calD_n}_{\calD_{n-1}}(U^{\dbl{\lambda,\lambda}}) &\cong 
			\Big[ \bigoplus_{\substack{\nu \prec \lambda \\ \cont(\lambda/\nu) > 0}} U^{\dbl{\nu,\lambda}} \Big]^{\oplus 2} 
			\oplus \Big[ \bigoplus_{\substack{\nu \prec \lambda \\ \cont(\lambda/\nu) = 0 }} U^{\dbl{\nu,\lambda}} \oplus \Pi(U^{\dbl{\nu,\lambda}}) \Big].
			\end{align*}
		
		\end{enumerate}
	\end{enumerate}
\end{proposition}

\begin{proof}
Apply Lemma \ref{lemma:Dzero-restriction} and reasoning like that used to establish Proposition \ref{prop:B-to-Bn-1-supermodule-restriction}, noting in case \eqref{item:(lam,mu)-Edbln-super} that if $\set{\nu,\nu} \prec \set{\lambda,\mu}$ with $\nu = \nu^*$ and $\abs{\nu}$ odd, then the parity change functor applied to $U^{\dbl{\nu,\nu}}$ is unnecessary but harmless because $U^{\dbl{\nu,\nu}}$ is of Type Q (and hence is isomorphic via an even isomorphism to its parity change module). Similar comments apply to other instances of Type Q simple supermodules.
\end{proof}

\subsection{The Lie superalgebra of reflections in type D} \label{subsec:Lie-superalgebra-D}

Let $\fd_n \subseteq \Lie(\CD_n)$ be the Lie super\-algebra generated by the set \eqref{eq:D-reflections} of all reflections in $\calD_n$. When the value of $n$ is clear from the context, we may write $\fdzero = (\fd_n)_{\zero}$. In this context, \eqref{eq:LieCG} and \eqref{eq:D(LieCG)} take the forms
	\begin{multline} \label{eq:CDn-Artin-Wedderburn}
	\Lie (\CD_n) \cong 
	\Big[ \bigoplus_{\substack{\dbl{\lambda,\mu} \in \Edbln \\ \lambda \neq \mu}} \fq( U^{\dbl{\lambda,\mu}} ) \Big] 
	\oplus 
	\Big[ \bigoplus_{\dbl{\lambda,\lambda} \in \Edbln} \fq( U^{\dbl{\lambda,+}} ) \oplus \fq( U^{\dbl{\lambda,-}} ) \Big] \\
	\oplus 
	\Big[ \bigoplus_{\substack{\dbl{\lambda,\mu} \in \Fdbln \\ \lambda \neq \mu}} \gl( U^{\dbl{\lambda,\mu}} ) \Big]
	\oplus 
	\Big[ \bigoplus_{\substack{\dbl{\lambda,\lambda} \in \Fdbln \\ \text{$n/2$ even}}} \gl( U^{\dbl{\lambda,+}} ) \oplus \gl( U^{\dbl{\lambda,-}} ) \Big] 
	\oplus
	\Big[ \bigoplus_{\substack{\dbl{\lambda,\lambda} \in \Fdbln \\ \text{$n/2$ odd}}} \fq( U^{\dbl{\lambda,\lambda}} ) \Big]
	\end{multline}

	\begin{multline} \label{eq:D(CDn)-Artin-Wedderburn}
	\fD(\CD_n) \cong 
	\Big[ \bigoplus_{\substack{\dbl{\lambda,\mu} \in \Edbln \\ \lambda \neq \mu}} \sq( U^{\dbl{\lambda,\mu}} ) \Big] 
	\oplus 
	\Big[ \bigoplus_{\dbl{\lambda,\lambda} \in \Edbln} \sq( U^{\dbl{\lambda,+}} ) \oplus \sq( U^{\dbl{\lambda,-}} ) \Big] \\
	\oplus 
	\Big[ \bigoplus_{\substack{\dbl{\lambda,\mu} \in \Fdbln \\ \lambda \neq \mu}} \fsl( U^{\dbl{\lambda,\mu}} ) \Big]
	\oplus 
	\Big[ \bigoplus_{\substack{\dbl{\lambda,\lambda} \in \Fdbln \\ \text{$n/2$ even}}} \fsl( U^{\dbl{\lambda,+}} ) \oplus \fsl( U^{\dbl{\lambda,-}} ) \Big] 
	\oplus
	\Big[ \bigoplus_{\substack{\dbl{\lambda,\lambda} \in \Fdbln \\ \text{$n/2$ odd}}} \sq( U^{\dbl{\lambda,\lambda}} ) \Big]
	\end{multline}

Let $\calX_n = \sum_{j=1}^n X_j = \sum_{j=1}^n \sum_{i=1}^{j-1} [(i,j) + t_it_j(i,j)]$ be the sum in $\CD_n$ of the YJM elements. The set \eqref{eq:D-reflections} is a single conjugacy classes in $\calD_n$. Then
	\begin{equation} \label{eq:dn-easy-inclusion}
	\fd_n \subseteq \fD(\CD_n) + \C \cdot \calX_n,
	\end{equation}
by Lemma \ref{lemma:g-in-D(g)+span}. Our goal in this section is to prove the following theorem:

\begin{theorem} \label{theorem:main-theorem-D}
Let $n \geq 4$. Then $\fd_n = \fD(\CD_n) + \C \cdot \calX_n$.
\end{theorem}

\begin{lemma} \label{lemma:base-cases-D}
If $n \in \set{4,5}$, then $\fd_n = \fD(\CD_n) + \C \cdot \calX_n$.
\end{lemma}

\begin{proof}
By dimension comparison in GAP \cite{GAP4}; see Appendix \ref{app:dim-small-ranks} for more details.
\end{proof}

\begin{lemma} \label{lemma:d0-generates-CD0}
For $n \geq 5$, the set
	\begin{equation} \label{eq:(ij)(kl)-CC}
	C_{(2^21^{n-4},\emptyset)} = \set{ (i,j)(k,\ell),\, t_it_j (i,j)(k,\ell), \, t_it_jt_kt_\ell (i,j)(k,\ell) : \text{$i,j,k,\ell$ distinct}}
	\end{equation}
is a conjugacy class $S'$ in $\Dzero$ such that $S' \subseteq \fdzero$ and $\Dzero = \subgrp{S'}$.
\end{lemma}

\begin{proof}
The set $S' = C_{(2^21^{n-4},\emptyset)}$ is a conjugacy class in $\calB_n$; the labeling corresponds to that in \cite[\S2]{Okada:1990} or \cite[\S2]{Mishra:2016}. If $i,j,k,\ell,m$ are distinct, then the element $t_it_j(i,j)(k,\ell) \in S'$ is centralized by $t_m$, which is an element of $\calB_n$ but not $\calD_n$. Then by the standard criterion, $S'$ remains a conjugacy class in $\calD_n$ (rather than splitting into a union of two conjugacy classes in $\calD_n$). Similarly, $t_it_j(i,j)(k,\ell)$ is centralized by $(k,\ell)$, which is an element of $\calD_n$ but not $\Dzero$, so $S'$ is a conjugacy class in $\Dzero$.

If $i,j,k,\ell$ are distinct, then
	\begin{gather*}
	(i,j)(k,\ell) = \tfrac{1}{2} \cdot [(i,j),(k,\ell)] \in \fdzero, \\
	t_i t_j (i,j)(k,\ell) = \tfrac{1}{2} \cdot [t_it_j(i,j),(k,\ell)] \in \fdzero, \\
	t_it_jt_kt_\ell (i,j)(k,\ell) = \tfrac{1}{2} \cdot [t_it_j(i,j),t_kt_\ell (k,\ell)] \in \fdzero,
	\end{gather*}
so $S' \subseteq \fdzero$. The group $\Dzero$ is generated as group by all elements of the form
	\[
	[t_it_j(i,j)(k,\ell)] \cdot [(i,j)(k,\ell)] = t_i t_j
	\]
for $i,j,k,\ell$ distinct, and by the alternating group $\fA_n$, which for $n \geq 5$ is generated by all elements of the form $(i,j)(k,\ell)$ for $i,j,k,\ell$ distinct. Then for $n \geq 5$, $\Dzero$ is generated by the set $S'$.
\end{proof}

\begin{proposition} \label{prop:n-geq-5-Dn-good}
Let $n \geq 5$. Then \ref{item:S-union-conj-classes}--\ref{item:S'-one-class-gen-G0} and \ref{item:no-G-simple-dim-2} are satisfied for the supergroup $G = \calD_n$, with $S$ the set \eqref{eq:D-reflections} of all reflections in $\calD_n$, and $S'$ the set \eqref{eq:(ij)(kl)-CC}. Consequently, the results of Section \ref{subsec:lie-superalgebras-odd-elements} all hold in this context.
\end{proposition}

\begin{proof}
Apply Lemmas \ref{lemma:d0-generates-CD0} and \ref{lemma:Dn-subminimal-dimensions}.
\end{proof}

\begin{lemma} \label{lemma:D-inductive-supergroup-ngeq6}
Let $n \geq 6$, let $G = \calD_n$, and let $S$ be the set \eqref{eq:D-reflections} of all reflections in $\calD_n$. Suppose Theorem \ref{theorem:main-theorem-D} is true for the value $n-1$. Then \ref{item:inductive-subgroup} is satisfied by the subsupergroup $H = \calD_{n-1}$, with $S_H$ equal to the set of all reflections in $\calD_{n-1}$.
\end{lemma}

\begin{proof}
Condition \ref{item:S-union-conj-classes} evidently holds for $H$, and \ref{item:gzero-generates-CGzero} holds by Lemma \ref{lemma:d0-generates-CD0}. If Theorem \ref{theorem:main-theorem-D} is true for the value $n-1$, then $\fh = \fD(\CH) + \C \cdot \calX_{n-1}$, and hence $\hzero = \fD(\CH)_{\zero}$. Finally, the last two bullet points in \ref{item:inductive-subgroup} are satisfied by Lemmas \ref{lemma:Dzero-subminimal-dimensions} and \ref{lemma:Dzero-restriction}.
\end{proof}

\begin{proposition} \label{prop:A5-for-D}
Let $n \geq 6$, and suppose Theorem \ref{theorem:main-theorem-D} is true for the value $n-1$. Then \ref{item:D(gzero)-to-sl(V)} holds for the supergroup $G = \calD_n$.
\end{proposition}

\begin{proof}
We defer the proof to Section \ref{subsec:A5-for-D}.
\end{proof}

We can now prove Theorem \ref{theorem:main-theorem-D}.

\begin{proof}[Proof of Theorem \ref{theorem:main-theorem-D}]
The proof is entirely parallel to the proof of Theorem \ref{theorem:main-theorem-B} given at the end of Section \ref{subsec:Lie-superalgebra-B} up through the definition of the operator $\Psi$ in \eqref{eq:Psi-B}. In this case one has
	\begin{equation} \label{eq:Psi-D}
	\Psi = 1_{\CD_n} 
	- \Big( \sum_{\substack{W \in \Irr_s(\calD_n) \\ \dim(W) > 2}} \id_{W} \Big).
	\end{equation}
For $n \geq 6$ there is only one simple $\CD_n$-supermodule $W$ with $\dim(W) \leq 2$, namely $W = W^{\dbl{[n],\emptyset}}$, and one sees that the operator $\Psi$ defined by \eqref{eq:Psi-D} is equal to $\id_{W^{\dbl{[n],\emptyset}}}$, which spans $\sq(W^{\dbl{[n],\emptyset}})$.
\end{proof}

\subsection{Proof of Proposition \ref{prop:A5-for-D}} \label{subsec:A5-for-D}

Let $n \geq 6$, and suppose Theorem \ref{theorem:main-theorem-D} is true for the value $n-1$. Set $G = \calD_n$ and $\g = \fd_n$. Then $\fD(\gzero)$ is semisimple by Proposition \ref{prop:n-geq-5-Dn-good}. Let $V \in \Irr(\Gzero)$, let $\rho_V: \fD(\gzero) \to \End(V)$ be the restriction to $\fD(\gzero)$ of the module structure map $V: \CGzero \to \End(V)$, and let $\g^V = \im(\rho_V) \subseteq \fsl(V)$. The map $\rho_V$ is evidently a surjection if $\dim(V) = 1$, since then $\fsl(V) = 0$, so assume that $\dim(V) > 1$. As in the proof of Lemma \ref{lemma:D(gzero)}, we may identify $\g^V$ with the orthogonal complement of $\g_V \coloneq \ker(\rho_V)$ with respect to the Killing form on $\fD(\gzero)$. Then $\fD(\gzero) = \g_V \oplus \g^V$, and $\rho_V: \fD(\gzero) \to \fsl(V)$ identifies with the projection map $\fD(\gzero) \twoheadrightarrow \g^V$.

Let $H = \calD_{n-1}$, and let $\fh \subseteq \Lie(\CH)$ be the Lie subsuperalgebra generated by the set $S_H$ of reflections in $H$. Then $H$ and the subset $S_H$ satisfy \ref{item:inductive-subgroup}, by Lemma \ref{lemma:D-inductive-supergroup-ngeq6}. The Lie algebra $\fD(\hzero)$ is semisimple by \eqref{eq:D(hzero)}, and $\fD(\hzero) \subseteq \fD(\gzero)$ because $\hzero \subseteq \gzero$. Let $\fh^V = \rho_V(\fD(\hzero))$. Then $\fh^V \subseteq \g^V$ are semisimple subalgebras of $\fsl(V)$, and $\g^V$ acts irreducibly on $V$ by Proposition \ref{prop:n-geq-5-Dn-good}.

\subsubsection{Multiplicity free restriction}

If $\Res^G_H(V)$ is multiplicity free, then one can argue in a manner entirely parallel to the proof of Proposition \ref{prop:A5-for-B} to show that $\g^V = \fsl(V)$, applying Lemma \ref{lemma:Dzero-subminimal-dimensions} and Remark \ref{remark:Dzero-subminimal-dimensions-n=5} in lieu of Lemma \ref{lemma:Bzero-subminimal-dimensions}, and applying Lemma \ref{lemma:Dzero-restriction} in lieu of Lemma \ref{lemma:Bzero-restriction}. In particular, the calculation when $\Res_H^G(V)$ contains a trivial factor proceeds in precisely the same manner as in the proof of Proposition \ref{prop:A5-for-B}.

\subsubsection{Non-multiplicity free restriction: Case (\ref{item:(lam,mu)-Eboxn}) of Lemma \ref{lemma:Dzero-restriction}}

Now suppose $\Res_H^G(V)$ is not multiplicity free. By Remark \ref{remark:D-repeated-factors}, this can happen in one of two ways.

First suppose $V = S^{\dbl{\lambda,\mu}}$ for some $\dbl{\lambda,\mu} \in \Edbln$ with $\lambda \neq \mu$, and suppose that $\mu \prec \lambda$ with $\mu = \mu^*$ and $\abs{\mu}$ odd. (The case when $\lambda \prec \mu$ with $\abs{\lambda}$ odd and $\lambda = \lambda^*$ is entirely similar.) Then
	\[
	\Res^G_H(V) = 
	\Big[ \bigoplus_{\substack{\set{\nu,\tau} \prec \set{\lambda,\mu} \\ \set{\nu,\tau} \neq \set{\mu,\mu}}} S^{\dbl{\nu,\tau}} \Big] 
	\oplus 
	\Big[ S^{\dbl{\mu,+}} \oplus S^{\dbl{\mu,-}} \Big].
	\]
One has $S^{\dbl{\mu,+}} \cong S^{\dbl{\mu,-}}$ as $\CH$-modules, and this is the only isomorphism among the simple $\CH$-factors in $\Res^G_H(V)$. Then the sum of the structure maps for the simple $\CH$-modules induces
	\[
	\fh^V \cong 
	\Big[ \bigoplus_{\substack{\set{\nu,\tau} \prec \set{\lambda,\mu} \\ \set{\nu,\tau} \neq \set{\mu,\mu}}} \fsl(S^{\dbl{\nu,\tau}}) \Big] 
	\oplus 
	\Big[ \fsl(S^{\dbl{\mu,\pm}}) \Big],
	\]
where $\fsl(S^{\dbl{\mu,\pm}})$ denotes the diagonally embedded copy of $\fsl(S^{\dbl{\mu,+}}) \cong \fsl(S^{\dbl{\mu,-}})$ in $\End(S^{\dbl{\mu,+}}) \oplus \End(S^{\dbl{\mu,-}})$. There exists $\tau \prec \mu$, and for such $\tau$ the simple ideal $\fsl(S^{\dbl{\lambda,\tau}})$ of $\fh^V$ has a composition factor $S^{\dbl{\lambda,\tau}}$ of multiplicity one in $V$. Say $\abs{\mu} = m$. Then $\abs{\lambda} = m+1$, $n = 2m+1$ (so $m \geq 3$ because $n \geq 6$), and for $\tau \prec \mu$ one has
	\[ \textstyle
	\dim(S^{\dbl{\lambda,\tau}}) = \binom{2m}{m+1} \cdot \dim(S^\lambda) \cdot \dim(S^\tau) \geq \binom{2m}{m+1} \cdot \dim(S^\lambda).
	\]
Since $\fsl(k)$ has no non-trivial modules of dimension less than $k$, it follows that each nontrivial $\fsl(S^{\dbl{\lambda,\tau}})$-module is of dimension at least $\binom{2m}{m+1} \cdot \dim(S^\lambda)$.

Now let $\g^V = \g_1 \oplus \cdots \oplus \g_t$ be the decomposition of the semisimple Lie algebra $\g^V$ into simple ideals. We want to show that $t = 1$. Since $\g^V$ acts irreducibly on $V$, it follows for each $1 \leq i \leq t$ that there exists a simple $\g_i$-module $V_i$ such that $V$ identifies with the external tensor product of modules $V_1 \boxtimes \cdots \boxtimes V_t$. Since $\g^V \subset \End(V)$ acts faithfully on $V$, it follows for each $i$ that $V_i$ cannot be the trivial $\g_i$-module (else $\g_i$ acts trivially on $V$), and hence $\dim(V_i) \geq 2$.

Fix $\tau \prec \mu$, and set $\fh^\tau = \fsl(S^{\dbl{\lambda,\tau}})$. Then each $V_i$ becomes an $\fh^\tau$-module via the sequence of canonical Lie algebra maps $\fh^\tau \hookrightarrow \fh^V \subseteq \g^V \twoheadrightarrow \g_i$. If $V_i$ is trivial as an $\fh^\tau$-module, it would follow that each $\fh^\tau$-composition factor in $V$ would be of multiplicity at least $\dim(V_i) \geq 2$, contradicting the observation that $\fh^\tau$ has a composition factor in $V$ of multiplicity one. Thus each $V_i$ must be a nontrivial $\fh^\tau$-module. Now if $t \geq 2$, one has
	\begin{align*}
	\dim(V) \geq \dim(V_1) \cdot \dim(V_2) &\geq \textstyle \binom{2m}{m+1} \cdot \dim(S^\lambda) \cdot \binom{2m}{m+1} \cdot \dim(S^\lambda) \\
	&\geq \textstyle 2m  \cdot \binom{2m}{m+1} \cdot \dim(S^\lambda) \cdot \dim(S^\lambda) \\
	&> \textstyle \frac{2m+1}{m}  \cdot \binom{2m}{m+1} \cdot \dim(S^\lambda) \cdot \dim(S^\mu) \\
	&= \textstyle \binom{2m+1}{m+1} \cdot \dim(S^\lambda) \cdot \dim(S^\mu) \\
	&= \dim(V),
	\end{align*}
where at the third line we observe that $\dim(S^\mu) \leq \dim(S^\lambda)$ if $\mu \prec \lambda$. Then $\dim(V) > \dim(V)$, a contradiction, so it must be the case that $t=1$ and $\g^V$ is a simple Lie algebra.

Now observe that
	\[
	2 \cdot \rk(\fh^V) - \dim(V) 
	= 
	\Big[ \sum_{\substack{\set{\nu,\tau} \prec \set{\lambda,\mu} \\ \set{\nu,\tau} \neq \set{\mu,\mu}}} \dim(S^{\dbl{\nu,\tau}})-2 \Big] 
	+ 
	\Big[ -2 \Big].
	\]
The assumptions on $\lambda$ and $\mu$ imply that the trivial $\CH$-module does not occur in $\Res_H^G(V)$. The assumptions also imply that $n \geq 7$, and hence imply by Lemma \ref{lemma:Dzero-subminimal-dimensions} that $\dim(S^{\dbl{\nu,\tau}}) \geq n-1 \geq 6$ for each $\set{\nu,\tau} \prec \set{\lambda,\mu}$ with $\set{\nu,\tau} \neq \set{\mu,\mu}$. There exist such $\set{\nu,\tau}$, namely, for $\tau \prec \mu$ one has $\set{\lambda,\tau} \prec \set{\lambda,\mu}$, so $2 \cdot \rk(\fh^V) - \dim(V) > 0$. Then $\g^V \cong \fsl(V)$ by \cite[Lemme 13]{Marin:2007}.

\subsubsection{Non-multiplicity free restriction: Case (\ref{item:(lam,lam)-Fdbln-n/2-odd}) of Lemma \ref{lemma:Dzero-restriction}} \label{subsubsec:4b-D}

Now suppose $n$ is even, $n/2$ is odd, $\dbl{\lambda,\lambda} \in \Fdbln$, and $V = S^{\dbl{\lambda,+}} \cong S^{\dbl{\lambda,-}}$. First we will show that $\g^V$ is a simple Lie algebra; the argument is an adaptation of an argument in the proof of \cite[Lemma 5.7.1]{DK:2025-preprint}. Observe that since $\lambda = \lambda^*$, then $\nu \prec \lambda$ if and only if $\nu^* \prec \lambda$. Then
	\[
	\Res^G_H(S^{\dbl{\lambda,+}}) \cong 
	\Big[ \bigoplus_{\substack{\nu \prec \lambda \\ \res(\lambda/\nu)> 0}} S^{\dbl{\nu,\lambda}} \oplus S^{\dbl{\nu^*,\lambda}} \Big] 
	\oplus 
	\Big[ \bigoplus_{\substack{\nu \prec \lambda \\ \res(\lambda/\nu) = 0}} S^{\dbl{\nu,\lambda,+}} \oplus S^{\dbl{\nu,\lambda,-}} \Big].
	\]
If $\res(\lambda/\nu) > 0$, then $S^{\dbl{\nu,\lambda}} \cong S^{\dbl{\nu^*,\lambda}}$ as $\CH$-modules, and these are the only isomorphisms among the simple $\CH$-module factors in $\Res^G_H(V)$. For $\nu \prec \lambda$, one has
	\[ \textstyle
	\dim(S^{\dbl{\nu,\lambda}}) = \binom{n-1}{n/2} \cdot \dim(S^\nu) \cdot \dim(S^\lambda) \geq \binom{n-1}{n/2} \cdot \dim(S^\lambda).
	\]
If $\nu \prec \lambda$ and $\nu = \nu^*$, then $\nu$ is not either $[n/2-1]$ or $[1^{n/2-1}]$, so $\dim(S^\nu) \geq 2$ and
	\[ \textstyle
	\dim(S^{\dbl{\nu,\lambda,\pm}}) = \frac{1}{2} \cdot \dim(S^{\dbl{\nu,\lambda}}) = \frac{1}{2} \cdot \binom{n-1}{n/2} \cdot \dim(S^\nu) \cdot \dim(S^\lambda) \geq \binom{n-1}{n/2} \cdot \dim(S^\lambda).
	\]
Then $\dim(W) \geq \binom{n-1}{n/2} \cdot \dim(S^\lambda)$ for each simple $\CH$-module factor $W$ in $V$.

The sum of the structure maps for the simple $\CH$-modules induces an isomorphism
	\begin{equation} \label{eq:hV-case-4b}
	\fh^V \cong 
	\Big[ \bigoplus_{\substack{\nu \prec \lambda \\ \res(\lambda/\nu) > 0 }} \fsl\dbl{\nu,\lambda} \Big] 
	\oplus 
	\Big[ \bigoplus_{\substack{\nu \prec \lambda \\ \res(\lambda/\nu) = 0 }}\fsl(S^{\dbl{\nu,\lambda,+}}) \oplus \fsl(S^{\dbl{\nu,\lambda,-}}) \Big],
	\end{equation}
where $\fsl\dbl{\nu,\lambda}$ denotes the diagonal copy of $\fsl(S^{\dbl{\nu,\lambda}}) \cong \fsl(S^{\dbl{\nu^*,\lambda}})$ in $\End(S^{\dbl{\nu,\lambda}}) \oplus \End(S^{\dbl{\nu^*,\lambda}})$. Thus, if there exists $\nu \prec \lambda$ with $\res(\lambda/\nu) = 0$, then the simple ideals $\fsl(S^{\dbl{\nu,\lambda,+}})$ and $\fsl(S^{\dbl{\nu,\lambda,-}})$ of $\fh^V$ have simple factors $S^{\dbl{\nu,\lambda,+}}$ and $S^{\dbl{\nu,\lambda,-}}$, respectively, of multiplicity one in $V$, and each remaining simple ideal $\fsl\dbl{\nu,\lambda}$ has a simple factor $S^{\dbl{\nu,\lambda}} \cong S^{\dbl{\nu^*,\lambda}}$ of multiplicity two in $V$.

Now let $\g^V = \g_1 \oplus \cdots \oplus \g_t$ be the decomposition of the semisimple Lie algebra $\g^V$ into simple ideals. We want to show that $t = 1$. Since $\g^V$ acts irreducibly on $V$, it follows for each $1 \leq i \leq t$ that there exists a simple $\g_i$-module $V_i$ such that $V$ identifies with the external tensor product of modules $V_1 \boxtimes \cdots \boxtimes V_t$. Since $\g^V \subset \End(V)$ acts faithfully on $V$, it follows for each $i$ that $V_i$ cannot be the trivial $\g_i$-module (else $\g_i$ acts trivially on $V$), and hence $\dim(V_i) \geq 2$.

Each $V_i$ is an $\fh^V$-module via the sequence of canonical maps $\fh^V \hookrightarrow \g^V \twoheadrightarrow \g_i$. Since the nontrivial $\fsl(k)$-modules are all of dimension at least $k$, it follows for each $i$ that if $V_i$ is a nontrivial $\fh^V$-module (i.e., if $\g_i$ is nontrivial for some simple ideal of $\fh^V$), then $\dim(V_i) \geq \binom{n-1}{n/2} \cdot \dim(S^\lambda)$. If there are two values of $i$ such that $V_i$ is a nontrivial $\fh^V$-module, say $i=1$ and $i=2$, then
	\begin{align*}
	\dim(V) &\geq \dim(V_1) \cdot \dim(V_2) \\
	&\geq \textstyle \binom{n-1}{n/2} \cdot \binom{n-1}{n/2} \cdot \dim(S^\lambda) \cdot \dim(S^\lambda) \\
	&= \textstyle \binom{n-1}{n/2} \cdot \frac{1}{2} \cdot \binom{n}{n/2} \cdot \dim(S^\lambda) \cdot \dim(S^\lambda) \\
	&= \textstyle \binom{n-1}{n/2} \cdot \dim(V) \\
	&> \dim(V),
	\end{align*}
a contradiction. Thus there is only one value of $i$, say $i=1$, such that $V_i$ is nontrivial for $\fh^V$. And since $\dim(V_i) \geq 2$ for each $i$, and each simple ideal of $\fh^V$ has a composition factor of multiplicity at most $2$ in $V$, it also follows that there is at most one value of $i$ such that $V_i$ is trivial for $\fh^V$, and if such a $V_i$ exists then $\dim(V_i) = 2$. Thus either $t = 1$ (as desired), or $t = 2$ and $V \cong V_1 \boxtimes V_2$ with $V_1$ a nontrivial $\fh_V$-module and $V_2$ a $2$-dimensional trivial $\fh^V$-module. In the latter case one has $\g^V = \g_1 \oplus \g_2$, with $\g_2 \cong \fsl(2)$ because $\dim(V_2) = 2$.

Observe that if $x \in \gzero$, and if $s \in \gzero \cap \Gzero$ with $s^2 = 1$, then
	\begin{equation} \label{eq:conjugation-bracket-difference} \textstyle
	x - \frac{1}{2} \cdot [s,[s,x]] = sxs.
	\end{equation}
This implies that if $I$ is an ideal in $\fD(\gzero)$ (hence an ideal in $\gzero$), then $s I s = I$. Since $\gzero$ contains the set \eqref{eq:(ij)(kl)-CC} of order-two generators for $\Gzero$, it follows that each ideal in $\fD(\gzero)$ is invariant under conjugation by arbitrary elements of $\Gzero$. And since $V$ is a $\CGzero$-module, it follows that the conjugation action on $\fD(\gzero)$ passes to an action on the simple ideals of $\g^V$; recall the comment from the beginning of Section \ref{subsec:A5-for-D} that we may identify $\g^V$ with a subalgebra of $\fD(\gzero)$.

Now suppose $t = 2$ as above. Then the conjugation action of $\Gzero$ on $\g_2 \cong \fsl(2)$ defines a group homomorphism $\phi: \Gzero \to \Aut(\fsl(2)) \cong PGL(2,\C)$ with finite image. Up to isomorphism, the nontrivial finite subgroups of $PGL(2,\C)$ are:
	\begin{itemize}
	\item the cyclic groups $C_m$ of order $m$,
	\item the dihedral groups $D_m$ of order $2m$,
	\item the alternating group $A_4$,
	\item the symmetric group $S_4$, and
	\item the alternating group $A_5$;
	\end{itemize}
see \cite{Beauville:2010}. Further, since $\Gzero$ is perfect by Remark \ref{rem:commutator-subgroups}, the image of $\phi$ must be perfect. This implies that either $\phi$ is the trivial map, or $\im(\phi) \cong A_5$. Suppose the latter. Then any simple $\CA_5$-module lifts, via $\phi$, to a simple $\CGzero$-module. The alternating group $A_5$ has simple modules of dimension $3$ by \cite[Theorem 2.5.15]{James:1981}, but applying Lemma \ref{lemma:Dzero-subminimal-dimensions} and Lemma \ref{lemma:Bn-subminimal-dimensions}\eqref{item:dim-S(lam,mu)}, one sees that if $n \geq 6$ is twice an odd integer, then $\Gzero = (\calD_n)_{\zero}$ has no simple modules of dimension $3$. This is a contradiction, so $\phi$ must be the trivial map. Then each element $x \in \g_2$ is fixed under conjugation by all elements of $\Gzero$. This implies by Lemma \ref{lemma:lie-algebra-center} that $\g_2 \subseteq Z(\g^V)$, which is again a contradiction because $\g_2$ is a nonzero ideal in $\g^V$, but $\g^V$ is a semisimple Lie algebra and hence $Z(\g^V) = 0$. Thus, it must be the case that $t = 1$ and $\g^V$ is simple.

Now we will show that $\g^V = \fsl(V)$. From \eqref{eq:hV-case-4b} one gets
	\[
	\rk(\fh^V) = 
	\sum_{\substack{\nu \prec \lambda \\ \res(\lambda/\nu) > 0 }} \left[ \dim(S^{\dbl{\nu,\lambda}})-1 \right] 
	+ 
	\sum_{\substack{\nu \prec \lambda \\ \res(\lambda/\nu) = 0 }} \left[ \dim(S^{\dbl{\nu,\lambda}}) -2 \right],
	\]
Let $r_\lambda$ be the number of removable boxes in the Young diagram of $\lambda$. Then
	\begin{align*}
	2 \cdot \rk(\fh^V) &= 
	\sum_{\substack{\nu \prec \lambda \\ \nu \neq \nu^* }} \left[ \dim(S^{\dbl{\nu,\lambda}})-1 \right] 
	+ 
	\sum_{\substack{\nu \prec \lambda \\ \nu = \nu^* }} \left[ 2 \cdot \dim(S^{\dbl{\nu,\lambda}}) - 4 \right] \\
	&= \sum_{\nu \prec \lambda} \left[ \dim(S^{\dbl{\nu,\lambda}})-1 \right] + \sum_{\substack{\nu \prec \lambda \\ \nu = \nu^* }} \left[ \dim(S^{\dbl{\nu,\lambda}}) - 3 \right] \\
	&= \Big[ \dim(V) - r_\lambda \Big] + \sum_{\substack{\nu \prec \lambda \\ \nu = \nu^* }} \left[ \dim(S^{\dbl{\nu,\lambda}}) - 3 \right].
	\end{align*}
If $\nu \prec \lambda$ with $\nu = \nu^*$, then $\dim(S^{\dbl{\nu,\lambda}}) \geq 3$ by Lemma \ref{lemma:Dzero-subminimal-dimensions}, and evidently $r_\lambda < \abs{\lambda} = \frac{n}{2}$. One has
	\[ \textstyle
	\dim(V) = \dim(S^{\dbl{\lambda,+}}) = \frac{1}{2} \cdot \binom{n}{n/2} \cdot \dim(S^\lambda) \cdot \dim(S^\lambda) \geq \frac{1}{2} \cdot n \cdot 2 \cdot 2 = 2n,
	\]
and hence $2 \cdot \rk(\fh^V) \geq 2n - r_\lambda > 2n - \frac{n}{2} = \frac{3n}{2}$. This implies that
	\[ \textstyle
	4 \cdot \rk(\g^V) \geq 4 \cdot \rk(\fh^V) > \frac{3n}{2} + \left( \dim(V) - \frac{n}{2} \right) = n + \dim(V) > \dim(V).
	\]
Now either $\dim(V) < 2 \cdot \rk(\g^V)$, or $2n \leq \dim(V) < 4 \cdot \rk(\g^V)$. Suppose the latter. Since
	\[ \textstyle
	\dim(V) = \frac{1}{2} \cdot \binom{n}{n/2} \cdot \dim(S^\lambda) \cdot \dim(S^\lambda) \geq \frac{1}{2} \cdot \binom{n}{n/2} \cdot 2 \cdot 2 = 2 \cdot \binom{n}{n/2} \geq 2 \cdot \binom{6}{3} = 40,
	\]
this would imply by \cite[Lemme 14]{Marin:2007} that either $\g^V \cong \fso(V)$ or $\g^V \cong \fsp(V)$, i.e., that $\g^V$ is a simple Lie algebra of type C or D and $V$ is isomorphic to the natural module for $\g^V$. But the natural modules for $\fso(V)$ and $\fsp(V)$ are self-dual as Lie algebra representations, whereas we know that $V \not\cong V^{*,\Lie}$ by Proposition \ref{prop:n-geq-5-Dn-good}, Lemma \ref{lemma:D-inductive-supergroup-ngeq6}, and Lemma \ref{lemma:Gzero-modules-not-dual}. Thus it must be the case that $\dim(V) < 2 \cdot \rk(\g^V)$, and hence $\g^V = \fsl(V)$ by \cite[Lemme 13]{Marin:2007}.

\appendix

\section{Dimensions for small ranks} \label{app:dim-small-ranks}

Table \ref{tab:upper-bounds} summarizes the upper bounds on the dimensions of the Lie superalgebras $\fb_n$, $\fs_n$, and $\fd_n$ provided by the inclusions \eqref{eq:bn-easy-inclusion}, \eqref{eq:sn-easy-inclusion}, and \eqref{eq:dn-easy-inclusion} and the isomorphisms \eqref{eq:D(CBn)-Artin-Wedderburn}, \eqref{eq:D(CSn)-Artin-Wedderburn}, and \eqref{eq:D(CDn)-Artin-Wedderburn}, respectively.

\begin{table}[htb]
\begin{tabular}{|l|l|l|l|}
\hline
$n$ & $\dim(\fb_n) \leq$ & $\dim(\fs_n) \leq$ & $\dim(\fd_n) \leq $ \\
\hline
2 & 7 & 2 & \\
\hline
3 & 45 & 5 & \\
\hline
4 & 375 & 22 & 185 \\
\hline
5 & 3824 & 117 & 1911 \\
\hline
\end{tabular}

\ 

\caption{Upper bounds on dimensions of $\fs_n$, $\fb_n$, and $\fd_n$.} \label{tab:upper-bounds}
\end{table}

For example, for the group $\fS_5$, one finds that a complete set of representatives for $E_5/\!\!\sim$ is $\{2^21, 21^3, 1^5 \}$, while $F_5 = \{ 31^2 \}$; for the sake of improved readability, we have written partitions here without additional enclosing brackets. Then
	\[
	\fD(\CS_5) = \sq(W^{[2^21]}) \oplus \sq(W^{[21^3]}) \oplus \sq(W^{[1^5]}) \oplus \fsl(W^{[31^2]}),
	\]
so $\dim(\fD(\CS_5)) = \dim(\CS_5) -4 = 116$, and hence $\dim(\fs_5) \leq 116 + 1 = 117$.

For the group $\calB_4$, one finds that
	\begin{gather*}
	E[4] = \set{ [4,\emptyset], [31,\emptyset], [2^2,\emptyset], [21^2,\emptyset], [1^4,\emptyset], [3,1], [21,1], [1^3,1], [2,2]}, \\
	F[4] = \set{ [2,1^2], [1^2,2]},
	\end{gather*}
so $\dim(\fD(\CB_4)) = \dim(\CB_4) - 11 = 384-11 = 373$, and hence $\dim(\fb_4) \leq 373+2 = 375$.

For the group $\calD_4$, one finds that
	\[
	E\dbl{4} = \set{\dbl{ 4, \emptyset}, \dbl{31,\emptyset}, \dbl{3,1}, \dbl{2,2}} 
	\qquad \text{and} \qquad 
	F\dbl{4} = \set{ \dbl{2^2,\emptyset}, \dbl{21,1}, \dbl{2,1^2}}.
	\]
Then, noting that the term $\dbl{2,2} \in E\dbl{4}$ corresponds to two $\sq$ terms, one has
	\begin{gather*}
	\fD(\CD_4) = \sq(W^{\dbl{ 4, \emptyset}}) \oplus \sq(W^{\dbl{31,\emptyset}}) \oplus \sq(W^{\dbl{3,1}}) \oplus \sq(W^{\dbl{2,+}}) \oplus \sq(W^{\dbl{2,-}}) \\
	{} \oplus \fsl(W^{\dbl{2^2,\emptyset}}) \oplus \fsl(W^{\dbl{21,1}}) \oplus \fsl(W^{\dbl{2,1^2}}),
	\end{gather*}
so $\dim(\fD(\CD_4)) = \dim(\CD_4) - 8 = 192 - 8 = 184$, and hence $\dim(\fd_4) \leq 184 +1 = 185$.

The reasoning for the other entries in Table \ref{tab:upper-bounds} is similar.

Using GAP \cite{GAP4}, one can check that the inequalities in Table \ref{tab:upper-bounds} are equalities, thereby verifying Lemmas \ref{theorem:main-theorem-B}, \ref{lemma:base-cases-Sn}, and \ref{lemma:base-cases-D}. The code we used to carry out these calculations is included as an ancillary file with the version of this paper posted on the arXiv (\href{https://arxiv.org/abs/2509.00945}{arXiv:2509.00945}). We ran the code using the Gap.app frontend for MacOS, running version 4.12.1 of GAP. The algorithm works as follows (with notation as in Section \ref{subsec:lie-superalgebras-odd-elements}): Let $V_1$ be the subspace of $\Lie(\CG)$ spanned by the elements of the set $S$, and then for $k \geq 2$, let $V_k$ be the subspace of $\Lie(\CG)$ spanned by $V_{k-1}$ together with all elements of the form $[s,v]$ for $s \in S$ and $v \in V_{k-1}$. Then the $V_k$ are an increasing chain of subspaces of the Lie superalgebra $\g$ with $\bigcup_{k \geq 1} V_k = \g$. The algorithm successively computes (and outputs) the dimensions of the subspaces $V_k$ for $k \geq 2$ until $\dim(V_k) = \dim(V_{k-1})$, at which point the algorithm terminates.

\makeatletter
\renewcommand*{\@biblabel}[1]{\hfill#1.}
\makeatother

\bibliographystyle{eprintamsplain}
\bibliography{lie-superalgebras-reflections}

\end{document}